\documentclass[12pt,a4paper]{report}
\usepackage{amsmath, amsthm, amssymb,latexsym,amscd,amsfonts,enumerate}
\usepackage{indentfirst}
\usepackage[mathscr]{eucal}
\usepackage{color, fancyhdr, graphicx, wrapfig}
\usepackage{natbib}
\usepackage{imakeidx}
\usepackage{caption}
\usepackage[intoc]{nomencl}

\usepackage{fancybox}
\usepackage{mathtools}
\usepackage{appendix}
\usepackage{url}
\usepackage[unicode]{hyperref}
\usepackage[nameinlink]{cleveref}
\usepackage{float}
\fontsize{12pt}{18pt}\selectfont 
\newtheorem{theorem}{Theorem}[chapter]
\theoremstyle{definition}
\newtheorem{definition}[theorem]{Definition}
\newtheorem{proposition}[theorem]{Proposition}
\newtheorem{corollary}{Corollary}[theorem]
\newtheorem{lemma}[theorem]{Lemma}
\newtheorem{conjecture}[theorem]{Conjecture}
\theoremstyle{remark}
\newtheorem*{remark}{Remark}

\theoremstyle{definition}
\newtheorem{example}[theorem]{Example}
\makeindex[intoc]
\makenomenclature
\begin{document}
\crefformat{subsection}{#2Subsection~#1#3}
\crefname{theorem}{Theorem}{Theorems}
\crefname{definition}{Defintion}{Definitions}
\crefname{proposition}{Proposition}{Propositions}
\crefname{corollary}{Corollary}{Corollaries}
\crefname{lemma}{Lemma}{Lemmas}
\crefname{conjecture}{Conjecture}{Conjectures}
\crefname{subsection}{Subsection}{Subsections}
\begin{titlepage}
	\thisfancypage{
		\setlength{\fboxsep}{1pt}
		\setlength{\fboxrule}{1.2pt}
		\doublebox}{}
	\vspace*{-0.5cm}
	\begin{center}
		{\fontsize{13pt}{1}\selectfont  VIETNAM NATIONAL UNIVERSITY, HANOI} \\
			{\fontsize{13pt}{1}\selectfont  VNU UNIVERSITY OF SCIENCE} \\
			{\fontsize{13pt}{1}\selectfont  \bf FACULTY OF MATHEMATICS - MECHANICS - INFORMATICS}
	\end{center}
	\vspace*{2.5cm}
	\begin{center}
		{\bf\fontsize{14pt}{1}\selectfont Le Xuan Hoang}
	\end{center}
	\vspace*{2.5cm}
	\begin{center}
		{\bf\fontsize{18pt}{1}\selectfont On Modular Invariants of Truncated Polynomial Rings}
	\end{center}
	\vspace*{0.5cm}
	\begin{center}
		{\fontsize{14pt}{1}\selectfont Mathematics Major}\\
		{\fontsize{14pt}{1}\selectfont Talented Program}\\
	\end{center}
	\vspace*{0.5cm}
	
	\vfill
	\begin{center}
		{\bf\fontsize{14pt}{1}\selectfont Hanoi - 2025}
	\end{center}
\end{titlepage}
	\begin{titlepage}
		\thisfancypage{
			\setlength{\fboxsep}{1pt}
			\setlength{\fboxrule}{1.2pt}
			\doublebox}{}
		\vspace*{-0.5cm}
		\begin{center}
			{\fontsize{13pt}{1}\selectfont  VIETNAM NATIONAL UNIVERSITY, HANOI} \\
			{\fontsize{13pt}{1}\selectfont  VNU UNIVERSITY OF SCIENCE} \\
			{\fontsize{13pt}{1}\selectfont  \bf FACULTY OF MATHEMATICS - MECHANICS - INFORMATICS}
			
		\end{center}
		\vspace*{2.5cm}
		\begin{center}
			{\bf\fontsize{14pt}{1}\selectfont Le Xuan Hoang}
		\end{center}
		\vspace*{2.5cm}
		\begin{center}
			{\bf\fontsize{18pt}{1}\selectfont On Modular Invariants of Truncated Polynomial Rings}
		\end{center}
		\vspace*{0.5cm}
		\begin{center}
			{\fontsize{14pt}{1}\selectfont Mathematics Major}\\
		{\fontsize{14pt}{1}\selectfont Talented Program}\\
		\end{center}
		\vspace*{0.5cm}
		\begin{center}
			{\bf\fontsize{14pt}{1}\selectfont Supervisor: Assoc. Prof. Le Minh Ha}
		\end{center}
		\vfill
		\begin{center}
			{\bf\fontsize{14pt}{1}\selectfont Hanoi - 2025}
		\end{center}
	\end{titlepage}
\vspace*{2cm}
\tableofcontents
\newpage
\chapter*{
Introduction
}
\addcontentsline{toc}{chapter}{{\bf{Introduction}}\rm }
Modular Invariant Theory is a branch of mathematics that explores the behavior of polynomial functions invariant under group actions, particularly over fields with positive characteristic. Overall, modular invariant theory serves as a vital link connecting algebraic methods with combinatorial and topological applications, enriching each field through its interdisciplinary reach. In this thesis, we investigate various aspects of modular invariant theory, with a particular focus on the theory of \emph{Schur functions over finite fields} and the \emph{invariant theory of truncated polynomial algebras}.\smallbreak
Firstly, it is well-known that Schur functions are an important family of symmetric functions that play a pivotal role in algebraic combinatorics. In 1992, Macdonald \cite{macdonald_schur_1992} introduced nine variations of Schur functions, thereby enabling new connections between mathematical objects that previously appeared unrelated. Motivated by the compelling analogy between the symmetric group $\Sigma_n$ and the general linear group $GL_n(\mathbb{F}_q)$ (when $q\to 1$), with the hope of unifying various structures and concepts across combinatorics and invariant theory, the first object of interest in this thesis is Macdonald's 7th variation of Schur functions, which are called \emph{Schur functions over finite field}. Notably, this variation comprises a family of polynomials invariant under the action of the general linear group; this construction reinforces the deep connection between $\Sigma_n$ and $GL_n(\mathbb{F}_q)$, further illuminating the interplay between combinatorial and algebraic structures. Working on Macdonald's 7th variation of Schur functions, we prove a generalization of \cite[Conjecture (7.25)]{macdonald_schur_1992}, leading to an extension of the Stong-Tamagawa formula, providing a basis-free expression of Schur functions.\smallbreak
Secondly, a \emph{truncated polynomial algebra} $Q(m, n)$ is defined as the quotient ring $\mathbb{F}_q[x_1, \ldots, x_n]/(x_1^{q^m}, \ldots, x_n^{q^m})$. This ring admits an action of $GL_n(\mathbb{F}_q)$ induced from the polynomial ring $\mathbb{F}_q[x_1, \ldots, x_n]$. The $GL_n(\mathbb{F}_q)$-module structure of $Q(m, n)$ plays a critical part in analyzing certain important problems in algebraic topology (see \cite[Section 1]{Ha_Hai_Nghia_2024} and the references therein). Motivated by finite field analogues of the representation theory of finite reflection groups--particularly the theory of parking spaces--in 2017, Lewis, Reiner, and Stanton \cite{Lewis_Reiner_Stanton_2017} proposed a set of conjectures concerning the Hilbert series of the invariant rings of $Q(m, n)$ under the action of parabolic subgroups of $GL_n(\mathbb{F}_q)$, offering a novel interpretation of this previously enigmatic structure. Recently, in 2024, L. M. Ha, N. D. H. Hai, and N. V. Nghia \cite{Ha_Hai_Nghia_2024} proved the Lewis-Reiner-Stanton conjecture for the case of the Borel subgroup by explicitly constructing a basis for the corresponding invariant ring. Our research aims to extend the methodology of this proof to other subgroups of $GL_n(\mathbb{F}_q)$, as well as answering additional structural questions that arise from this line of inquiry.\smallbreak
An overview of the thesis structure and the key results presented in each chapter is provided below.
\begin{itemize}
    \item \Cref{chapter:invariant_theory} provides a concise introduction to modular invariant theory, focusing on the Dickson algebra and related structures. It also examines the concept of Schur functions over finite fields, as presented in \cite{macdonald_schur_1992}. A generalization of \cite[Conjecture (7.25)]{macdonald_schur_1992} is established in \cref{gen_of_7.25}, leading to an extension of the Stong–Tamagawa formula in \cref{main_rewrite}, representing a Schur function as a sum over complete flags.
    \item \Cref{chap:trucated_invariants} focuses on analyzing the Lewis-Reiner-Stanton conjecture (\cref{conj:parabolic_conj}) and the proof provided by Ha, Hai and Nghia \cite{Ha_Hai_Nghia_2024} for the Borel subgroup case; using a similar argument, \cref{thm: q=p_case} indicates a construction of bases for invariant rings of truncated algebras under the action of the unipotent subgroup, in the case the field is of prime order. Additionally, we examine the delta operators--a pivotal family of operators used in \cite{Ha_Hai_Nghia_2024}--with particular attention to their \emph{polynomiality} property; we propose a conjecture (\cref{conj:polynomiality}) regarding the module structure of the set of polynomials that remain within the polynomial ring upon the application of delta operators. Our main results in this chapter are \cref{thm:pol_bivariate}, explicitly describing the module structure in the bivariate case; while \cref{cor:weak_version_m_less_than_n} and \cref{prop:weakened_version} provide evidences and partial answers for the conjecture.
    \item \Cref{chap:conclusion} concludes the thesis by discussing potential directions for future research in this area.
    \item \Cref{appendix:brauer_character} introduces the theory of modular characters to elucidate statements in \cref{cor:brauer_iso}. 
\end{itemize}
While every effort has been made to ensure the accuracy and clarity of this thesis, I acknowledge that oversights may remain. I welcome any constructive feedback and sincerely apologize for any errors that might have been overlooked.
\chapter*{Acknowledgements}
I am profoundly grateful to my supervisor, Associate Professor, Doctor Le Minh Ha, for his invaluable guidance, insightful feedback, and unwavering support throughout this research. 

I would like to express my sincere appreciation to the members at Faculty of Mathematics – Mechanics – Informatics, VNU University of Sience, whose instruction and mentorship have been instrumental in my academic development. In particular, I am grateful to Associate Professor Ngo Quoc Anh, Dr. Pham Van Tuan, Associate Professor Trinh Viet Duoc, and Associate Professor Le Quy Thuong, their guidance has not only deepened my understanding of mathematical concepts but has also provided invaluable insights and inspiration for my future endeavors.

I am also thankful to my friends and classmates for their companionship and the shared experiences that have enriched my memorable journey at VNU University of Science. Their support and friendship have been a source of motivation and comfort throughout my studies.

Finally, I extend my deepest gratitude to my family for their unwavering support and encouragement. Their belief in me has been a constant source of strength, enabling me to persevere through challenges and achieve my goals.
\begin{flushright}
    \textit{Hanoi, May 2025}\\
    Le Xuan Hoang
\end{flushright}
\addcontentsline{toc}{chapter}{{\bf{Acknowledgements}}\rm }
\nomenclature[a01]{$\mathbb{N}$}{the natural numbers, not including $0$.}
\nomenclature[a02]{$\mathbb{Z}$}{the integers.}
\nomenclature[a03]{$\mathbb{Z}_{\geq 0}$}{the nonnegative integers.}
\nomenclature[a04]{$\mathbb{Q}$}{the rational numbers.}
\nomenclature[a05]{$\mathbb{C}$}{the complex numbers.}
\nomenclature[a06]{$\mathbb{K}$}{a field.}
\nomenclature[a07]{$\mathbb{F}_q$}{the finite field having $q$ elements, where $q$ is a power of a prime $p$.}
\nomenclature[a08]{$\mathbb{K}^{\times}$}{$\mathbb{K}\backslash \{0\}$.}
\nomenclature[a09]{$V$}{a finite-dimensional vector space over $\mathbb{K}$.}
\nomenclature[a10]{$GL(V)$}{the group of linear automorphisms of $V$.}
\nomenclature[a11]{$GL_n(\mathbb{K})$}{the group of linear automorphisms of $\mathbb{K}^n$.}
\nomenclature[a12]{$V^{*}$}{the dual $Hom_{\mathbb{K}}(V, \mathbb{K})$.}
\nomenclature[a13]{$\mathbb{K}[V]$}{the symmetric algebra on $V$.}
\nomenclature[a14]{$G$}{a group, usually a finite subgroup of $GL(V)$.}
\nomenclature[a15]{$|G|$}{the order of $G$.}
\nomenclature[a16]{$\mathbb{K}[V]^G$}{the invariant ring.}
\nomenclature[a17]{$\mathbb{K}[V]_G$}{the cofixed space.}
\nomenclature[a18]{$Hilb(M, t)$}{the Hilbert-Poincare series of a graded vector space.}
\nomenclature[a19]{$x^{\alpha}$}{the monomial $x_1^{\alpha_1}\cdots x_n^{\alpha_n}$, where $\alpha = (\alpha_1, \ldots, \alpha_n)$.}
\nomenclature[a20]{$\mathcal{D}_n$}{The Dickson algebra.}
\nomenclature[a21]{$v_p(a)$}{the $p$-adic valuation of $a\in \mathbb{Z}$.}
\printnomenclature
\chapter{Fundamentals of Modular Invariant Theory}\label{chapter:invariant_theory}
\section{Generalities on Modular Invariant Theory}\label{sec: invariant_theory}
Let $G$ be a finite subgroup of $GL(V)$, acting on the polynomial ring $S = \mathbb{K}[V]$. Fixing a basis $\{x_1, \ldots, x_n\}$ of $V$, $G$ corresponds to a group of invertible matrices acting on $S = \mathbb{K}[x_1, \ldots, x_n]$ by linear substitutions. To clarify, if $\sigma = (\sigma_{ij})_{i, j\in \{1, \ldots, n\}}\in G$, then
\begin{equation*}
    \sigma(f)(x_1, \ldots, x_n) = f(\sigma (x_1), \ldots, \sigma (x_n)),
\end{equation*}
where $\sigma(x_j) = \sum\limits_{i=1}^{n}\sigma_{ij}x_i$.\smallbreak
The main object of interest of invariant theory is the ring of invariants\index{ring of invariants} 
\begin{equation*}
    S^G = \{f\in S\mid \sigma(f) = f, \text{ for all } \sigma\in G\}.
\end{equation*}
Given a group $G$, invariant theorists want to understand the structure of $S^G$ as detailed as possible. For example, 
\begin{enumerate}
    \item What conditions on $G$ suffice to ensure that $S^G$ is a polynomial algebra?
    \item How to calculate the Hilbert series of $S^G$? 
\end{enumerate}
The \emph{non-modular} case is when $|G|$ is invertible in the field $\mathbb{K}$. In this case, the answers to questions $(1)$ and $(2)$ are known, specifically, \emph{Chevalley-Shephard-Todd theorem} \cite[$\S$ 18-1]{kane2001reflection} states that $S^G$ is a polynomial algebra if and only if $G$ is generated by \emph{pseudo-reflections}, whilst \emph{Molien's formula} \cite[$\S$ 17-2]{kane2001reflection} provides an explicit formula of $Hilb(S^G, t)$ using elements of $G$.\smallbreak
On the other hand, results in modular invariant theory mainly focus on the \emph{modular case}, that is, when $|G|$ is not invertible in $\mathbb{K}$. Less is known about the structure of $S^G$ in general, yet some interesting phenomena can be observed. The following result provides a criterion for determining whether $S^G$ is a polynomial algebra, and also for finding a basis for $S^G$ as a free $\mathbb{K}$-algebra.
\begin{proposition}[{\cite[Proposition 4.5.5]{neusel_invariant_2002}}]\label{prop:prod_of_deg_equals_G_implies_polynomial}
If there exist homogeneous polynomials $d_1, \ldots, d_n\in S^G$, such that $S$ is integral over the $\mathbb{K}$-algebra generated by $d_1, \ldots, d_n$, and 
\begin{equation*}
 \prod\limits_{i=1}^n \deg(d_i) = |G|,   
\end{equation*}
then $d_1, \ldots, d_n$ are algebraically independent over $\mathbb{K}$, and
\begin{equation*}
    S^G = \mathbb{K}[d_1, \ldots, d_n].
\end{equation*}
\end{proposition}
Before proving \cref{prop:prod_of_deg_equals_G_implies_polynomial}, we need an auxiliary lemma.
\begin{lemma}\label{lemma:integrality_and_fraction_field}
    The action of $G$ on $S$ naturally extends to an action on $Frac(S)$ by variable substitutions. We have an equality of fields
    \begin{equation*}
        Frac(S)^G = Frac(S^G).
    \end{equation*}
\end{lemma}
\begin{proof}
    Obviously, $Frac(S^G)\subseteq Frac(S)^G$. Conversely, suppose that $\dfrac{f}{g}\in Frac(S)^G$; then let $h=\prod\limits_{\sigma\in G\backslash\{1\}}\sigma(g)\in S\backslash\{0\}$, we see that $g\cdot h = \prod\limits_{\sigma\in G}\sigma(g) \in S^G$, and
    \begin{equation*}
        \dfrac{f}{g} = \dfrac{fh}{gh}\in Frac(S)^G.
    \end{equation*}
    For any $\sigma\in G$, 
    \begin{equation*}
        \dfrac{fh}{gh}=\dfrac{f}{g} =\sigma\left(\dfrac{f}{g}\right)= \dfrac{\sigma(fh)}{\sigma(gh)}= \dfrac{\sigma(fh)}{gh}.
    \end{equation*}
    Hence $\sigma(fh) = fh,\text{ for all } \sigma\in G$, which implies that $fh\in S^G$. Therefore, 
    \begin{equation*}
        \dfrac{f}{g}=\dfrac{fh}{gh}\in Frac(S^G).\qedhere
    \end{equation*}
\end{proof}
\begin{proof}[Proof of \cref{prop:prod_of_deg_equals_G_implies_polynomial}]
    Denote by $R$ the $\mathbb{K}$-algebra generated by $d_1, \ldots, d_n$. First of all, since $S$ is integral over $R$, the Krull dimensions of $S$ and $R$ are equal. We know that the Krull dimension of $S$ is $n$, and since $R$ is generated by $d_1, \ldots, d_n$, the Krull dimension of $R$ is $n$ if and only if $d_1, \ldots, d_n$ are algebraically independent.\smallbreak
    For each $i\in \{1, \ldots, n\}$, suppose that $x_i$ is a root of a monic polynomial in $R[x]$ of degree $l_i \geq 1$. Then, $A = \{x^{\alpha}\mid 0\leq \alpha_i< l_i, \text{ for all } i\in \{1, \ldots, n\} \}$ is a generating set of $S$ as an $R$-module. Furthermore, $A$ is also a generating set of the fraction field $Frac(S)$ as a $Frac(R)$-vector space; indeed, consider $\dfrac{f}{g}\in Frac(S)$ for some $f, g\in S, g\neq 0$, then $g$ is integral over $R$. Let $F(X)\in R[X]$ be the minimal polynomial of $g$ over $R$, then assume that
    \begin{equation*}
        F(X) = X^k + c_{k-1}X^{k-1} + \cdots + c_0,
    \end{equation*}
    for some $k>0, c_0, \ldots, c_{k-1}\in R$, and $c_0\neq 0$. Then, $h = g^{k-1} + c_{k-1}g^{k-2} + \cdots +c_1 \neq 0$, and $g\cdot h = -c_0\in R$. Therefore,
    \begin{equation*}
        \dfrac{f}{g} = \dfrac{fh}{gh} = \dfrac{fh}{-c_0}\in Frac(R)\langle A\rangle.
    \end{equation*}
    Let $B = \{u_1, \ldots, u_m\}\subseteq A$ be a basis of $Frac(S)$ as a $Frac(R)$-vector space, then we have the following well-known calculation
    \begin{align*}
        Hilb(S, t) &= \dfrac{1}{(1-t)^{n}},\\
        Hilb(R, t) &= \dfrac{1}{\prod\limits_{i=1}^{n}\left(1-t^{\deg(d_i)}\right)}.
    \end{align*}
    Since $B$ is linearly independent over $Frac(R)$, $RB$ is also a free $R$-module with basis $B$. As $RB\subseteq S$, we have $Hilb(RB, t)\leq Hilb(S, t)$, hence,
    \begin{equation}\label{eq:ineq_with_basis_1}
        \dfrac{\sum\limits_{j=1}^{m}t^{\deg(u_j)}}{\prod\limits_{i=1}^{n}\left(1-t^{\deg(d_i)}\right)} \leq \dfrac{1}{(1-t)^n}, \text{ for all } t\in (0, 1).
    \end{equation}
    For each $e = x^{\alpha}\in A$, write $e$ as a linear combination of $\{u_1, \ldots, u_m\}$ over $Frac(R)$, here, we assume without loss of generality that the denominators of the summands are the same:
    \begin{equation*}
        e = \sum\limits_{j=1}^{m}\dfrac{f_{j, e}}{g_{e}}u_j.
    \end{equation*}
    Notice that $e$ and $u_1, \ldots, u_m$ are homogeneous, we may also assume that $g_e$ is homogeneous for any $e$. Now, let $\Delta =\prod\limits_{e\in A}g_e$, suppose that $\delta = \deg(\Delta)$, then $\Delta\cdot e\in RB, \text{ for all } e\in A$. Therefore, $\Delta\cdot S\in RB$, which implies that $Hilb(\Delta\cdot S, t)\leq Hilb(RB, t)$, or
    \begin{equation}\label{eq:ineq_with_basis_2}
        \dfrac{t^{\delta}}{(1-t)^n}\leq \dfrac{\sum\limits_{j=1}^{m}t^{\deg(u_j)}}{\prod\limits_{i=1}^{n}\left(1-t^{\deg(d_i)}\right)}, \text{ for all } t\in (0, 1).
    \end{equation}
    Combining \eqref{eq:ineq_with_basis_1} and \eqref{eq:ineq_with_basis_2}, we obtain
    \begin{equation*}
        t^{\delta}\prod\limits_{i=1}^{n}\dfrac{1-t^{\deg(d_i)}}{1-t}\leq \sum\limits_{j=1}^{m}t^{\deg(u_j)}\leq \prod\limits_{i=1}^{n}\dfrac{1-t^{\deg(d_i)}}{1-t}, \text{ for all } t\in (0, 1).
    \end{equation*}
    Letting $t\to 1^{-}$ in the above inequality yields $m = \prod\limits_{i=1}^{n}\deg(d_i) = |G|$. Hence, $Frac(S)/Frac(R)$ is a field extension of degree $|G|$. We also know that $Frac(S)/Frac(S)^G$ is a Galois extension of degree $|G|$, and $Frac(R) \subseteq Frac(S)^G$, using \cref{lemma:integrality_and_fraction_field}, we deduce that $Frac(R) = Frac(S)^G = Frac(S^G)$; in particular, $S^G\subseteq Frac(R)$. We use the fact that $R = \mathbb{K}[d_1, \ldots, d_n]$ is an integrally closed domain, and that $S^G \subseteq S$ is integral over $R$, thus $R = S^G$, as desired.
\end{proof}
\begin{example}\label{example:symmetric_polynomial}
    Let $G=\Sigma_n$ be the symmetric group acting on $S = \mathbb{K}[x_1, \ldots, x_n]$ by permuting the variables. The following polynomial $F(X)\in S^G[X]$ has $x_1, \ldots, x_n$ as roots
    \begin{equation*}
        F(X) = \prod\limits_{i=1}^{n}(X-x_i) = X^n - e_1 X^{n-1} + \cdots  + (-1)^n e_n,
    \end{equation*}
    where $e_1, \ldots, e_n$ are the elementary symmetric polynomials\index{elementary symmetric polynomials} of degree $1, \ldots, n$, respectively. Since $|\Sigma_n| = n! = \prod\limits_{i=1}^n \deg(e_i)$, by \cref{prop:prod_of_deg_equals_G_implies_polynomial}, $S^G = \mathbb{K}[e_1, \ldots, e_n]$. Additionally, $S$ is a free $S^G$-module, generated by monomials
    \begin{equation*}
        \{x_1^{\alpha_1}\cdots x_n^{\alpha_n}\mid 0\leq\alpha_i \leq n-i, \text{ for all } i\in \{1, \ldots, n\}\}.
    \end{equation*}
\end{example}
\section{Rings of Invariants of Parabolic Subgroups}
In this section, we restrict our attention to the case $\mathbb{K} = \mathbb{F}_q$, and $G$ is a \emph{parabolic subgroup}\index{parabolic subgroup} of $GL_n(\mathbb{F}_q)$. 
\subsection{The Dickson Algebra}
\begin{definition}\label{def:flag}
    Let $V$ be a $\mathbb{K}$-vector space of dimension $n>0$. If $\alpha = (\alpha_1, \ldots, \alpha_{\ell})$ is a tuple of positive integers so that $n = \alpha_1 + \ldots + \alpha_{\ell}$ (such tuples are called \emph{compositions of $n$}.)\index{composition of a number} For each $i\in \{0, \ldots, \ell\}$, define the partial sum 
    \begin{equation*}
        A_i = \sum\limits_{k=1}^{i}\alpha_i
    \end{equation*}
    (here, $A_0$ is defined to be $0$.) A flag\index{flag} $F_{\bullet} = (V_0, \ldots, V_{\ell})$ of type $\alpha$ is a chain of subspaces 
    \begin{equation*}
        V_0 \subseteq V_1 \subseteq \cdots \subseteq V_{\ell} = V,
    \end{equation*}
    such that for any $i\in \{0, \ldots, \ell\}$, $\dim_{\mathbb{K}}(V_i) = A_i$. Consider a flag $F_{\bullet} = (V_0, \ldots, V_{\ell})$ of type $\alpha$, the parabolic subgroup of type $\alpha$ corresponding to $F_{\bullet}$ is defined to be
    \begin{equation*}
        P_{\alpha}(F_{\bullet}) = \{\sigma\in GL(V)\mid \sigma(V_i) = V_i, \text{ for all } i\in \{0, \ldots, \ell\}\}. 
    \end{equation*}
\end{definition}
It is not hard to see that for two flags $F_{\bullet}$ and $G_{\bullet}$ of type $\alpha$, there exists $\tau\in GL(V)$ such that
\begin{align*}
    c_{\tau}: P_{\alpha}(F_{\bullet}) &\to P_{\alpha}(G_{\bullet})\\ 
\sigma &\mapsto \tau \sigma \tau^{-1}
\end{align*}
is an isomorphism. Therefore, when the flag $F_{\bullet}$ is clear from context, we shall write $P_{\alpha}$ in place of $P_{\alpha}(F_{\bullet})$. In particular, fixing a basis $\{x_1, \ldots, x_n\}$ of $V$, the standard flag associated to this basis is $F_{\bullet} = (V_0, \ldots, V_{\ell})$, where
\begin{equation*}
    V_i = span_{\mathbb{K}}\{x_1, \ldots, x_{A_i}\}, \text{ for all } i\in \{0, \ldots, \ell\}.
\end{equation*}
The parabolic subgroup of type $\alpha$ corresponding to this flag will be denoted by $P_{\alpha}$ in the rest of this manuscript.
\begin{example}\label{example:parabolic_subgrps}
    \begin{enumerate}
        \item When $\alpha = (n)$, the group $P_{\alpha}$ is simply the full general linear group $GL(V)$\index{general linear group}.
        \item When $\alpha = \underbrace{(1, \ldots, 1)}_{n}$, the group $P_{\alpha}$ is called the \emph{Borel subgroup}\index{Borel subgroup}, and is often denoted by $B$. $B$ corresponds to the subgroup of $GL_n(\mathbb{K})$ consisting of upper triangular diagonal matrices. 
    \end{enumerate}
\end{example}
Following \cite{hewett_modular_1996, mui,wilkerson}, we analyze the structure of the invariant ring $S^G$, where $S = \mathbb{F}_q[x_1, \ldots, x_n]$ is the polynomial ring on $n$ indeterminates, and $G$ is one of the following groups.
\begin{itemize}
    \item The general linear group $GL_n(\mathbb{F}_q)$,
    \item The unipotent group\index{unipotent group} $U$ consisting of upper triangular matrices of which all diagonal entries equal $1$.
    \item The parabolic subgroups $P_{\alpha}$ for a composition $\alpha = (\alpha_1, \ldots, \alpha_n)$.
\end{itemize}
\begin{lemma}[{\cite[(7.6) and (7.16)]{macdonald_schur_1992}}]\label{lemma:f_W}
    Let $W$ be a finite-dimensional vector subspace of $S$. Define a map
    \begin{align*}
        f_W: S&\to S\\
        x&\mapsto \prod\limits_{w\in W}(x+w).
    \end{align*}
    \begin{itemize}
    \item[(1)] If $w_1, \ldots, w_m$ is a basis of $W$, then 
    \begin{equation*}
        f_W(x) = \dfrac{\begin{vmatrix}
        x^{q^m} & x^{q^{m-1}} & \cdots & x^{q^0}\\
        w_1^{q^m} & w_1^{q^{m-1}} & \cdots & w_1^{q^0}\\
        \vdots & \vdots & \ddots  & \vdots\\
        w_m^{q^m} & w_m^{q^{m-1}} & \cdots & w_m^{q^0}\\
        \end{vmatrix}}{\begin{vmatrix}
        w_1^{q^{m-1}} & w_1^{q^{m-2}} & \cdots & w_1^{q^0}\\
        \vdots & \vdots & \ddots  & \vdots\\
        w_m^{q^{m-1}} & w_m^{q^{m-2}} & \cdots & w_m^{q^0}\\
        \end{vmatrix}}, \text{ for all } x\in V.
    \end{equation*}
        \item[(2)] $f_W$ is a linear map between $\mathbb{F}_q$-vector spaces, and $\ker(f_W) = W$.
        \item[(3)] If $V$ is a subspace of $S$ containing $W$, define $V/W = f_W(V)$. Suppose that $W \subseteq V\subseteq U$ as vector subspaces of $S$, then $(U/W)/(V/W) = U/V$ as subspaces of $S$.
    \end{itemize}
\end{lemma}
\begin{proof}
    Let $A = (f_1, \ldots, f_k)$ be an ordered tuple of vectors in $S$, such that $A$ is linearly independent. A reduced vector with respect to $A$ is an element of the form $f_j + \sum\limits_{i=j+1}^{k}a_i f_i$, for some $j\in \{1, \ldots, k\}$, and $a_{j+1}, \ldots, a_k\in \mathbb{F}_q$. Let $R(A)$ be the set of reduced vectors with respect to $A$, then, the cardinality of $R(A)$ is $\sum\limits_{j=1}^{k}q^{j-1}$. Consider the free symmetric algebra over $\mathbb{F}_q$ generated by $A$, graded by the convention that $\deg(f_1) =\cdots = \deg(f_k)=1$. Then the determinant
    \begin{equation*}
        L(A) = \begin{vmatrix}
        f_1^{q^{k-1}} & f_1^{q^{k-2}} & \cdots & f_1^{q^0}\\
        \vdots & \vdots & \ddots  & \vdots\\
        f_k^{q^{k-1}} & f_k^{q^{k-2}} & \cdots & f_k^{q^0}\\
        \end{vmatrix}
    \end{equation*}
    is a homogeneous polynomial of degree $|R(A)| = \sum\limits_{j=1}^{k}q^{j-1}$. Furthermore, let $v = f_j + \sum\limits_{i=j+1}^{k}a_i f_i$, $L(A)$ is divisible by $v$, because after adding to the $j$th row the sum of $a_i$ times the $i$th row for $i\in \{j+1, \ldots, k\}$, each entry of the $j$th row is of the form
    \begin{equation*}
        f_j^{q^{\ell}} + \sum\limits_{i=j+1}^{k}a_i f_i^{q^{\ell}} = v^{q^{\ell}} 
    \end{equation*}
    for some $\ell\in \{0, \ldots, k-1\}$, therefore, after appropriately performing row operations, each entry of the $j$th row of the determinant defining $L(A)$ is divisible by $v$, hence $L(A)$ is divisible by $v$. Comparing degrees, it turns out that there exists $a\in \mathbb{F}_q^{\times}$ such that
    \begin{equation*}
        L(A) = a\cdot \prod\limits_{v\in R(A)}v.
    \end{equation*}
    Consider the graded lexicographic ordering on $\mathbb{F}_q[f_1, \ldots, f_k]$ following from the conventional ordering $f_1 >\cdots >f_k$. From the definition, the leading monomial of $L(A)$ is
    \begin{equation*}
        f_1^{q^{k-1}}f_2^{q^{k-2}}\cdots f_k^{q^0},
    \end{equation*}
    which is equal to the leading monomial of $\prod\limits_{v\in R(A)}v$. We conclude that $a=1$. \smallbreak
    Next, if $x\in W$, from the argument presented in the previous part of the proof, we see that $f_W(x) = 0$. If $x\notin W$, then $f_W(x)\neq 0$; let $A = \{x, w_1, \ldots, w_m\}$, then $A$ is an ordered tuple of vectors in $S$, and $A$ is linearly independent. Furthermore, we have
    \begin{equation*}
        R(A) = \{x+w\mid w\in W\}\sqcup R((w_1, \ldots, w_m)).
    \end{equation*}
    It turns out that
    \begin{align*}
        \dfrac{\begin{vmatrix}
        x^{q^m} & x^{q^{m-1}} & \cdots & x^{q^0}\\
        w_1^{q^m} & w_1^{q^{m-1}} & \cdots & w_1^{q^0}\\
        \vdots & \vdots & \ddots  & \vdots\\
        w_m^{q^m} & w_m^{q^{m-1}} & \cdots & w_m^{q^0}\\
        \end{vmatrix}}{\begin{vmatrix}
        w_1^{q^{m-1}} & w_1^{q^{m-2}} & \cdots & w_1^{q^0}\\
        \vdots & \vdots & \ddots  & \vdots\\
        w_m^{q^{m-1}} & w_m^{q^{m-2}} & \cdots & w_m^{q^0}\\
        \end{vmatrix}} &= \dfrac{L(A)}{L((w_1, \ldots, w_m))} \\
        &=\dfrac{\left(\prod\limits_{w\in W}(x+w)\right) \cdot L((w_1, \ldots, w_m))}{L((w_1, \ldots, w_m))}\\
        &= f_W(x).
    \end{align*}
    This equation proves the first claim in \cref{lemma:f_W}. The second claim follows from the first claim, because for any $x, y\in S, a, b\in \mathbb{F}_q$, we have
    \begin{align*}
        f_W(ax+by) &= \dfrac{\begin{vmatrix}
        ax^{q^m} + by^{q^m} & ax^{q^{m-1}} + by^{q^{m-1}} & \cdots & ax^{q^0} + by^{q^0}\\
        w_1^{q^m} & w_1^{q^{m-1}} & \cdots & w_1^{q^0}\\
        \vdots & \vdots & \ddots  & \vdots\\
        w_m^{q^m} & w_m^{q^{m-1}} & \cdots & w_m^{q^0}\\
        \end{vmatrix}}{\begin{vmatrix}
        w_1^{q^{m-1}} & w_1^{q^{m-2}} & \cdots & w_1^{q^0}\\
        \vdots & \vdots & \ddots  & \vdots\\
        w_m^{q^{m-1}}
        \end{vmatrix}}\\
        &= af_W(x) + bf_W(y).
    \end{align*}
    From the definition of $f_W$, $f_W(x) = 0$ if and only if $x+w=0$ for some $w\in W$, which is equivalent to $x\in W$. Thus $\ker(f_W) = W$.\smallbreak
    For the third claim, we only need to show that
    \begin{equation}\label{eq:quotient_f}
        f_{V/W}(f_W(u)) = f_V(u), \text{ for all } u\in U.
    \end{equation}
    Suppose that $V_0$ is a linear complement of $W$ in $V$, that is, $V = V_0 \oplus W$. It is not hard to see that $f_W: V_0\to V/W$ is an isomorphism, therefore,
    \begin{align*}
        f_{V/W}(f_W(u)) &= \prod\limits_{v_0\in V_0}(f_W(v_0)+f_W(u)) \\ 
        &=\prod\limits_{v_0\in V_0}\prod\limits_{w\in W}(w + v_0+u) \\
        &=\prod\limits_{v\in V}(v+u) = f_V(u).\qedhere
    \end{align*}
\end{proof}
Let $V\subseteq S$ be the $\mathbb{F}_q$-vector space generated by $\{x_1, \ldots, x_n\}$, from \cref{lemma:f_W}, we may write
\begin{equation*}
    f_V(x) = \sum\limits_{i=0}^{n}(-1)^i x^{q^{n-i}}Q_{n, n-i}(x_1, \ldots, x_n),
\end{equation*}
where for each $i$, the coefficient $Q_{n, n-i}(x_1, \ldots, x_n)$ is a homogeneous polynomial of degree $q^n - q^{n-i}$. Furthermore, each $Q_{n, n-i}$ is $GL_n(\mathbb{F}_q)$-invariant, since the definition of $f_V$ does not depend on the choice of a basis of $V$.
\begin{remark}
    The polynomial $f_V(x)$ has $V$ as its roots; in particular, it shows that $S$ is integral over the algebra generated by $Q_{n, 0}, \ldots, Q_{n, n-1}$. Therefore, equation $f_V(x)=0$ is often referred to as the \emph{fundamental equation}\index{fundamental equation}. More generally, we have the following lemma.
\end{remark}
\begin{lemma}[{\cite[Lemma 2]{steinberg_dicksons_1987}}]\label{lemma:showing_polynomial_of_the_right_deg}
    For each $i\in \{1, \ldots, n\}$, $V\backslash span_{\mathbb{F}_q}\{x_1, \ldots, x_{i-1}\}$ is the set of roots of the monic polynomial 
    \begin{equation*}
        F_i(x) = \dfrac{f_V(x)}{f_{span\{x_1, \ldots, x_{i-1}\}}(x)}.
    \end{equation*}
    This polynomial is of degree $q^n - q^{i-1}$, and furthermore, 
    \begin{equation*}
        F_i(x) \in\mathbb{F}_q[x_1, \ldots, x_{i-1}, Q_{n, i-1}, \ldots, Q_{n, n-1}][x].
    \end{equation*}
    \end{lemma}
    \begin{proof}
        From the definition, it is not hard to see that $F_i(x)$ is monic of degree $q^n-q^{i-1}$, and the set of roots of $F_i$ is as claimed in \cref{lemma:showing_polynomial_of_the_right_deg}. Suppose that
        \begin{equation*}
            F_i(x) = x^{q^{n}-q^{i-1}} + \sum\limits_{j=0}^{q^n - q^{i-1}-1}a_j x^{q^n - q^{i-1}-1-j},
        \end{equation*}
        where $a_j\in S, \text{ for all } j$. We prove by induction on $j$ that 
        \begin{equation}\label{eq:induct_on_j_integral}
        a_j\in \mathbb{F}_q[x_1, \ldots, x_{i-1}, Q_{n, 0}, \ldots, Q_{n, n-1}].    
        \end{equation}
        Notice that the coefficients of $f_{span_{\mathbb{F}_q}\{x_1, \ldots, x_{i-1}\}}(x)$ are in $\mathbb{F}_q[x_1, \ldots, x_{i-1}]$, and the coefficients of $f_V(x)$ are in $\mathbb{F}_q[Q_{n, 0}, \ldots, Q_{n, n-1}]$; furthermore, 
        \begin{equation}\label{eq:product_of_fundamental_eq}
        f_V(x) = F_i(x)\cdot f_{span\{x_1, \ldots, x_{i-1}\}}(x).    
        \end{equation}
        Comparing the coefficients of $x^{q^n-2}$ in both sides of \eqref{eq:product_of_fundamental_eq} shows that \eqref{eq:induct_on_j_integral} holds for $j=0$. Suppose that \eqref{eq:induct_on_j_integral} holds for all $j\leq j_0$, then comparing the coefficients of $x^{q^n - 3-j_0}$ in \eqref{eq:product_of_fundamental_eq} shows that \eqref{eq:induct_on_j_integral} holds for $j=j_0+1$. Moreover, $f_V(x)$ and $f_{span_{\mathbb{F}_q}\{x_1, \ldots, x_{i-1}\}}(x)$ are both homogeneous in $\mathbb{F}_q[x_1, \ldots, x_n, x]$, hence for each $j$, $a_j\in \mathbb{F}_q[x_1, \ldots, x_n]$ is of degree at most $q^n - q^{i-1}$. From \eqref{eq:induct_on_j_integral}, because $Q_{n, 0}, \ldots, Q_{n, i-2}$ are of degree larger than $q^n-q^{i-1}$, we conclude that
        \begin{equation*}
            F_i(x)\in \mathbb{F}_q[x_1, \ldots, x_{i-1}, Q_{n, i-1}, \ldots, Q_{n, n-1}][x].\qedhere
        \end{equation*}
    \end{proof}
\begin{definition}[{\cite{dickson_fundamental_1911}}]\label{def:dickson_invariants}
    The polynomials $\{Q_{n, 0}, Q_{n, 1}, \ldots, Q_{n, n-1}\}$ are called the \emph{Dickson invariants}\index{Dickson invariants}. The $\mathbb{F}_q$-algebra generated by $\{Q_{n, 0}, Q_{n, 1}, \ldots, Q_{n, n-1}\}$ is called the \emph{Dickson algebra}\index{Dickson algebra}.
\end{definition}
\begin{theorem}[{\cite{dickson_fundamental_1911}, \cite[Theorem B]{steinberg_dicksons_1987}}]\label{thm:dickson}
    The ring of invariants under the action of the general linear group is
    \begin{equation*}
        S^{GL_n(\mathbb{F}_q)} = \mathbb{F}_q[Q_{n, 0}, \ldots, Q_{n, n-1}].
    \end{equation*}
    Furthermore, $S$ is a free $S^{GL_n(\mathbb{F}_q)}$-module, generated by 
    \begin{equation*}
       \mathcal{G}= \{x_1^{\alpha_1}\cdots x_n^{\alpha_n}\mid 0\leq \alpha_i< q^n - q^{i-1}, \text{ for all } i\in \{1, \ldots, n\}\}.
    \end{equation*}
\end{theorem}
\begin{proof}
    The Remark preceding \cref{def:dickson_invariants} shows that $S$ is integral over the algebra generated by the $n$ Dickson invariants. Moreover, we have
    \begin{equation*}
        |GL_n(\mathbb{F}_q)| = \prod\limits_{i=1}^{n}(q^n - q^{n-i}) =\prod\limits_{i=1}^{n}\deg(Q_{n, n-i}). 
    \end{equation*}
    Applying \cref{prop:prod_of_deg_equals_G_implies_polynomial} gives the description of the invariant ring.\smallbreak
    Denote the Dickson algebra by $\mathcal{D}_n$. For each monomial $x^{\beta} = x_1^{\beta_1}\cdots x_n^{\beta_n}\notin \mathcal{G}$, let $m(\beta)$ be the largest index $i$ such that $\beta_i \geq q^n - q^{i-1}$, then by \cref{lemma:showing_polynomial_of_the_right_deg}, $x_i^{\beta_i} \in \mathcal{D}_n[x_1, \ldots, x_{i-1}]\cdot \{x_i^{\alpha_i}\mid 0\leq \alpha_i < q^n - q^{i-1}]$, which means that $x^{\beta}$ is in the $\mathcal{D}_n$-module generated by monomials $x^{\gamma}$ such that either $x^{\gamma}\in \mathcal{G}$ or $m(\gamma)<m(\beta)$. Continuing this process, we conclude that $S$ is generated by $\mathcal{G}$ as a $\mathcal{D}$-module. Using the proof of \cref{prop:prod_of_deg_equals_G_implies_polynomial} for $R = \mathcal{D}$, we see that $\mathcal{G}$ is also a generating set for $Frac(S)$ as a $Frac(\mathcal{D}) = Frac(S)^{GL_n(\mathbb{F}_q)}$-vector space. Observe that $|\mathcal{G}| = |GL_n(\mathbb{F}_q)|$, which implies that $\mathcal{G}$ is a basis of $Frac(S)$ as a $Frac(\mathcal{D})$-vector space. Therefore, $S$ is a free $\mathcal{D}$-module generated by $\mathcal{G}$.
\end{proof}
The Dickson invariants and the polynomials $f_W$ are building blocks for the invariant rings under the action of the unipotent group and the parabolic subgroups.
\begin{theorem}[\cite{mui}]\label{thm:mui_invariants}
    Let $U$ be the group of upper triangular matrices acting on $S$. For each $i\in \{1, \ldots, n\}$, define
    \begin{equation*}
        V_i = f_{span\{x_1, \ldots, x_{i-1}\}}(x_i)
    \end{equation*}
    (in particular, $V_1 = f_{\{0\}}(x_1) = x_1$.) The invariant ring under the action of $U$ is given by
    \begin{equation*}
        S^U = \mathbb{F}_q[V_1, \ldots, V_n].
    \end{equation*}
\end{theorem}
\begin{proof}
    For each $i$, $V_i$ is an element of $S^U$, because for any $\sigma\in U$, 
    \begin{equation*}
        \sigma(span_{\mathbb{F}_q}\{x_1, \ldots, x_{i-1}\}) = span_{\mathbb{F}_q}\{x_1, \ldots, x_{i-1}\},
    \end{equation*}
    and $\sigma(x_i) = x_i + w_0$ for some $w_0\in span_{\mathbb{F}_q}\{x_1, \ldots, x_{i-1}\}$, which implies that
    \begin{equation*}
        \sigma(V_i) = \prod\limits_{w\in span\{x_1, \ldots, x_{i-1}\}}(x_i + w_0 + w)= V_i.
    \end{equation*}
    Next, we show that the fundamental equation $f_V(x) = 0$ has coefficients in the algebra $R = \mathbb{F}_q[V_1, \ldots, V_n]$. Indeed, using \eqref{eq:quotient_f}, for each $i\in \{1, \ldots, n\}$, let $W_i = span_{\mathbb{F}_q}\{x_1, \ldots, x_i\}$, we have
    \begin{align*}
        f_V(x) &= f_{W_n/W_{n-1}}(f_{W_{n-1}}(x))\\
        &=f_{W_n/W_{n-1}}f_{W_{n-1}/W_{n-2}}f_{W_{n-2}}(x)\\
        & = \cdots\\
        &=f_{W_n/W_{n-1}}f_{W_{n-1}/W_{n-2}}\cdots f_{W_{2}/W_1}f_{W_1}(x).
    \end{align*}
    We have $f_{W_1}(x) = x^q - x\cdot V_1^{q-1} \in R[x]$, $f_{W_2/W_1}(f_{W_1}(x)) = f_{W_1}(x)^q - f_{W_1}(x)\cdot V_2^{q-1}\in R[x]$, similarly, observe that for each $i\in \{2, \ldots, n\}$, $W_i/W_{i-1}$ is a one-dimensional vector space, generated by $V_i$, we conclude that $f_V(x)\in R[x]$. Therefore, $S$ is integral over $R$; finally, it is obvious that 
    \begin{equation*}
        |U| = \prod\limits_{i=0}^{n-1}q^i =\prod\limits_{i=1}^{n}\deg(V_i). 
    \end{equation*}
    By \cref{prop:prod_of_deg_equals_G_implies_polynomial}, we conclude that $S^U = \mathbb{F}_q[V_1, \ldots, V_n]$.
\end{proof}
The unipotent invariants and the Dickson invariants are related by the following well-known formula.
\begin{proposition}[{\cite[Proposition 1.3]{wilkerson}}]\label{prop:decrease_deg}
    For any $n\geq 2$ and $i\in \{0, \ldots, n\}$, we have
    \begin{equation}\label{eq:decrease_deg}
        Q_{n, i}(x_1, \ldots, x_n) = Q_{n-1, i-1}(x_1, \ldots, x_{n-1})^{q} + V_n(x_1, \ldots, x_n)^{q-1} \cdot Q_{n-1, i}(x_1, \ldots, x_{n-1}).
    \end{equation}
\end{proposition}
\begin{proof}
    Let $V = span_{\mathbb{F}_q}\{x_1, \ldots, x_n\}$ and $W = span_{\mathbb{F}_q}\{x_1, \ldots, x_{n-1}\}$. From \cref{lemma:f_W}, $V/W$ is one-dimensional, and is generated by $V_n = f_W(x_n)$; consequently,
    \begin{equation*}
        f_{V/W}(x) = x^q - x\cdot V_n^{q-1}.
    \end{equation*}
    Furthermore, we have an equality of polynomials
    \begin{equation*}
        f_V(x)= f_{V/W}(f_W(x)).
    \end{equation*}
    Expanding both sides of the above equation, we obtain
    \begin{align*}
        \sum\limits_{i=0}^{n}(-1)^{n-i}Q_{n, i}x^{q^{i}} &= \left(\sum\limits_{j=0}^{n-1}(-1)^{n-1-j}Q_{n-1, j}x^{q^{j}}\right)^q - V_n^{q-1}\cdot \left(\sum\limits_{j=0}^{n-1}(-1)^{n-1-j}Q_{n-1, j}x^{q^{j}}\right)\\
        &= \sum\limits_{j=0}^{n}(-1)^{n-j}x^{q^{j}}\cdot (Q_{n-1, j-1}^{q} + V_n^{q-1}\cdot Q_{n-1, j}),
    \end{align*}
    from which \eqref{eq:decrease_deg} follows by regarding both sides as polynomials in $x$ and comparing coefficients.
\end{proof}
\begin{theorem}[{\cite[Theorem 2.6]{hewett_modular_1996}}]
    Let $\alpha = (\alpha_1, \ldots, \alpha_r)$ be an ordered tuple of positive integers, such that $n = \alpha_1 + \cdots + \alpha_r$. We reuse the notations $A_i$, $V_i$ in \cref{def:flag}. Let $P_{\alpha}$ be the parabolic subgroup of type $\alpha$, acting on $S = \mathbb{F}_q[x_1, \ldots, x_n]$. For each $s\in \{1, \ldots, r\}, i\in \{1, \ldots, \alpha_s\}$, define
    \begin{equation*}
        P_{s, i}=Q_{\alpha_s, i-1}(f_{V_{s-1}}(x_{A_{s-1}+1}), \ldots, f_{V_{s-1}}(x_{A_{s}})).
    \end{equation*}
    The invariant ring under the action of $P_{\alpha}$ is given by
    \begin{equation*}
        S^{P_{\alpha}} = \mathbb{F}_q[P_{s, i}\mid 1\leq s\leq r \land 1\leq i\leq \alpha_s].
    \end{equation*}
\end{theorem}
\begin{proof}
    First, we compute the order of $P_{\alpha}$. Notice that $P_{\alpha}$ is isomorphic to a subgroup of $GL_n(\mathbb{F}_q)$ consisting of all block matrices of the form
    \begin{equation*}
        \begin{pmatrix}
M_1 & * & \cdots & * & *\\
0 & M_2 & \cdots & * & *\\
\vdots & \vdots & \ddots & \vdots\\
0 & 0 & \cdots & M_{r-1} & * \\
0 & 0 & \cdots & 0 & M_r
        \end{pmatrix},
    \end{equation*}
    where $M_s\in GL_{\alpha_s}(\mathbb{F}_q), \text{ for all } s\in \{1, \ldots, r\}$. Therefore, 
    \begin{equation*}
        |P_{\alpha}| = \prod\limits_{s=1}^{r}|GL_{\alpha_s}(\mathbb{F}_q)|\cdot q^{A_{s-1}\cdot \alpha_s}.
    \end{equation*}
    Notice that for each $s\in \{1,\ldots, r\}$ 
    \begin{align*}
        \prod\limits_{i=1}^{\alpha_s}\deg(P_{s, i}) &= \prod\limits_{i=1}^{\alpha_s}q^{A_{s-1}}\cdot (q^{\alpha_s}-q^{i-1})\\
        &=|GL_{\alpha_s}(\mathbb{F}_q)|\cdot q^{\alpha_s \cdot A_{s-1}}.
    \end{align*}
    Therefore, $\prod\limits_{s=1}^{r}\prod\limits_{i=1}^{\alpha_s}\deg(P_{s, i}) = |P_{\alpha}|$. Next, suppose that $\sigma\in P_{\alpha}$, for some $s\in \{1, \ldots, r\}, i\in \{1, \ldots, \alpha_s\}$, we have $\sigma(V_{s-1}) = V_{s-1}$, and $\sigma(V_s) = V_s$; therefore, there exists $\sigma_0\in GL(V_s/V_{s-1})$ (here the quotient is understood as in \cref{lemma:f_W}), such that
    \begin{equation*}
        \sigma(f_{V_{s-1}}(x_{A_{s-1}+j})) = f_{V_{s-1}}(\sigma(x_{A_{s-1}+j})) = \sigma_0(f_{V_{s-1}}(x_{A_{s-1}+j})), \text{ for all } j\in \{1, \ldots, \alpha_s\}.
    \end{equation*}
    Notice that $P_{s, i}$ is a Dickson invariant in the symmetric algebra generated by $V_{s}/V_{s-1}$, thus, $\sigma(P_{s, i}) = \sigma_0(P_{s, i}) = P_{s, i}$. We deduce that
    \begin{equation*}
        \{P_{s, i}\mid s\in \{1,\ldots, r\}, i\in \{1, \ldots, \alpha_s\}\}\subseteq S^{P_{\alpha}}.
    \end{equation*}
    In order to apply \cref{prop:prod_of_deg_equals_G_implies_polynomial}, we only need to show that the fundamental equation $f_V(x)$ has coefficients in the algebra 
    \begin{equation*}
        R = \mathbb{F}_q[P_{s, i}\mid 1\leq s\leq r \land 1\leq i\leq \alpha_s].
    \end{equation*}
    Similar to the proof of \cref{thm:mui_invariants}, we have
    \begin{align*}
        f_{V_r}(x) &= f_{V_r/V_{r-1}}(f_{V_{r-1}}(x))\\
        & = \cdots \\
        & = f_{V_r/V_{r-1}}f_{V_{r-1}/V_{r-2}}\cdots f_{V_2/V_1}f_{V_1/V_0}(x).
    \end{align*}
    For each $s\in \{1, \ldots, r\}$, the polynomial $f_{V_s/V_{s-1}}(x)$ is an element of $R[x]$, since
    \begin{equation*}
        f_{V_s/V_{s-1}}(x) = x^{q^{\alpha_s}} + \sum\limits_{i=1}^{\alpha_s}(-1)^{\alpha_s-i+1}P_{s, i}\cdot x^{q^{i-1}}.  
    \end{equation*}
    As $f_V$ is a composition of these polynomials, it turns out that $f_V(x)\in R[x]$.
\end{proof}
\subsection{Schur Polynomials over Finite Fields}\label{subsection:schur_7}
Schur polynomials\index{Schur polynomial} are an important class of symmetric functions that play important roles in combinatorics, geometry, and representation theory \cite{sagan_symmetric_2001, stanley_enumerative_1999}. In \cite[Section 7]{macdonald_schur_1992}, Macdonald introduced a $q$-analogue of the classical Schur functions, which he called the $7$th variation of Schur functions; in \cite{shinyao}, the author called these polynomials \emph{generalized Dickson invariants}. In this subsection, we recall the definitions and results in \cite[Section 7]{macdonald_schur_1992}. 
\begin{definition}[{\cite[(7.1)]{macdonald_schur_1992}}]\label{def:schur}
    Let $\alpha = (\alpha_1 > \ldots > \alpha_n)$ be a decreasing sequence of nonnegative integers. Define
    \begin{equation*}
        A_{\alpha}(x_1, \ldots, x_n) = \begin{vmatrix}
            x_1^{\alpha_1} & x_1^{\alpha_2} & \cdots  & x_1^{\alpha_n}\\
            x_2^{\alpha_1} & x_2^{\alpha_2} & \cdots  & x_2^{\alpha_n}\\
            \vdots & \vdots & \ddots & \vdots \\
            x_n^{\alpha_1} & x_n^{\alpha_2} & \cdots  & x_n^{\alpha_n}.
        \end{vmatrix}
    \end{equation*}
    Let $\delta_n= (n-1, n-2, \ldots, 0)$, for each partition\index{partition} $\lambda = (\lambda_1\geq \lambda_2 \geq \ldots \geq \lambda_n\geq 0)$, define the Schur polynomial $S_{\lambda}$ by the following formula
    \begin{equation*}
        S_{\lambda}(x_1, \ldots, x_n) = \dfrac{A_{\lambda + \delta_n}(x_1, \ldots, x_n)}{A_{\delta_n}(x_1, \ldots, x_n)}.
    \end{equation*}
\end{definition}
\begin{remark}
\begin{enumerate}
    \item From the proof of \cref{lemma:f_W}, the denominator in the definition of $S_{\lambda}$ is divisible by any linear combination of the variables $x_1, \ldots, x_n$, which implies that $S_{\lambda}$ is indeed a polynomial in $\mathbb{F}_q[x_1, \ldots, x_n]$, Furthermore, for any $\sigma=(\sigma_{ij})_{i, j\in \{1, \ldots, n\}}\in GL_n(\mathbb{F}_q)$, we have 
    \begin{equation*}
        \sigma(x_a)^{q^b} = \left(\sum\limits_{i=1}^{n}\sigma_{ia}x_i\right)^{q^b} =\sum\limits_{i=1}^{n}\sigma_{ia}x_i^{q^b}, \text{ for all } b\geq 0, a\in \{1, \ldots, n\}, 
    \end{equation*}
    therefore, for any decreasing sequence of non-negative integers $\alpha$, 
    \begin{align*}
        \sigma\left(\begin{vmatrix}
            x_1^{\alpha_1} & x_1^{\alpha_2} & \cdots  & x_1^{\alpha_n}\\
            x_2^{\alpha_1} & x_2^{\alpha_2} & \cdots  & x_2^{\alpha_n}\\
            \vdots & \vdots & \ddots & \vdots \\
            x_n^{\alpha_1} & x_n^{\alpha_2} & \cdots  & x_n^{\alpha_n}
        \end{vmatrix}\right) &=\det\left(\sigma^T \cdot \begin{pmatrix}
            x_1^{\alpha_1} & x_1^{\alpha_2} & \cdots  & x_1^{\alpha_n}\\
            x_2^{\alpha_1} & x_2^{\alpha_2} & \cdots  & x_2^{\alpha_n}\\
            \vdots & \vdots & \ddots & \vdots \\
            x_n^{\alpha_1} & x_n^{\alpha_2} & \cdots  & x_n^{\alpha_n}
\end{pmatrix}\right)\\
&= \det(\sigma)\cdot A_{\alpha}(x_1, \ldots, x_n).
    \end{align*}
    We deduce that
    \begin{equation*}
        \sigma(S_{\lambda}(x_1, \ldots, x_n)) = \dfrac{\det(\sigma)A_{\lambda+\delta_n}}{\det(\sigma)A_{\delta_n}} = S_{\lambda}(x_1, \ldots, x_n), \text{ for all } \sigma\in GL_n(\mathbb{F}_q).
    \end{equation*}
    Therefore, each polynomial $S_{\lambda}$ belongs to the Dickson algebra $\mathcal{D}_n$. Consequently, the definition of $S_{\lambda}$ does not depend on the choice of a basis $\{x_1, \ldots, x_n\}$ of $V = span_{\mathbb{F}_q}\{x_1, \ldots, x_n\}$, hence we also use the notation $S_{\lambda}(V)$ to refer to the Schur polynomial $S_{\lambda}(x_1, \ldots, x_n)$.
    \item From the definition of the Dickson invariants, using the Laplace transform, for each $i\in \{0, \ldots, n-1\}$, suppose that $(1^{n-i})$ is the partition consisting of $(n-i)$ ones and $i$ zeros, then
    \begin{equation*}
    Q_{n, i} = \dfrac{A_{\delta_n + (1^{n-i})}}{A_{\delta_n}} = S_{(1^{n-i})}(V).
    \end{equation*}
    This property resembles a property of the classical Schur functions, that is, $s_{(1^{n-i})}(x_1, \ldots, x_n)$ equals the elementary symmetric polynomial of degree $(n-i)$. Therefore, Macdonald used the notation $E_{n-i}(V)$ in place of the Dickson invariant $Q_{n, i}$. It is also conventional to define $E_k(V) = 0$ if $k<0$ or $k>n$.\smallbreak
    Akin to the class of complete homogeneous symmetric polynomials, for each $r\geq 0$, let $(r) = (r, 0, \ldots, 0)$, and 
    \begin{equation*}
        H_r(V) = S_{(r)}(V).
    \end{equation*}
    We use the convention that $H_r(V) = 0$ if $r<0$. These polynomials play a similar role as the homogeneous symmetric polynomials in the finite field setting, as shown in the following results.
\end{enumerate}
\end{remark}
Let $\varphi: \mathbb{F}_q[V]\to \mathbb{F}_q[V]$ denote the Frobenius map
\begin{equation*}
    \varphi(u) = u^q, \text{ for all } u\in \mathbb{F}_q[V].
\end{equation*}
Recall that $\widehat{S}(V) = \cup_{r \geq 0} \mathbb{F}_q [x_1^{q^{-r}}, \ldots, x_n^{q^{-r}}]$ is the algebra obtained from $S(V) = \mathbb{F}_q [x_1, \ldots, x_n]$ by inverting the Frobenius map $\varphi$. In other words,
\begin{equation*}
\widehat{S}(V) = colim_{\varphi} (S(V) \xrightarrow{\varphi} S(V) \xrightarrow{\varphi} S(V) \xrightarrow{} \ldots ).    
\end{equation*}
Define two infinite upper-triangular matrices with entries in $\widehat{S}(V)$ as follows.
\begin{align*}
    E(V) &= ((-1)^{j-i}\varphi^j E_{j-i}(V))_{i, j\in \mathbb{Z}},\\
    H(V) & =(\varphi^{i+1} H_{j-i}(V))_{i, j\in \mathbb{Z}}.
\end{align*}
\begin{theorem}[{\cite[(7.9) and (7.17)]{macdonald_schur_1992}}]\label{thm:H_is_E_inverse}
    For any finite interval $I\subseteq \mathbb{Z}$, define
    \begin{align*}
    E_I(V) &= ((-1)^{j-i}\varphi^j E_{j-i}(V))_{i, j\in I},\\
    H_I(V) & =(\varphi^{i+1} H_{j-i}(V))_{i, j\in I}.
\end{align*}
Then $E_I(V) = H_I(V)^{-1}$. Furthermore, for any vector subspace $U$ of $V$, we have
\begin{equation*}
    H_I(V/U) = H_I(V)\cdot \varphi^{\dim(V) - \dim(U)}(E_I(U)).
\end{equation*}
\end{theorem}
The analogues of the Jacobi-Trudi, and the Nagelsbach-Kostka formulas\index{Jacobi-Trudi formula}\index{Nagelsbach-Kostka formula} \cite[(0.2), (0.3)]{macdonald_schur_1992} are
\begin{theorem}[{\cite[(7.10)]{macdonald_schur_1992}}]\label{thm:jacobi_trudi}
    For any partition $\lambda$, let $\lambda' = (\lambda'_1, \ldots, \lambda'_m)$ be the \emph{conjugate} partition of $\lambda$, obtained by transposing the Young diagram of $\lambda$. Then we have
    \begin{align*}
        S_{\lambda}(V) &= \det(\phi^{1-j}H_{\lambda_i - i+j}(V))_{1\leq i, j\leq n},\\
        &=\det(\phi^{j-1}E_{\lambda'_i-i+j}(V))_{1\leq i, j\leq m}.\\
    \end{align*}
\end{theorem}
\begin{proof}
    We recall the proof in \cite{macdonald_schur_1992} for completeness. By definition, for any $r\geq 0$, using the Laplace expansion for the first column of $A_{(r)+\delta_n}$, we see that $H_r(V)$ is of the form
    \begin{equation}\label{eq:H_r}
        H_r(V) = \sum\limits_{i=1}^{n}u_o\cdot \varphi^{n-r+1}(x_i),
    \end{equation}
    where the coefficients $u_i$ are rational functions that are independent of $r$. Therefore, if $\alpha = (\alpha_1 > \ldots > \alpha_n)\in \mathbb{Z}_{\geq 0}^n$, 
    \begin{equation*}
        \varphi^{1-j}(H_{\alpha_i-n+j}) = \sum\limits_{k=1}^{n}\varphi^{\alpha_i}(x_k)\varphi^{1-j}(u_k), \text{ for all } i, j\in \{1, \ldots, n\}.
    \end{equation*}
    Consequently,
    \begin{equation*}
        \left(\varphi^{1-j}(H_{\alpha_i-n+j})\right)_{1\leq i, j\leq n} = \left(\varphi^{\alpha_i}(x_k)\right)_{1\leq i, k\leq n}\cdot\left(\varphi^{1-j}(u_k)\right)_{1\leq k, j\leq n}. 
    \end{equation*}
    By taking determinants of both sides, we have
    \begin{equation*}
        \det\left(\varphi^{1-j}(H_{\alpha_i-n+j})\right)_{1\leq i, j\leq n} = A_{\alpha}\det\left(\varphi^{1-j}(u_k)\right)_{1\leq k, j\leq n}.
    \end{equation*}
    If $\alpha = \delta_n$, then for all $i, j$, $\alpha_i -n+j = n-i-n+j = j-i$, thus the left-hand side of the above equation becomes the determinant of an upper-triangular matrix with $1$'s on the diagonal, which is $1$. Therefore, $\det\left(\varphi^{1-j}(u_k)\right)_{1\leq k, j\leq n} = \dfrac{1}{A_{\delta_n}}$. Finally, taking $\alpha = \lambda+\delta_n$ gives the Jacobi-Trudi formula (expressing $S_{\lambda}$ in terms of $H_{r}$), and the second formula is deduced from the first \cref{thm:H_is_E_inverse} (see \cite[Chapter 1, (2.9)]{macdonald_symmetric_1979}.)
\end{proof}
In particular, \cref{thm:jacobi_trudi} shows that $S_{\lambda}(V)$ is an $n\times n$ minor of $H(V)$ having row indices $\{-1, \ldots, -n\}$, column indices $\{\lambda_1 - 1, \ldots, \lambda_n - n\}$. In a similar fashion, one can define the skew Schur functions for $\mathbb{F}_q$ as follows.
\begin{definition}[{\cite[(7.11), (7.11'), and (7.12)]{macdonald_schur_1992}}]\label{def:skew_schur}
    If $\lambda$ and $\mu$ are partitions of length at most $n$, define
    \begin{align*}
     S_{\lambda/\mu}(V) &= \det(\varphi^{\mu_j-j+1}H_{\lambda_i - \mu_j - i+j}(V))_{i, j\in \{1, \ldots, n\}}.   
    \end{align*}
    It turns out that for $m = \max\{\lambda_1, \mu_1\}$,
    \begin{equation*}
        S_{\lambda/\mu}(V) = \det(\varphi^{-\mu_j+j-1}E_{\lambda'_i - \mu'_j - i+j}(V))_{i, j\in \{1, \ldots, m\}}.
    \end{equation*}
    Moreover, $S_{\lambda/\mu} = 0$ unless $0\leq \lambda'_i -\mu'_i \leq n$, for all $1\leq i\leq \max\{\lambda_1, \mu_1\}$.
\end{definition}
From \cref{def:skew_schur}, we see that for partitions $\lambda, \mu$ of length at most $n$, $S_{\lambda/\mu}(V)$ is an $n\times n$ minor of $H(V)$ having row indices $\{\mu_1-1, \ldots, \mu_n-n\}$, column indices $\{\lambda_1 - 1, \ldots, \lambda_n - n\}$.
\begin{remark}
    When $\lambda_i\geq \mu_i$ for all $i\in \{1, \ldots, n\}$ (equivalently, $\lambda'_i\geq \mu'_i$ for all $1\leq i\leq \max\{\lambda_1, \mu_1\}$), we say that $\mu$ is a subpartition of $\lambda$, denoted by $\mu\subseteq \lambda$. Visually, the Young diagram of $\mu$ is contained in the Young diagram of $\lambda$.
\end{remark}
The class of Schur polynomials in the Dickson algebra also gives rise to an $\mathbb{F}_q$-basis of $\mathcal{D}_n$ as follows.
\begin{proposition}[\cite{stanton_another_2015}]\label{prop:stanton_basis}
    An $\mathbb{F}_q$-basis of $\mathcal{D}_n$ is given by
    \begin{equation*}
        \mathcal{B} =\left\{ S_\lambda(V) \cdot \prod\limits_{i=1}^{n}E_i(V)^{a_i} \mid 0\leq a_i\leq q^{\lambda_i} - 1, \text{ for all } i\in \{1,\ldots, n\} \right\}.
    \end{equation*}
\end{proposition}
\begin{proof}
    If $\lambda$ is a partition of length at most $n$, and $(a_1, \ldots, a_n)$ is a list of nonnegative integers such that $0\leq a_i\leq q^{\lambda_i} - 1, \text{ for all } i\in \{1,\ldots, n\}$, denote $B(\lambda, a_1,\ldots, a_n) = S_\lambda(V) \cdot \prod\limits_{i=1}^{n}E_i(V)^{a_i}$.
    First of all, we prove that the set $\mathcal{B}$ is linearly independent by showing that we can determine the polynomial $B(\lambda, a_1, \ldots, a_n)$ uniquely by its leading monomial (with respect to the lexicographic ordering $x_1 > \cdots > x_n$.) Indeed, suppose that $x_1^{\alpha_1}x_2^{\alpha_2}\cdots x_n^{\alpha_n}$ is the leading monomial of $B(\lambda, a_1, \ldots, a_n)$, then \begin{equation*}
        x_1^{\alpha_1}x_2^{\alpha_2}\cdots x_n^{\alpha_n} = \prod\limits_{i=1}^{n}x_i^{q^{n-i + \lambda_i} - q^{n-i}} \cdot \prod\limits_{i=1}^{n}\prod\limits_{k=1}^{i}x_k^{a_i(q^{n-k+1} - q^{n-k})}.
    \end{equation*} Comparing the power of $x_n$, we have
    \begin{equation*}
        \alpha_n = q^{\lambda_n} - 1 + a_n(q-1) = (q-1)(a_n + \sum\limits_{j=0}^{\lambda_n - 1} q^j).
    \end{equation*}
    Notice that $0\leq a_n\leq q^{\lambda_n}-1$, thus $\lambda_n$ and $a_n$ are uniquely determined by $\alpha_n$. Next, by comparing the power of $x_{i}$ for $1\leq i\leq n$, we have
    \begin{equation*}
        \alpha_i = q^{n-i+\lambda_{i}} - q^{n-i} + \left(\sum\limits_{k = i}^{n} a_k\right)(q^{n-i+1} - q^{n-i}) =q^{n-i}(q-1)\left(\sum\limits_{k = i}^{n} a_k + \sum\limits_{j=0}^{\lambda_i - 1} q^j\right).
    \end{equation*}
    By backward induction, we see that the partition $\lambda$ and the numbers $a_1, \ldots, a_n$ are uniquely determined by $\alpha_1, \ldots, \alpha_n$. Thus $\mathcal{B}$ is linearly independent. Now let $M$ be the $\mathbb{F}_q$-vector space generated by $\mathcal{B}$, we show that $M = \mathcal{D}_n$ by comparing their Hilbert series (because $M\subseteq \mathcal{D}_n$, we have $M=\mathcal{D}_n$ if and only if $Hilb(M, t) = Hilb(\mathcal{D}_n, t)$ as formal power series.) Indeed, we decompose $M$ as follows
    \begin{equation*}
        M = \bigoplus\limits_{\lambda}\mathbb{F}_q\left\langle S_\lambda(V) \cdot \prod\limits_{i=1}^{n}E_i(V)^{a_i} \hspace{0.1cm}|\hspace{0.1cm} 0\leq a_i\leq q^{\lambda_i} - 1, \text{ for all } i\in \{1,\ldots, n\}\right\rangle.
    \end{equation*}
    Each subspace in the above direct sum has Hilbert polynomial
    \begin{equation*}
        t^{\sum\limits_{i=1}^{n}q^{n-i+\lambda_i} - q^{n-i}}\cdot \dfrac{\prod\limits_{i=1}^{n}(1-t^{(q^n - q^{n-i)q^{\lambda_i}}})}{\prod\limits_{i=1}^{n}(1-t^{q^n - q^{n-i}})}=\prod\limits_{i=1}^{n}(t^{q^{n-i + \lambda_i} - q^{n-i}} - t^{q^{n+\lambda_i} - q^{n-i}})\cdot Hilb(\mathcal{D}_n, t).
    \end{equation*}
    Thus, we have to show that \begin{equation*}
        \sum\limits_{\lambda}\prod\limits_{i=1}^{n}(t^{q^{n-i + \lambda_i} - q^{n-i}} - t^{q^{n+\lambda_i} - q^{n-i}})=1.
    \end{equation*}
    The sum on the left-hand side can be easily computed as follows: first, fix the numbers $\lambda_2,\ldots, \lambda_{n}$, we see that the sum is actually a telescoping sum
    \begin{align*}
        &\sum\limits_{\lambda_2\geq \ldots\geq \lambda_n}\prod\limits_{i=2}^{n}(t^{q^{n-i + \lambda_i} - q^{n-i}} - t^{q^{n+\lambda_i} - q^{n-i}})\sum\limits_{\lambda_1 = \lambda_2}^{+\infty}(t^{q^{n-1 + \lambda_1} - q^{n-1}} - t^{q^{n+\lambda_1} - q^{n-1}})\\
      =& \sum\limits_{\lambda_2\geq \ldots\geq \lambda_n}\prod\limits_{i=2}^{n}(t^{q^{n-i + \lambda_i} - q^{n-i}} - t^{q^{n+\lambda_i} - q^{n-i}})\cdot t^{q^{n-1 + \lambda_2} - q^{n-1}}\\
      =&\sum\limits_{\lambda_3\geq \ldots\geq \lambda_n}\prod\limits_{i=3}^{n}(t^{q^{n-i + \lambda_i} - q^{n-i}} - t^{q^{n+\lambda_i} - q^{n-i}})\sum\limits_{\lambda_2 = \lambda_3}^{+\infty}(t^{q^{n-2 + \lambda_2} + q^{n-1 + \lambda_2} - q^{n-2} - q^{n-1}}- \\
      &-t^{q^{n+\lambda_2} + q^{n-1 + \lambda_2}- q^{n-2}- q^{n-1}})\\
      =&\sum\limits_{\lambda_3\geq \ldots\geq \lambda_n}\prod\limits_{i=3}^{n}(t^{q^{n-i + \lambda_i} - q^{n-i}} - t^{q^{n+\lambda_i} - q^{n-i}})\cdot t^{q^{n-1 + \lambda_3} + q^{n-2 + \lambda_3} - q^{n-1} - q^{n-2}}\\
      =&\cdots\\
      =& \sum\limits_{\lambda_n = 0}^{+\infty}\left(t^{\sum\limits_{j=0}^{n-1}(q^{j + \lambda_n} - q^j)} - t^{\sum\limits_{j=1}^{n}(q^{j + \lambda_n} - q^j)}\right)\\
      =& 1.
    \end{align*}
    Consequently, $M = \mathcal{D}_n$.
\end{proof}
\subsection{Stong-Tamagawa Formula for Schur Functions}\label{subsection:stong_tamagawa_gen}\index{Stong-Tamagawa formula}
We know that the Schur polynomials are $GL_n(\mathbb{F}_q)$-invariant, which means that they do not depend on the choice of a basis of $V$; therefore, it is reasonable to ask for a \emph{basis-free} description of $S_{\lambda}(V)$. There are known formulas of this kind for the Dickson invariants, which are a subclass of the Schur polynomials. For example, we always have
\begin{align*}
    Q_{n, 0} &= \prod\limits_{v\in V\backslash\{0\}}v\\
    &=\sum\limits_{v\in V}v^{q^n-1}.
\end{align*}
More generally, there is the Stong-Tamagawa formula for the Dickson invariants, which expresses each Dickson invariant as a sum of products of vectors over the set of complete flags of $V$ (see \cite[Proposition 8.1.3]{benson_1993}, \cite[Chapter 1, Section 2, Example 26]{macdonald_symmetric_1979}, \cite[(7.24)]{macdonald_schur_1992}.) We state and prove a generalization of the Stong-Tamagawa formula for the skew Schur functions, motivated by \cite[Conjecture (7.25)]{macdonald_schur_1992}.
\begin{definition}\label{def:T} If $\lambda$ is a partition, and $V$ is a finite-dimensional vector subspace of $\mathbb{F}_q[V]$, define $T_{\lambda} (V)$ as the following function in $\widehat{S}(V)$: 
\begin{equation*}
T_{\lambda} (V) = \det((-1)^{\lambda_i - i+ j}\varphi^{\lambda_i - i} E_{\lambda_i - i+ j}(V))_{1 \leq i, j \leq \ell (\lambda)} 
\end{equation*}
More generally, if $\mu$ is another partition, we define   
\begin{equation*}
    T_{\lambda/\mu}(V) = \det((-1)^{\lambda_i - \mu_j - i+ j}\varphi^{\lambda_i - i} E_{\lambda_i - \mu_j - i+ j}(V))_{1\leq i, j\leq \max\{\ell(\lambda), \ell(\mu)\}}.   
\end{equation*}
Here, $\ell(\lambda)$ is called the \emph{length}\index{length of a partition} of partition $\lambda$, which is the number of nonzero elements in $\lambda$.
\end{definition}
\begin{remark}
    By direct determinant manipulation, it is not hard to see that
    \begin{equation*}
        T_{\lambda/\mu}(V)=(-1)^{|\lambda|-|\mu|}\det(\varphi^{\lambda_i - i} E_{\lambda_i - \mu_j - i+ j}(V))_{1\leq i, j\leq \max\{\ell(\lambda), \ell(\mu)\}},
    \end{equation*}
    here, the \emph{weight}\index{weight of a partition} $|\lambda|$ of partition $\lambda$ is defined to be the sum of all elements of $\lambda$.
\end{remark}
\begin{lemma}\label{lemma:define_T_as_H}
    We have the following description of $T_{\lambda/\mu}(V)$ in terms of $H_r$: 
\begin{equation}\label{eq: T in H}
T_{\lambda/\mu} (V) = (-1)^{|\lambda|+|\mu|}\det (\varphi^{-\lambda_i' +i} H_{\lambda_i' - \mu_j' - i +j}(V))_{1\leq i, j\leq \max\{\lambda_1, \mu_1\}}. 
\end{equation}
\end{lemma}
Note that $T_{\lambda/\mu}(V)$ is of the same degree as that of $S_{\lambda/\mu}(V)$, which is $\sum (q^{n+\lambda_i -i} - q^{n + \mu_i -i})$, but it is not always a polynomial.
\begin{proof}
The proof is similar to the classical case, using \cref{thm:H_is_E_inverse}. For completeness, we will provide the proof here. Let $\lambda$ and $\mu$ be two partitions of length $\leq p$ such that $\lambda'$ and $\mu'$ have length $\leq q$. It is known \cite[Chapter 1]{macdonald_symmetric_1979} that $(-1+i-\lambda'_i)_{1\leq i\leq q}, (\lambda_i-i)_{1\leq i\leq p}$ is a partition of $I=\{-p, \ldots, -1, 0, 1, \ldots, q-1\}$, similarly, $(-1+i-\mu'_i)_{1\leq i\leq q}, (\mu_i-i)_{1\leq i\leq p}$ is a partition of $I$. From the definition, $T_{\lambda/\mu}(V)$ is the $p\times p$ minor of $E_I(V)$ with row indices $(\mu_i-i)_{1\leq i\leq p}$, column indices $(\lambda_i-i)_{1\leq i\leq p}$, which is equal to $(-1)^{|\lambda|+|\mu|}$ times the $q\times q$ minor of $H_I(V)$ with row indices $(-1+i-\lambda'_i)_{1\leq i\leq q}$, column indices $(-1+i-\mu'_i)_{1\leq i\leq q}$, therefore,\begin{equation*}
    T_{\lambda/\mu}(V) = (-1)^{|\lambda|+|\mu|}\cdot\det (\varphi^{-\lambda_i' +i} H_{\lambda_i' - \mu_j' - i +j}(V)).\qedhere 
\end{equation*}
\end{proof} 
\begin{lemma}\label{lemma:T_strips}
   $T_{\lambda/\mu}(V) = 0$ unless $0 \leq \lambda_i - \mu_i \leq \dim (V)$ for all $i$. In particular, if $T_{\lambda/\mu}(V) \neq 0$, then $\mu \subseteq \lambda$.
\end{lemma}
\begin{proof}
    If there exists an index $k$, $1 \leq k \leq n$, such that $\lambda_k < \mu_k$, then
    \begin{equation*}
   \lambda_i - \mu_j -i + j < 0 \quad \text{for all $i \geq k \geq j$}.   
    \end{equation*}
    So the entire block of size $(n-k+1) \times k$ in the lower left corner of the matrix defining $T_{\lambda/\mu}(V)$ equals zero. This implies that the determinant of this matrix must be zero. Similar argument applies to the case $\lambda_k - \mu_k > n$ for some $k$.  
\end{proof}
It is natural to ask whether $T_{\lambda/\mu}(V)$ also has a description in terms of the quotient of determinants similar to that of the Schur function, at least when $\lambda$ or $\mu$ is some special partition. The answer is yes, but more complicated. 
 \begin{theorem}\label{thm:T_as_determinant} 
 Suppose that $\lambda = (n^m)$, and $\mu\subseteq \lambda$. We have
 \begin{equation*}
 T_{\lambda/\mu} (V) = (-1)^{nm + |\mu|} \; \frac{\det ( x_i^{q^{n+j-\mu_j' -1}})_{1\leq i, j\leq n}}{\det (x_i^{q^{n+j-m -1}})_{1\leq i, j\leq n}}. 
 \end{equation*}  
 \end{theorem}
\begin{proof}
 We adapt the strategy of the proof of the fundamental equation relating $E(V)$ and $H(V)$. For any $r\geq 1-n$, expanding the determinant $A_{(r) + \delta_{n}}$ down the first column, $H_r (V)$ can be written in the form
 \begin{equation*}
 H_r = \sum_{k=1}^n u_k \varphi^{n+r-1} (x_k), 
 \end{equation*}
 where the coefficients $u_k$ are rational functions independent of $r$. From equation \eqref{eq: T in H}, notice that $\lambda'_i - \mu'_j - i + j = m-\mu'_j-i+j\geq j-i\geq 1-n$ for all $i$, $j$, we have
 \begin{align*}
   T_{\lambda/\mu}(V) =& (-1)^{nm + |\mu|}\det\left( \sum_{k=1}^n \varphi^{-m +i} (u_k)  \varphi^{-m +i+ n+m - \mu_j' - i +j-1} (x_k)\right)\\
   =&  (-1)^{nm + |\mu|}\det\left(\sum_{k=1}^n \varphi^{-m +i} (u_k))   (\varphi^{n - \mu_j'+j-1} (x_k)\right)\\
   =& (-1)^{nm + |\mu|}\det (\varphi^{-m +i} (u_j)) \times \det (\varphi^{n - \mu_j'+j-1} (x_i)).
 \end{align*}
In particular, if $\mu = \lambda$, then 
\begin{equation*}
1 = \det ( \varphi^{-m +i} (u_j)) \times \det (\varphi^{n - m+j-1} (x_i)). 
\end{equation*}
The result follows immediately. 
\end{proof} 
Let $\mu$ be a subpartition of $\lambda$. We say that $\lambda/\mu$ is a \emph{vertical strip}\index{vertical strip} if no two boxes in the skew shape $\lambda/\mu$ are in the same row.
\begin{lemma}\label{vertical_strip} 
    If $\dim V = 1$, then 
    \begin{equation*}
    T_{\lambda/\mu}(V) =
    \begin{cases}
   (-1)^{|\lambda/\mu|} \prod_{s \in \lambda/\mu} \varphi^{c(s)} E_1(V) & \text{if $\lambda/\mu$ is a vertical strip;}\\
   0 & \text{otherwise.} 
    \end{cases}
    \end{equation*}
\end{lemma}
Here $c(s)$ is the content of $s$. That is, $c(s) = j-i$ if $s = (i,j)$. 
\begin{proof}
 Let us analyze the case $0 \leq \lambda_i - \mu_i \leq 1$ for all $1 \leq i \leq n$. If $\lambda/\mu$ is a vertical strip, then  $\lambda_i -\mu_j -i + j \leq -1$ for all $i \geq j+2$ and $\lambda_i -\mu_j -i + j \geq 2$ for all $j \geq i+2$. Hence the matrix under consideration is a tridiagonal one.  
 
If $\lambda_i - \mu_i = 0$ for some $i$, then the $(i,i-1)$ entry in this matrix is zero, and the determinant does not change if the row $i$th and column $i$th are removed. Repeating this process if necessary, we obtain a lower-triangular matrix in which the diagonal entries are exactly $- \varphi^{\lambda_i - i} E_1(V)$ in row $i$ where $\lambda_i - \mu_i = 1$. On the other hand, it is clear that 
\begin{equation*}
\lambda/\mu = \{(i, \lambda_i) \colon \lambda_i - \mu_i = 1, 1 \leq i \leq m\}. 
\end{equation*}
 The proof is finished.   
\end{proof}
\begin{proposition}\label{new_decompose}
Let $U$ be a subspace of $V$. Then for all partitions $\lambda, \mu$, we have 
\begin{equation}\label{eq: SST}
    S_{\lambda/\mu}(V/U) = \sum\limits_{\mu \subseteq \nu\subseteq \lambda} S_{\nu/\mu}(V)\cdot \varphi^{\dim(V/U)}T_{\lambda/\nu}(U), \quad \text{and}  
\end{equation}
\begin{equation}\label{eq: TST}
    T_{\lambda/\mu}(V/U) = \sum\limits_{\mu \subseteq \nu\subseteq \lambda} \varphi^{\dim(V/U)} S_{\nu/
    \mu}(U) \cdot T_{\lambda/\nu}(V). 
\end{equation}
\end{proposition}
\begin{proof}
    Let $m = \max(l(\lambda), l(\mu))$. From the equation $(7.17) (ii)$ in \cite{macdonald_schur_1992}, we have\
    \begin{equation}\label{inversed}
        H(V)\cdot \varphi^{\dim(V) -\dim(U)}(E(U)) = H(V/U).
    \end{equation}
    (Notice that $E(U) = H(U)^{-1}$). Furthermore, $S_{\lambda/\mu}(V/U)$ is an $m\times m$ minor of $H(V/U)$ having row indices $\{\mu_1-1, \ldots, \mu_m -m\}$, column indices $\{\lambda_1 - 1, \ldots, \lambda_{,} - m\}$ (which is actually an entry of the exterior power $\Lambda^{m}(H(V/U))$); $S_{\nu/\mu}(V)$ is an $m\times m$ minor of $H(V)$ having row indices $\{\mu_1-1, \ldots, \mu_m - m\}$, column indices $\{\nu_1 - 1, \ldots, \nu_{m} - m\}$ (which is actually an entry of the exterior power $\Lambda^{m}(H(V))$); $T_{\lambda/\nu}(U)$ is an $m\times m$ minor of $E(U)$ having row indices $\{\nu_1-1, \ldots, \nu_{m} -m\}$, column indices $\{\lambda_1 - 1, \ldots, \lambda_{m} - m\}$ (which is actually an entry of the exterior power $\Lambda^{m}(E(U))$). Furthermore, from \eqref{inversed} and the functoriality of the exterior product, we have \begin{equation}\label{functoriality}
        \Lambda^{m}(H(V/U)) = \Lambda^{m}(H(V))\cdot \Lambda^{m}(\varphi^{\dim(V) -\dim(U)}(E(U))).
    \end{equation}
    The desired equality follows directly from \eqref{functoriality} by matrix multiplication. 

    The proof of equation \eqref{eq: TST} is similar.
\end{proof}
The following theorem is a generalization of \cite[Conjecture (7.25)]{macdonald_schur_1992}.
\begin{theorem}\label{gen_of_7.25}
Let $U$ be a vector subspace of $V$, and let $\mu\subseteq \lambda$ be partitions.  If $\lambda/\mu$ has strictly less than $\dim(U)$ nonzero rows, then we have
\begin{equation*}
    S_{\lambda/\mu}(V) = \sum\limits_{\substack{L\subseteq U\\
    \dim(L) = 1}}S_{\lambda/\mu}(V/L).  
\end{equation*}
Similarly, if $\lambda/\mu$ has strictly less than $\dim (U)$ nonzero columns, then we have 
\begin{equation*}
    T_{\lambda/\mu}(V) = \sum\limits_{\substack{L\subseteq U\\
    \dim(L) = 1}}T_{\lambda/\mu}(V/L). 
\end{equation*}
\end{theorem}
The original conjecture is for $\mu = 0$ and $U =V$. 
\begin{proof}
        We use the following result from \cite[Proposition 9.5]{Campbell1996BasesFR}.
\begin{lemma}\label{power_sum} 
Let $U$ be a vector space over $\mathbb{F}_q$ and let $\ell$ be a positive integer. If $\sum\limits_{u\in U} u^{\ell} \neq 0$ in $S(U)$, then $(q-1)$ divides $\ell$ and the sum of digits of $\ell$ in its $q$-adic representation is at least $(q-1) \dim U$.
    \end{lemma}
 \begin{corollary}
Let $U$ be a $k$-dimensional vector space over $\mathbb{F}_q$; for $1 \leq  \ell <k$, let $0\leq a_1 <a_2 < \cdots < a_{\ell}$ be integers and denote $m = \sum\limits_{j=1}^{\ell} q^{a_j}$. Then 
\begin{equation}\label{middle_sum}
    \sum\limits_{\substack{L\subseteq U\\ \dim(L) =1 }} E_1(L)^m = 0.
\end{equation}
 \end{corollary}
     Using Proposition \ref{new_decompose} for the $1$-dimensional subspaces in $U \subset V$, because the number of lines in $U$ is congruent to $1$ modulo $q$, we have
     \begin{equation*}
        \sum\limits_{\substack{L\subseteq U\\
    \dim(L) = 1}}S_{\lambda/\mu}(V/L) = S_{\lambda/\mu}(V) +  \sum\limits_{\substack{\mu \subseteq \nu\subsetneq \lambda \\ \lambda/\nu \text{ is a vertical strip}}}S_{\nu/\mu}(V)\cdot\left(\sum\limits_{\substack{L\subseteq U\\
    \dim(L) = 1}} \varphi^{n-1}(T_{\lambda/\nu}(L))\right). 
    \end{equation*}
    By assumption, $\lambda/\nu$ has less than $\dim(U)$ boxes, thus from Lemma \ref{vertical_strip} and equation \eqref{middle_sum}, we have
    \begin{equation*}
        \sum\limits_{\substack{L\subseteq U\\
    \dim(L) = 1}}S_{\lambda/\mu}(V/L) = S_{\lambda/\mu}(V).
    \end{equation*}
    For the second part, we argue similarly, using the equation
     \begin{equation*}
        \sum\limits_{\substack{L\subseteq U\\
    \dim(L) = 1}} T_{\lambda/\mu}(V/L) = T_{\lambda/\mu}(V) +  \sum\limits_{\substack{\mu \subsetneq \nu\subseteq \lambda \\ \lambda/\nu \text{ horizontal strip}}}  \left(\sum\limits_{\substack{L\subseteq U\\
    \dim(L) = 1}} \varphi^{n-1}(S_{\nu/\mu}(L))\right) T_{\lambda/\nu} (V). 
    \end{equation*}
    The proof is completed.
\end{proof}
The following formula is the generalization of the Stong-Tamagawa formula for skew Schur functions. Its proof is an application of \cref{gen_of_7.25}.
\begin{theorem}\label{main_rewrite}
Let $V$ be an $n$-dimensional vector space over $\mathbb{F}_q$, and let $\mathcal{F}$ denote the the set of all complete flags\index{complete flag} $\underline{V} = (V=V_0 \supset \ldots \supset V_{n-1} \supset V_n = 0$). For any partition $\lambda$ of length at most $n$, we have the following formula
\begin{equation*}  
    S_\lambda(V) = (-1)^{|\lambda|} \sum\limits_{\underline{V} \in \mathcal{F}}\prod\limits_{i=1}^{n} \pi_i (\underline{V})^{\frac{q^{\lambda_{i}}-1}{q-1}},
\end{equation*}
where for a subspace $W \subset \mathbb{F}_q[V]$, $\pi(W)$ is the product of all non-zero vectors in $W$, and for a full flag $\underline{V} = (0 = V_n \subsetneq V_{n-1}\subsetneq \cdots \subsetneq V_1\subsetneq V_0 = V)$, and $\pi_i(\underline{V}) = \pi(V_{i-1}/V_i)$ for $1 \leq i \leq n$.
\end{theorem}
\begin{proof}
If $\ell (\lambda) = n$, in other words $\lambda_n > 0$, then $\lambda_1' = n$ and $E_{\lambda_1' - 1+j} = 0$ unless $j=1$, in which case it is $E_n$. Thus 
\begin{align*}
S_{\lambda}(V) =& \det (\varphi^{j-1} E_{\lambda_i' - i +j})\\
=&  E_n (V) \det (\varphi^{j-1} E_{\lambda_i' - i +j})_{i,j \geq 2}\\
=& E_n (V) \varphi S_{(\lambda_1-1, \ldots, \lambda_n-1)} (V). 
\end{align*}
Repeating this process, we get 
\begin{equation*}
S_{\lambda} (V) = E_n(V)^{\frac{q^{\lambda_n}-1}{q-1}} \varphi^{\lambda_n} S_{(\lambda_1 - \lambda_n, \ldots, \lambda_{n-1}-\lambda_n, 0)} (V). 
\end{equation*}
Now the length of $(\lambda_1 - \lambda_n, \ldots, \lambda_{n-1} - \lambda_n,0)$ is strictly less than $n = \dim V$. We can apply Theorem \ref{gen_of_7.25} in the case $\mu = 0$: 
\begin{equation*}
S_{\lambda} (V) =  E_n(V)^{\frac{q^{\lambda_n}-1}{q-1}} \sum_{V_{n-1} \leq V, \dim V_{n-1} =1} S_{(\lambda_1 - \lambda_n, \ldots, \lambda_{n-1} - \lambda_n)}^{q^{\lambda_n}} (V/V_{n-1}).
\end{equation*}
Continue this process for each $(n-1)$-dimensional spaces $V/V_{n-1}$ and partitions $(\lambda_1- \lambda_n, \ldots, \lambda_{n-1}-\lambda_n)$ of length at most $(n-1)$, we can write $S_{\lambda}(V)$ as a sum over all complete flags: 
\begin{equation*}
\sum_{\underline{V} \in \mathcal{F}} E_{n}(V)^{\frac{q^{\lambda_n}-1}{q-1}}   E_{n-1}(V/V_{n-1})^{\frac{q^{\lambda_{n-1}} -q^{\lambda_n}}{q-1}}  E_{n-2}(V/V_{n-2})^{\frac{q^{\lambda_{n-2}} -q^{\lambda_{n-1}}}{q-1}}  \ldots E_{1}(V/V_{1})^{\frac{q^{\lambda_{1}}-q^{\lambda_2}}{q-1}}   
\end{equation*}
The result now follows from the fact that
\begin{equation*} 
E_{\dim V/U} (V/U) = (-1)^{\dim (V/U)} \pi (V/U) =(-1)^{\dim V - \dim U} \frac{\pi (V)}{\pi (U)}. 
\end{equation*}
For the general case, if $\lambda/\mu$ has exactly $n$ rows, then $\lambda_i - \mu_i > 0$ for all $1 \leq i \leq n$. It follows that $\lambda_1' = n$. If $\mu_1' = n$ then the first column of $S_{\lambda/\mu} (V)$ consists of all zero except in entries $(1,1)$ where it equals $\varphi^{-n} E_0 = 1$. We then have 
\begin{equation*}
S_{\lambda/\mu} (V) = S_{(\lambda_1-1, \ldots, \lambda_n-1)/(\mu_1-1, \ldots, \mu_n-1)} (V). 
\end{equation*}
Continue this process $\mu_n$ times, we can assume that $\mu_1' <n$ (or equivalently, $\mu_n = 0$).   
\end{proof}
We can obtain a similar description for the function $T_{\lambda}(V)$. It is known that $T_{\lambda} (V) = 0$ if $\lambda_1 > n$. If $\lambda_1 = n$, then by considering the top row, we have
\begin{align*}
    T_{\lambda} (V) =& (-1)^{n} \varphi^{n-1} E_n \cdot \det ((-1)^{\lambda_i - i+ j}\varphi^{\lambda_i - i} E_{\lambda_i - i+ j}(V))_{i,j \geq 2}\\
    =& (-1)^{n} \varphi^{n-1} E_n \cdot \varphi^{-1}T_{(\lambda_2, \ldots, \lambda_n)}(V) 
\end{align*}
 Thus for each $k$, $1 \leq k \leq n$, if the number of occurrences of $k$ in the partition $\lambda$ is $m_k$, then 
 \begin{align*}
     T_{\lambda}(V) &= \sum\limits_{\underline{V}\in \mathcal{F}}(-1)^{nm_n}E_n(V)^{\frac{q^n - q^{n-m_n}}{q-1}}\cdots (-1)^{1\cdot m_1}E_{1}(V/V_1)^{\frac{q^{1-m_n-\cdots-m_2}-q^{1-m_n-\cdots - m_{1}}}{q-1}}\\
     &=(-1)^{|\lambda|}\cdot \sum\limits_{\underline{V}\in \mathcal{F}}\prod\limits_{i=1}^{n}\left((-1)^i\cdot \dfrac{\pi(V)}{\pi(V_i)}\right)^{\frac{q^{i-\lambda'_{i+1}}-q^{i-\lambda'_i}}{q-1}}\\
     &=\sum\limits_{\underline{V}\in \mathcal{F}}\prod\limits_{i=1}^{n}\pi_i(V)^{m_i(\lambda)},
 \end{align*}
 here $m_i(\lambda) = \sum\limits_{k=i}^{n}\dfrac{q^{k-\lambda'_{k+1}}-q^{k-\lambda'_k}}{q-1}, 1\leq i\leq n$. In particular, if $\lambda = (1^r)$, then 
 \begin{equation*}
     T_{(1^r)}(V) =\sum\limits_{\underline{V}\in \mathcal{F}}\pi_1(V)^{\frac{q-q^{1-r}}{q-1}}. 
 \end{equation*}
If $\lambda = (r) (1\leq r\leq n)$, then
\begin{equation*}
    T_{(r)}(V)=\sum\limits_{\underline{V}\in \mathcal{F}}\left(\prod\limits_{i=1}^{r}\pi_i(V)\right)^{q^{r-1}}.
\end{equation*}
\begin{remark}
    A special case of Theorem \ref{main_rewrite} for all partitions $\lambda = (1^r) (0\leq r\leq n)$ was proven in \cite[Chapter I, Section 2, Example 26(e)]{macdonald_symmetric_1979}.
\end{remark}
\begin{remark}
Suppose $(v_1, \ldots, v_n)$ is a basis for $V$ such that $V_i$ is spanned by $(v_1, v_{2}, \ldots, v_{n-i})$. Then 
\begin{equation*}
\pi_i (\underline{V}) = \pi (V_{i-1}/V_{i}) = \prod_{v \in V_{i-1} - V_{i}} v = L(v_1, \ldots, v_{n-i+1})^{q-1}. 
\end{equation*}
and the result does not depend on the choice of such a basis for the flag $\underline{V}$. We can then rewrite $S_{\lambda}(V)$ as follows. 
\begin{equation}
S_{\lambda} (V) = (-1)^{|\lambda|} \sum_{\underline{V} \in \mathcal{F}} L (v_1, \ldots, v_n)^{q^{\lambda_1}-1} \ldots L(v_1)^{q^{\lambda_n}-1}. 
\end{equation}
\end{remark}
 \begin{example} If $\lambda = (r)$, then 
 \begin{equation*}
 H_r (V) = (-1)^r \sum_{\underline{V} \in \mathcal{F}} L(v_1, \ldots, v_n)^{q^r-1}. 
 \end{equation*}   
 \end{example}
\begin{example}
If $\lambda = 1^r$, we have 
\begin{equation*}
E_r (V) = (-1)^r\sum_{\underline{V} \in \mathcal{F}} (L(v_1, \ldots, v_n) L(v_1, \ldots, v_{n-1}) \cdots L(v_1, \ldots, v_{n-r+1}))^{q-1}.
\end{equation*}
\end{example}
 \begin{remark}
The two descriptions of the Schur function are both useful. For example, when a complete flag $\underline{V}$ is fixed, choose an adapted basis $(v_1, \ldots, v_n)$ for this flag, that is $V_i$ is spanned by $(v_1, \ldots, v_{n-i})$. Then the Schur function $S_{\lambda/\mu} (V)$ can be written in terms of upper triangular invariants as follows: 
\begin{equation*}
S_{\lambda/\mu} (V) = (-1)^{|\lambda - \mu|} \sum_T \prod_{s \in \lambda/\mu} \varphi^{T(s) -i+j-1} \pi_{T(s)} (\underline{V}). 
\end{equation*}
where the sum is over all column strict tableaux $T \colon \lambda - \mu \to [1,n]$.

For example, if $\lambda = 1^r, \mu = 0$, a column strict tableaux $T$ is just an $r$-tuple of integers $1 \leq j_1 < \ldots \leq j_r \leq n$. Therefore
\begin{equation*}
E_r (V) = (-1)^r \sum_{1 \leq j_1 < \ldots <j_r \leq n} \prod_{i=1}^r L(v_1, \ldots, v_{n-j_i+1})^{q^{j_i-i}(q-1)}. 
\end{equation*}
Notice that in the product, while the $r$-tuple of indices of the upper triangular invariant is a strictly decreasing sequence.
\begin{equation*}
n \geq n-j_1 +1 > n-j_2 +1 > \ldots > n-j_r +1 \geq 1,
\end{equation*}
the exponents is non-increasing: $0 \leq j_1-1 \leq j_2-2 \leq \ldots \leq j_r-r \leq n-r$. This description is equivalent to a similar description in \cite[Theorem 4.3]{Campbell1996}. 
 \end{remark}
\section{Truncated Polynomial Rings}\label{sec:truncated_rings}
We review the materials in \cite{Lewis_Reiner_Stanton_2017} about the conjectural Hilbert series of the invariant rings of truncated polynomial rings, under the action of parabolic subgroups.\smallbreak
\begin{definition}[{\cite[Section 1]{Lewis_Reiner_Stanton_2017}}]\label{def:truncated_rings}
Let $G = GL_n(\mathbb{F}_q)$ act via invertible linear substitutions on the polynomial ring $S=\mathbb{F}_q[x_1, \ldots, x_n]$. For each $m\geq 0$, define an ideal
$I_{n, m} = (x_1^{q^m}, \ldots, x_n^{q^m})$, and 
\begin{equation*}
    Q(m, n) = S/I_{n, m}.
\end{equation*}
The rings $Q(m, n)$ are called \emph{truncated polynomial rings}\index{truncated polynomial rings}.
\end{definition}
\begin{remark}
    It is clear that $I_{n, m}$ is stable under the action of $GL_n(\mathbb{F}_q)$, hence the action of $G$ on $S$ naturally induces an action of $G$ on $Q(m, n)$.
\end{remark}
\begin{definition}[{\cite{reiner_qt-analogues_2010}}]\label{def:q_t_multinomial_coefficient}
    For $n\in \mathbb{N}$ and a composition $\alpha = (\alpha_1, \ldots, \alpha_{\ell})\in \mathbb{Z}_{\geq 0}^{\ell}$ of $n$, for each $i\in \{0, \ldots, n\}$, define the partial sum
    \begin{equation*}
        A_i = \sum\limits_{j=1}^{i}\alpha_j.
    \end{equation*}
    The $(q, t)$-multinomial coefficient\index{$(q, t)$-multinomial coefficient} is defined to be
    \begin{equation*}
    \genfrac{[}{]}{0pt}{}{n}{\alpha}_{q, t} = \dfrac{Hilb(S^{P_{\alpha}}, t)}{Hilb(S^{G}, t)} = \dfrac{\prod\limits_{k=0}^{n-1}\left(1-t^{q^n-q^k}\right)}{\prod\limits_{i=1}^{\ell}\prod\limits_{j=0}^{\alpha_i-1}\left(1-t^{q^{A_i} - q^{A_{i-1}+j}}\right)}.
    \end{equation*}
\end{definition}
The parabolic conjecture in \cite{Lewis_Reiner_Stanton_2017} predicts the Hilbert series of the invariant ring $Q(m, n)^{P_{\alpha}}$ as follows. For each $\beta \in \mathbb{Z}_{\geq 0}^{\ell}$, we say that $\beta\leq \alpha$ if $\beta_i\leq \alpha_i$, for all $i\in \{1, \ldots, \ell\}$. We also denote $|\beta| = \sum\limits_{j=1}^{\ell}\beta_j$, and the partial sums $B_i = \sum\limits_{j=1}^{i}\beta_j$ for $i\in \{1, \ldots, \ell\}$. 
\begin{conjecture}[{\cite[Parabolic conjecture 1.5]{Lewis_Reiner_Stanton_2017}}]\label{conj:parabolic_conj}\index{parabolic conjecture}
    For $m\geq 0$ and $\alpha$ a composition of $n$, the invariant ring $Q(m, n)^{P_{\alpha}}$ has Hilbert series
    \begin{equation*}
        Hilb(Q(m, n)^{P_{\alpha}}, t) = \sum\limits_{\substack{\beta\leq \alpha\\ \beta\leq m}}t^{e(m, \alpha, \beta)}\genfrac{[}{]}{0pt}{}{m}{\beta, m-|\beta|}_{q, t},
    \end{equation*}
    where 
    \begin{equation*}
        e(m, \alpha, \beta) = \sum\limits_{i=1}^{\ell}(\alpha_i - \beta_i)(q^m-q^{B_i}).
    \end{equation*}
\end{conjecture}
\begin{remark}
    \cref{conj:parabolic_conj} is proven for the following cases.
    \begin{itemize}
        \item[(1)] In \cite{Ha_Hai_Nghia_2024}, the authors proved the case $\alpha = (1, 1, \ldots, 1)$. In this case, $P_{\alpha} = B$, the Borel subgroup of $GL_n(\mathbb{F}_q)$. 
        \item[(2)] $\alpha = (n)$. In this case, $P_{\alpha} = GL_n(\mathbb{F}_q)$ (see \cite{ha_hai_nghia_GL}.)
        \item[(3)] For $n\leq 3$ and all $\alpha$ (see \cite{ha_hai_nghia_parabolic_for_low_rank}.)
    \end{itemize}
    We shall investigate the proven cases in more detail in \cref{chap:trucated_invariants}. 
\end{remark}
Following \cite{Lewis_Reiner_Stanton_2017}, we give some structural properties of $Q(m, n)$ as a representation of $G = GL_n(\mathbb{F}_q)$. First, consider the following non-degenerate, symmetric bilinear form $B: Q(m, n)\times Q(m, n) \to \mathbb{F}_q$, given by the formula
\begin{equation*}
    B(x^{\alpha}, x^{\beta}) = \begin{cases}
        1 & \text{ if $\alpha_i + \beta_i = n(q^m-1)$ for all $i$,}\\
        0 & \text{ otherwise}.
    \end{cases}
\end{equation*}
\begin{proposition}[{\cite[Proposition 3.1]{Lewis_Reiner_Stanton_2017}}]
    The bilinear form $B$ is $G$-equivariant, that is, for any $\sigma\in G$, $f, g\in Q(m, n)$, we have
    \begin{equation*}
        B(f, g) = B(\sigma(f), \sigma(g)).
    \end{equation*}
\end{proposition}
\begin{proof}
    We only need to show that for $h = fg\in Q(m, n)$, the coefficient of $(x_1\cdots x_n)^{q^m-1}$ in $h$ is equal to the coefficient of $(x_1\cdots x_n)^{q^m-1}$ in $\sigma(h)$. If 
    \begin{equation*}
        h = \sum\limits_{\alpha < (q^m-1, \ldots, q^m-1)}c_{\alpha} x^{\alpha} + c_0 \cdot (x_1\cdots x_n)^{q^m-1},
    \end{equation*}
    then for all $\alpha < (q^m-1, \ldots, q^m-1)$, the degree of $x^{\alpha}$ is $|\alpha|< n(q^m-1)$. Since $\sigma$ is a graded map of degree $0$, $\sigma(x^{\alpha})$ does not contribute to the coefficient of $(x_1\cdots x_n)^{q^m-1}$. Moreover, $G$ is generated by the permutations, the diagonal matrices, and the transvection $\mu$ that maps $x_1$ to $x_1+x_2$ and fixes $x_2, \ldots, x_n$. If $\sigma$ is a permutation, then obviously $\sigma((x_1\cdots x_n)^{q^m-1}) = (x_1\cdots x_n)^{q^m-1}$; if $\sigma$ is a diagonal action, we also have $\sigma((x_1\cdots x_n)^{q^m-1}) = (x_1\cdots x_n)^{q^m-1}$; finally, for $\sigma=\mu$, we have
    \begin{align*}
        \mu((x_1\cdots x_n)^{q^m-1}) - (x_1\cdots x_n)^{q^m-1} &= (x_2\cdots x_n)^{q^m-1} ((x_1+x_2)^{q^m-1}-x_1^{q^m-1})\\
        &=x_2^{q^m}\cdot h(x_1, \ldots, x_n)\\
        &=0 \text{ in $Q(n, m)$.}
    \end{align*}
    Therefore, $\sigma((x_1\cdots x_n)^{q^m-1}) = (x_1\cdots x_n)^{q^m-1}$ for all $\sigma\in G$.
\end{proof}
\begin{corollary}[{\cite[Corollary 3.4]{Lewis_Reiner_Stanton_2017}}]\label{cor:poincare_duality}
    Let $d_0 = n(q^m-1)$. For each degree $d\in \{0, \ldots, d_0\}$, let $Q(m, n)_d$ be the subspace of homogeneous elements of degree $d$ in $Q(m, n)$. We have an isomorphism of $G$-representations
    \begin{equation*}
        (Q(m, n)_d)^{*} \cong Q(m, n)_{d_0-d},
    \end{equation*}
    where the action of $G$ on $(Q(m, n)_d)^{*}$ is the \emph{contragredient} action.
\end{corollary}
\begin{proof}
    We prove that the map
    \begin{align*}
        \psi: Q(m, n)_{d_0-d} &\to (Q(m, n)_{d})^{*}\\
        f&\mapsto B(f, \cdot)
    \end{align*}
    is a $G$-isomorphism. First, it is an isomorphism of $\mathbb{F}_q$-vector space, since it is linear, and one can check that $f(\{x^{\alpha}\mid |\alpha| = d_0-d\})$ is the dual basis of the basis $\{x^{\beta}\mid |\beta| =d\}$ of $Q(m, n)_d$. For any $\sigma\in G$, $f\in Q(m, n)_{d_0-d}$, and $g\in Q(m, n)_d$, we have
    \begin{equation*}
        \psi(\sigma(f))(g) = B(\sigma(f), g) = B(f, \sigma^{-1}(g)) = \psi(f)(\sigma^{-1}(g)) = (\sigma\cdot \psi(f))(g).
    \end{equation*}
    Therefore, $\psi$ is a $G$-map, as required.
\end{proof}
To conclude this section, we present the notion of \emph{graded} and \emph{ungraded parking spaces} from \cite[Section 6]{Lewis_Reiner_Stanton_2017}, which are $G$-representations that are motivated by the representation theory of Coxeter groups. 
\begin{definition}[{\cite[Definition 6.1]{Lewis_Reiner_Stanton_2017}}]\label{def:graded_parking_spaces}\index{graded parking space}
    For a field $\mathbb{K}$ containing $\mathbb{F}_q$, the graded parking space for $G = GL_n(\mathbb{F}_q)$ over $\mathbb{K}$ is 
    \begin{equation*}
        Q_{\mathbb{K}}(m, n) = \mathbb{K}[x_1, \ldots, x_n]/(x_1^{q^m}, \ldots, x_n^{q^m}).
    \end{equation*}
    $G$ is considered a subgroup of $GL_n(\mathbb{K})$, acting on $\mathbb{K}[x_1, \ldots, x_n]$ by linear substitutions. 
\end{definition}
In the following results, it is required that $m>0$.
\begin{definition}[{\cite[Definition 6.2]{Lewis_Reiner_Stanton_2017}}]\label{def:ungraded_parking_spaces}\index{ungraded parking spaces}
    For a field $\mathbb{K}$ containing $\mathbb{F}_q$, the ungraded parking space for $G = GL_n(\mathbb{F}_q)$ over $\mathbb{K}$ is 
    \begin{equation*}
        \mathbb{K}[\mathbb{F}_{q^m}^{n}] = span_{\mathbb{K}}\{e_v\mid v\in \mathbb{F}_{q^m}^{n}\},
    \end{equation*}
    $G$ is considered a subgroup of $GL_n(\mathbb{F}_{q^m}^{n})$, permuting the vectors $e_v$. In particular, for any $\sigma\in G$, $\sigma(e_v) = e_{\sigma(v)}, \text{ for all } v\in \mathbb{F}_{q^m}^{n}$.
\end{definition}
\begin{remark}
    The graded and the ungraded parking space over a field $\mathbb{K}$ are not in general isomorphic as $G$-representations (see \cite[Example 6.4]{Lewis_Reiner_Stanton_2017}.)
\end{remark}
\begin{proposition}[{\cite[Proposition 6.9]{Lewis_Reiner_Stanton_2017}}]\label{prop:ungraded_is_iso_to_R_k}
    Fix a field $\mathbb{K}$ containing $\mathbb{F}_{q^m}$. Let $\mathfrak{n} = (x_1^{q^m}-x_1, \ldots, x_n^{q^m}-x_n)$ be an ideal of $\mathbb{K}[x_1, \ldots, x_n]$, define
    \begin{equation*}
        R_{\mathbb{K}}(m, n) = \mathbb{K}[x_1, \ldots, x_n]/\mathfrak{n}.
    \end{equation*}
    Then $R(m, n)\cong \mathbb{K}[\mathbb{F}_{q^m}^{n}]$ as $G$-representations.
\end{proposition}
\begin{proof}[Sketch of the proof]
    First, it is not hard to see that $R_{\mathbb{K}}(m, n)$ is indeed a $G$-representation, since $\mathfrak{n}$ is stable under the action of $G$. Consider the "evaluation" map 
    \begin{align*}
        \epsilon: \mathbb{K}[x_1, \ldots, x_n]&\to \mathbb{K}[\mathbb{F}_{q^m}^{n}]\\
        f(x_1, \ldots, x_n) &\mapsto \sum\limits_{v\in \mathbb{F}_{q^m}^{n}} f(v)\cdot e_v,
    \end{align*}
    which is a surjective $G$-map with kernel $\mathfrak{n}$ (since for any $a\in \mathbb{F}_{q^m}$, $a^{q^m}-a=0$.) Therefore, $R_{\mathbb{K}}(m, n)$ and $\mathbb{K}[\mathbb{F}_{q^m}^{n}]$ are isomorphic $\mathbb{K}G$-modules.
\end{proof}
The ring $R_{\mathbb{K}}(m, n)$ admits a filtration
\begin{equation*}
    F_0 \subseteq F_1 \subseteq \cdots \subseteq R_{\mathbb{K}}(m, n),
\end{equation*}
where $F_i$ is the image of the polynomials in $\mathbb{K}[x_1, \ldots, x_n]$ of degree at most $i$. Define the associated graded ring\index{associated graded ring}
\begin{equation*}
    \mathfrak{gr}R_{\mathbb{K}}(m, n) = F_0 \oplus \bigoplus\limits_{i=1}^{+\infty}F_{i}/F_{i-1},
\end{equation*}
with multiplication $F_{i}/F_{i-1}\times F_{j}/F_{j-1}\to F_{i+j}/F_{i+j-1}$ induced from the standard multiplication $F_{i}\times F_{j}\to F_{i+j}$.
\begin{proposition}[{\cite[Proposition 6.10]{Lewis_Reiner_Stanton_2017}}]\label{prop:grR_and_Q}
    When $\mathbb{K}$ contains $\mathbb{F}_{q^m}$, the rings $Q_{\mathbb{K}}(m, n)$ and $\mathfrak{gr}R_{\mathbb{K}}(m, n)$ are isomorphic as graded $\mathbb{K}G$-modules.
\end{proposition}
\begin{proof}[Sketch of the proof]
    Consider the natural ring map
    \begin{align*}
        \phi: \mathbb{K}[x_1, \ldots, x_n] &\to \mathfrak{gr}R_{\mathbb{K}}(m, n)\\ 
        x_i &\mapsto [x_i] \in F_1/F_0.
    \end{align*}
    This map is surjective since $\{[x_1], \ldots, [x_n]\}$ generates $\mathfrak{gr}R_{\mathbb{K}}(m, n)$ as an algebra. The equality $x_i^{q^m} = x_i$ in $R_{\mathbb{K}}(m, n)$ implies that $[x_i]^{q^m} = [x_i]\in F_{q^m}/F_{q^m-1}$. As $m>0$, $q^m-1\geq 1$, which implies that $[x_i]\in F_{q^m-1}$, thus $[x_i]^{q^m} = 0$. Therefore, $\phi$ descends to a $G$-equivariant surjective ring map $\phi: Q_{\mathbb{K}}(m, n)\to \mathfrak{gr}R_{\mathbb{K}}(m, n)$. Observe that both spaces have dimension $q^{mn}$ as $\mathbb{K}$-vector spaces, hence $\phi$ is an isomorphism.
\end{proof}
We say that two $\mathbb{K}G$-modules $M_1$ and $M_2$ are \emph{Brauer isomorphic} if they have the same list of composition factors.
\begin{corollary}[{\cite[Corollary 6.11]{Lewis_Reiner_Stanton_2017}}]\label{cor:brauer_iso}
    When $\mathbb{K}$ contains $\mathbb{F}_q$, one has a Brauer isomorphism of $\mathbb{K}G$-modules
    \begin{equation*}
        Q_{\mathbb{K}}(m, n) \approx \mathbb{K}[\mathbb{F}_{q^m}^{n}].
    \end{equation*}
\end{corollary}
\begin{proof}[Sketch of the proof]
    Without loss of generality, assume that $\mathbb{K}$ contains $\mathbb{F}_{q^m}$ (since two $\mathbb{K}G$-modules are Brauer isomorphic if and only if they are Brauer isomorphic after extending scalars.) It is not hard to see that $R_{\mathbb{K}}(m, n)$ and $\mathfrak{gr}R_{\mathbb{K}}(m, n)$ have the same composition series, so by \cref{prop:ungraded_is_iso_to_R_k} and \cref{prop:grR_and_Q}, we have a sequence of Brauer isomorphic $\mathbb{K}G$-modules
    \begin{equation*}
        \mathbb{K}[\mathbb{F}_{q^m}^{n}]\cong R_{\mathbb{K}}(m, n) \approx \mathfrak{gr}R_{\mathbb{K}}(m, n) \cong Q_{\mathbb{K}}(m, n).\qedhere
    \end{equation*}
\end{proof}
The claims in \cref{cor:brauer_iso} are clarified in \Cref{appendix:brauer_character}, where we introduce the character theory of modular representations.
\section{Cofixed Spaces}\label{sec:cofixed_spaces}
Let $S = \mathbb{F}_q[x_1, \ldots, x_n]$ and $G$ be a subgroup of $GL_n(\mathbb{F}_q)$ acting on $S$. In this section, we investigate the \emph{cofixed space} with respect to this action, which is the largest fixed quotient of $S$. 
\begin{definition}\label{def:cofixed_space}\index{cofixed space}
    Let $V$ be an $\mathbb{F}_q G$-module, and $V_{(G)}$ be the $\mathbb{F}_q$-vector space generated by the set
    \begin{equation*}
        \{\sigma(v) - v\mid \sigma\in G, v\in V\}.
    \end{equation*}
    The cofixed space of $V$ with respect to $G$ is the quotient space
    \begin{equation*}
        V_G = V/V_{(G)}.
    \end{equation*}
\end{definition}
\begin{remark}
\begin{enumerate}
    \item 
Consider the cofixed space $S_G$, it is not only an $\mathbb{F}_q$-vector space, but also an $S^G$-module, since $M_G$ is stable under multiplication by a $G$-invariant polynomial.
\item It is known \cite[Proposition 5.7]{Lewis_Reiner_Stanton_2017} that $S_G$ is an $S^G$-module of rank $1$, that is, $Frac(S^G)\otimes_{S^G}S_G$ is a one-dimensional $Frac(S^G)$-vector space.
    \item For $G = GL_n(\mathbb{F}_q)$ (or $G$ is a parabolic subgroup of $GL_n(\mathbb{F}_q)$), it is an open question to determine an $\mathbb{F}_q$-basis of $S_G$, or furthermore, the $S^G$-module structure of $S_G$. In \cite[Appendix]{Lewis_Reiner_Stanton_2017}, the bivariate case ($n=2$) of this question has been investigated.
\end{enumerate}
\end{remark}
We give the proof of a structural property of $S_G$, which shows that certain elements of $S^G$ are non-zero divisors of $S_G$.
\begin{proposition}[{\cite[Lemma 6.3]{karagueguzian_symonds_examples}}]\label{prop:injectivity_cofix}
    Consider a homogeneous element $u\in S^G$ of degree $d$, such that the coefficient of $x_n^d$ in $u$ is nonzero. Then the multiplication map
    \begin{align*}
        \mu_u: S_G &\to S_G\\
        [f]&\mapsto [uf]
    \end{align*}
    is injective.
\end{proposition}
\begin{proof}
    Let $U_0$ be a Sylow $p$-subgroup of $G$, then since $u\in S^G$, we can assume without loss of generality that $U_0$ is a subgroup of the unipotent group $U\subseteq GL_n(\mathbb{F}_q)$. Define
    \begin{equation*}
        S[d-1] = span_{\mathbb{F}_q}\{x^{\alpha}\mid \alpha_n\leq d-1\}.
    \end{equation*}
    Then, $S[d-1]$ is an $\mathbb{F}_q U$-submodule of $S$. Furthermore, we show that there is a decomposition of $S$ as a direct sum of $\mathbb{F}_q U$-modules 
    \begin{equation*}
        S = uS\oplus S[d-1].
    \end{equation*}
    Obviously, $uS$ and $S[d-1]$ are $\mathbb{F}_q U$-modules; because $u\notin S[d-1]$, for any nonzero $f$, $uf\notin S[d-1]$ by direct calculation. Moreover, $u$ is a monic polynomial with respect to $x_n$, hence any polynomial $f$ can be written as
    \begin{equation}\label{eq:U_decom}
        f = f_1 \cdot u + f_2, 
    \end{equation}
    where $f_1, f_2\in S$, and $f_2$ is zero or is of $x_n$-degree at most $(d-1)$. Consequently, \eqref{eq:U_decom} holds. Since $[G:U]$ is invertible in $\mathbb{F}_q$, by a generalization of Maschke's theorem, we deduce that there exists a $\mathbb{F}_q G$-submodule $S_0$ of $S$, such that $S = uS \oplus S_0$. Let $\pi: S\to uS$ be the projection with respect to this direct sum decomposition; it is obvious that $\pi$ is $G$-equivariant. Now, suppose that for some $f\in S$, $uf\in S_{(G)}$, then we have
    \begin{equation*}
        uf = \sum\limits_{\sigma\in G}(\sigma(f_{\sigma}) - f_{\sigma}),
    \end{equation*}
    where $f_{\sigma}\in S, \text{ for all } \sigma\in G$. Applying $\pi$ to both sides, note that $uf\in uS$, we have
    \begin{align*}
        uf = \pi(uf) &= \sum\limits_{\sigma\in G}\pi(\sigma(f_{\sigma}) - f_{\sigma})\\
        &=\sum\limits_{\sigma\in G}(\sigma(\pi(f_{\sigma})) - \pi(f_{\sigma})).\\
    \end{align*}
    Since $\pi(f_{\sigma})\in uS$ for all $\sigma\in G$, dividing both sides of the above equation by $u$, we conclude that $f\in S_{(G)}$. 
\end{proof}
This result is useful when one considers the module structure of $S_G$ as an $S^G$-module; for example, when $G = GL_n(\mathbb{F}_q)$, the Dickson invariant $Q_{n, n-1}$ satisfies the condition of \cref{prop:injectivity_cofix}.\smallbreak
To conclude this section, we show the connection of the cofixed space $S_G$ with the invariant spaces of truncated polynomial rings (see \Cref{sec:truncated_rings}.)
\begin{proposition}[{\cite[Corollary 3.5]{Lewis_Reiner_Stanton_2017}}]\label{prop:dual_hilbert_series}
    For any subgroup $G\subseteq GL_n(\mathbb{F}_q)$, we have
    \begin{equation}\label{eq:limit_hilbert}
        Hilb(S_G, t) = \lim\limits_{m\to +\infty}t^{n(q^m-1)}Hilb(Q(m, n)^G, t^{-1}).
    \end{equation}
    The limit at the right-hand side of \eqref{eq:limit_hilbert} is understood as follows. For each degree $d \in \mathbb{Z}$, the coefficient of $t^d$ in the polynomial $t^{n(q^m-1)}Hilb(Q(m, n)^G, t^{-1})$ is equal to the coefficient of $t^d$ in the formal power series $Hilb(S_G, t)$ for all large enough $m$.
\end{proposition}
\begin{proof}
    Fix a degree $d\geq 0$, if $d\leq q^m-1$, since $I_{n, m}$ does not contain an element of degree at most $d$, we have the obvious isomorphism of $G$-representations
    \begin{equation*}
        S_d \cong Q(m, n)_d,
    \end{equation*}
    which induces an isomorphism of cofixed spaces $(S_G)_d \cong (Q(m, n)_G)_d$. Therefore,
    \begin{equation*}
        Hilb(S_G, t) = \lim\limits_{m\to +\infty}Hilb(Q(m, n)_G, t).
    \end{equation*}
    It suffices to prove that
    \begin{equation}\label{eq:hilb_dual}
        t^{n(q^m-1)}Hilb(Q(m, n)^G, t^{-1}) = Hilb(Q(m, n)_G, t).
    \end{equation}
    This follows from \cref{cor:poincare_duality}; indeed, for any $G$-module $V$, one has an isomorphism of vector spaces
    \begin{equation*}
        (V_G)^{*} \cong (V^{*})^G.
    \end{equation*}
    By the universal property of quotient spaces, a functional $f\in V^{*}$ descends to $(V_G)^{*}$ if and only if $f(\sigma v) = f(v), \text{ for all } \sigma\in G, v\in V$, which is equivalent to the claim that $f\in (V^{*})^G$. Hence, the two spaces given above are isomorphic. By \cref{cor:poincare_duality}, since $Q(m, n)_d$ is finite-dimensional, there is an isomorphism of $\mathbb{F}_q$-vector spaces
    \begin{equation*}
        (Q(m, n)_G)_d \cong ((Q(m, n)_d)_G)^{*}\cong ((Q(m, n)_d)^{*})^G \cong Q(m, n)_{n(q^m-1)-d}^G, 
    \end{equation*}
    from which \eqref{eq:hilb_dual} follows.
\end{proof}
\chapter{Invariants of Truncated Polynomial Rings}\label{chap:trucated_invariants}
As shown in \Cref{sec:truncated_rings} and \Cref{sec:cofixed_spaces}, determining an $\mathbb{F}_q$-basis for the invariant rings $Q(m, n)^G$ yields information about the cofixed space $S_B$. In this chapter, we present an explicit construction of bases for invariant spaces of truncated polynomial algebras, originating from the papers \cite{Ha_Hai_Nghia_2024, ha_hai_nghia_GL, ha_hai_nghia_parabolic_for_low_rank}. Following the novel construction of bases for these invariant rings, new problems arise; we examine them in more detail in subsequent sections.   
\section{Lewis-Reiner-Stanton Conjecture}\label{sec:the_conj}
\subsection{Invariants Under the Action of Borel Subgroup}
First, we recall \cref{conj:parabolic_conj}. 
\begingroup
\def\thetheorem{\ref{conj:parabolic_conj}}
\begin{conjecture}[{\cite[Parabolic conjecture 1.5]{Lewis_Reiner_Stanton_2017}}]
    For $m\geq 0$ and $\alpha$ a composition of $n$, the invariant ring $Q(m, n)^{P_{\alpha}}$ has Hilbert series
    \begin{equation*}
        Hilb(Q(m, n)^{P_{\alpha}}, t) = \sum\limits_{\substack{\beta\leq \alpha\\ \beta\leq m}}t^{e(m, \alpha, \beta)}\genfrac{[}{]}{0pt}{}{m}{\beta, m-|\beta|}_{q, t},
    \end{equation*}
    where 
    \begin{equation*}
        e(m, \alpha, \beta) = \sum\limits_{i=1}^{\ell}(\alpha_i - \beta_i)(q^m-q^{B_i}).
    \end{equation*}
\end{conjecture}
\addtocounter{theorem}{-1}
\endgroup
Following \cite{Ha_Hai_Nghia_2024, ha_hai_nghia_GL, ha_hai_nghia_parabolic_for_low_rank}, we define the \emph{delta operators}, which is a crucial operator used in the construction of bases for invariant rings of truncated algebras.
\begin{definition}[{\cite[Definition 1.2]{Ha_Hai_Nghia_2024}}]\label{def:delta_operator}\index{delta operators}
    Let $a, b, c$ be nonnegative integers, such that $1\leq a\leq c+1$. Define an operator
    \begin{equation*}
        \delta_{a;b}:\mathbb{F}_q(x_1, \ldots, x_c)\to \mathbb{F}_q(x_1, \ldots, x_{c+1})
    \end{equation*}
    from the ring of rational functions in $c$ variables to the ring of rational functions in $(c+1)$ variables as follows: If $f\in \mathbb{F}_q(x_1, \ldots, x_c)$ then $\delta_{a;b}(f)$ is defined as the quotient
    \begin{equation*}
        \delta_{a;b}(f)=\dfrac{\begin{vmatrix}
            x_1 & \cdots & x_a\\
            x_1^q & \cdots & x_a^q\\
            \vdots & \ddots & \vdots\\
            x_1^{q^{a-2}} & \cdots  & x_a^{q^{a-2}}\\
            x_1^{q^b}f(\widehat{x_1}, x_2, \ldots, x_{c+1}) & \cdots  & x_a^{q^b}f(x_1, \ldots, \widehat{x_a}, \ldots, x_{c+1})
        \end{vmatrix}}{\begin{vmatrix}
            x_1 & \cdots & x_a\\
            x_1^q & \cdots & x_a^q\\
            \vdots & \ddots & \vdots\\
            x_1^{q^{a-2}} & \cdots  & x_a^{q^{a-2}}\\
            x_1^{q^{a-1}} & \cdots  & x_a^{q^{a-1}}\\
        \end{vmatrix}}.
    \end{equation*}
    Here, the hat signifies that the corresponding entry is omitted.
\end{definition}
\begin{remark}
    \begin{enumerate}
        \item In particular, one can see that for $b-a+1$ and $f\equiv 1$, 
\begin{equation*}
    \delta_{a;b}(1)=H_{b-a+1}(x_1, \ldots, x_a),
\end{equation*}
the right-hand side of which is a Schur polynomial in $\mathcal{D}_n$ (see \cref{subsection:schur_7}.
\item In subsequent calculations, we use the notation
\begin{equation*}
    L(x_1, \ldots, x_n) =\begin{vmatrix}
            x_1 & \cdots & x_a\\
            x_1^q & \cdots & x_a^q\\
            \vdots & \ddots & \vdots\\
            x_1^{q^{a-2}} & \cdots  & x_a^{q^{a-2}}\\
            x_1^{q^{a-1}} & \cdots  & x_a^{q^{a-1}}\\
        \end{vmatrix} 
\end{equation*}
for the denominator in the definition of $\delta$ operators.
\end{enumerate}
\end{remark}
\begin{definition}[{\cite[Definitions 1.3]{Ha_Hai_Nghia_2024}}]
    For each $a$, let $D_a = \delta_{a;a}(1)$. For two sequences $I = (i_1, \ldots, i_k)$ and $J = (j_1, \ldots, j_k)$ of nonnegative integers, define the rational function
    \begin{equation*}
        Y_b(I;J) = \delta_{1;b}^{i_1}(D_1^{j_1}\delta_{2;b}^{j_2}(D_2^{j_2}(\cdots(\delta_{k;b}(D_k^{j_k}))\cdots))).
    \end{equation*}
    Define the Frobenius-like operator\index{Frobenius-like operator!for Borel invariants} $\Phi$ by $\Phi Y_b(I; J) = Y_{b+1}(0, I; 0; J)$.
\end{definition}
An explicit basis for $Q(m, n)^B$, where $B$ is the Borel subgroup of $GL_n(\mathbb{F}_q)$, is given in the following theorem.
\begin{theorem}[{\cite[Theorem 1.6]{Ha_Hai_Nghia_2024}}]\label{thm:hhn_borel}
    For $m\geq 0$, $n\geq 1$, define
    \begin{equation*}
        \mathcal{B}_m(n) = \bigsqcup\limits_{k=1}^{\min(n, m+1)}\mathcal{B}_m^{k}(n),
    \end{equation*}
    where $\mathcal{B}_m^{k}(n)$ denotes the set consisting of all elements $Y_m(I; J)$ for which the sequences $I = (i_1, \ldots, i_k)$ and $J= (j_1, \ldots, j_k)$ satisfy the following conditions
    \begin{equation*}
        \begin{cases}
            i_1+\cdots + i_k = n-k,\\
            j_1 <\dfrac{q^m-1}{q-1}, \ldots, j_{k-1}<\dfrac{q^{m-k+2}-1}{q-1}, j_k\leq \dfrac{q^{m-k+1}-1}{q-1}.
        \end{cases}
    \end{equation*}
    Then $\mathcal{B}_m(n)$ is an $\mathbb{F}_q$-basis for $Q(m, n)^B$. 
\end{theorem}
\begin{remark}
It is a simple counting problem that \cref{thm:hhn_borel} implies \cref{conj:parabolic_conj} in the Borel case. Moreover, in \cite{ha_hai_nghia_parabolic_for_low_rank}, the authors conjectured a refinement of \cref{conj:parabolic_conj} by proposing an explicit basis for the invariant ring $Q(m, n)^{P_{\alpha}}$, using the delta operators. The refined conjecture has been verified for $\alpha = (1, \ldots, 1)$ in \cite{Ha_Hai_Nghia_2024}, for $\alpha = (n)$ in \cite{ha_hai_nghia_GL}, and for $n\leq 3$ in \cite{ha_hai_nghia_parabolic_for_low_rank}.
\end{remark}
\subsection{Invariants Under the Action of the Unipotent Group}
Following the ideas in the proof of \cite[Theorem 1.6]{Ha_Hai_Nghia_2024}, we describe an $\mathbb{F}_q$-basis for the invariant ring $Q(m, n)^U$ in the case $q$ is prime, where $U\subseteq Gl_n(\mathbb{F}_q)$ is the group of upper triangular matrices with $1$'s on the diagonal. To begin with, let us examine the bivariate case.
\begin{example}
    When $n=2$, consider $U$ acting on $S = \mathbb{F}_q[x, y]$, with $q$ a power of a prime $p$. By \cref{thm:mui_invariants}, we know that $S^U = \mathbb{F}_q[V_1, V_2]$, where
    \begin{equation*}
        V_1 = x, V_2 = y^q - yx^{q-1}.
    \end{equation*}
    An explicit basis for $Q(m, 2)^U$ consists of two families of polynomials
    \begin{itemize}
    \item[(1)] $x^i y^j$, for $0\leq i\leq q^m-1$, $0\leq j\leq q^m-1$, and $i + p^{v_p(j)}\geq q^m$, and $q\nmid j$.
        \item[(2)]$V_1^a V_2^b$, for $0\leq a\leq q^m-1$, $0\leq b\leq q^{m-1}-1$.
    \end{itemize}
    The second family belongs to $S^U$, hence it automatically belongs to $Q(m, 2)^U$. For the first family, consider $f(x, y) = x^i y^j$, where $i + p^{v_p(j)}\geq q^m$, then for any $a\in \mathbb{F}_q$, we have
    \begin{equation*}
        f(x, y+ax) - f(x, y) = a^{p^{v_p(j)}}x^{i + p^{v_p(j)}} g(x, y),
    \end{equation*}
    for some $g\in S$. This means that $f(x, y+ax) = f(x, y)$ in $Q(m, 2)$, or $f\in Q(m, 2)^U$.\smallbreak
    Consider the graded lexicographic ordering on the monomials of $S$ by asserting that $x<y$. It is not hard to see that the leading monomials of the two given families of polynomials are in $Q(m, n)$, and there are no two distinct elements in the above set of polynomials having the same leading monomial, therefore, the given set of polynomials is linearly independent. Furthermore, if $f(x, y)\in Q(m, 2)^{U}$ is a homogeneous element, then we can write 
    \begin{equation*}
        f(x, y) =a_0 x^i y^j + x^{i+1} g(x, y),
    \end{equation*} 
    for some $0\leq i\leq q^m-1$, $g(x, y)\in \mathbb{F}_q[x, y]$. If $i + p^{v_p(j)}\geq q^m$, then we can proceed by induction, as the leading monomial of $f(x, y) - a_0 x^i y^j\in Q(2, m)^{U}$ is strictly smaller than $x^i y^j$. If $i + p^{v_p(j)}< q^m$, we show that $j$ must be divisible by $q$, hence the leading monomial of $f(x, y) - a_0 x^i V_2^{\frac{j}{q}}\in Q(2, m)^{U}$ is strictly smaller than $x^i y^j$, completing the inductive argument. Indeed, let $r = v_p(j)$, then for all $a\in \mathbb{F}_q$, we have
    \begin{equation*}
        f(x, y+ax)-f(x, y) = a_0 x^i ((y+ax)^j - y^j) + x^{i+1} (g(x, y+ax) - g(x, y)).
    \end{equation*}
    By Lucas's theorem, $p$ does not divide $\binom{j}{p^r}$, hence the term $a_0 a^{p^r}\binom{j}{p^r}x^{i+p^r}y^{j-p^r}$ is nonzero. Because $i+p^r<q^m, j-p^r<q^m$, this summand must appear in $x^{i+1} (g(x, y+ax) - g(x, y))$. However, $x^{i+1} (g(x, y+ax) - g(x, y))$ is a sum of temrs of the form $a_{i', j'} \cdot a^{s} \cdot \binom{j'}{s}x^{i'+s}y^{j'-s}$, for some $i'\geq i+1$, $s>0$. This term is similar to $a_0 a^{p^r}\binom{j}{p^r}x^{i+p^r}y^{j-p^r}$ (as monomials) if and only if $j - p^r = j'-s < j -s$, thus $s<p^r$. Collecting all terms that are similar to $a_0 a^{p^r}\binom{j}{p^r}x^{i+p^r}y^{j-p^r}$ and compute the corresponding coefficient, we get a polynomial equation (with $a$ as the unknown) of degree $p^r$ that equals $0$ for every $a\in \mathbb{F}_q$. This implies that $q|p^r$, as desired.
\end{example}
\begin{remark}
    By the above example, let $q=p^k$, we can compute the Hilbert series of $Q(m, 2)^{U}$ as follows.
    \begin{equation*}
    Hilb(Q(m, 2)^{U}, t)=\left(\sum\limits_{\substack{l=0\\ v_p(j) = l\\0\leq j<q^m}}^{k-1}\sum\limits_{i = q^m - p^l}^{q^m-1}t^{i+j}\right) + \sum\limits_{i=0}^{q^m-1}\sum\limits_{j=0}^{q^{m-1}-1}t^{i+qj}.
\end{equation*} 
The second summand equals
\begin{equation*}
    \left(\sum\limits_{i=0}^{q^m-1}t^i\right)\cdot \left(\sum\limits_{j=0}^{q^{m-1}-1}t^{qj}\right)=\dfrac{1-t^{q^m}}{1-t}\cdot \dfrac{1-t^{q^m}}{1-t^q}.
\end{equation*}
The first summand equals
\begin{equation*}
    \sum\limits_{\substack{l=0\\ v_p(j) = l\\\\0\leq j<q^m}}^{k-1}t^j\cdot \dfrac{t^{q^m-p^l}-t^{q^m}}{1-t}.
\end{equation*}
For each fixed $l$, we have
\begin{align*}
    \sum\limits_{\substack{0\leq j\leq q^m-1\\ v_p(j) = l}}t^j &= \sum\limits_{a = 0}^{p^{km-l}-1}t^{p^l a} - \sum\limits_{b = 0}^{p^{km-l-1}-1}t^{p^{l+1} b}\\
    &=\dfrac{1-t^{p^{km}}}{1-t^{p^l}} - \dfrac{1 - t^{p^{km}}}{1-t^{p^{l+1}}}.
\end{align*}
Substituting this into the first summand above, we have
\begin{align*}
    \sum\limits_{\substack{l=0\\ v_p(j) = l}}^{k-1}t^j\cdot \dfrac{t^{q^m-p^l}-t^{q^m}}{1-t} & =\sum\limits_{l=0}^{k-1}\left(\dfrac{1-t^{q^{m}}}{1-t^{p^l}} - \dfrac{1 - t^{q^{m}}}{1-t^{p^{l+1}}}\right)\cdot \dfrac{t^{q^m-p^l}-t^{q^m}}{1-t}\\
    &=\sum\limits_{l=0}^{k-1}\dfrac{t^{q^m-p^l}(1-t^{q^m})}{1-t}-\dfrac{1-t^{q^m}}{1-t}\cdot \dfrac{t^{q^m-p^l}-t^{q^m}}{1-t^{p^{l+1}}}\\
    &=\dfrac{1-t^{q^m}}{1-t}\sum\limits_{l=0}^{k-1}\dfrac{t^{q^m}-t^{q^m-p^l + p^{l+1}}}{1-t^{p^{l+1}}}.
\end{align*}
Therefore,
\begin{equation*}
    Hilb(Q(m, 2)^{U}, t) = \dfrac{t^{q^m}(1-t^{q^m})}{1-t}\sum\limits_{l=0}^{k-1}\dfrac{1-t^{p^{l+1}-p^l}}{1-t^{p^{l+1}}}+\dfrac{1-t^{q^m}}{1-t}\cdot \dfrac{1-t^{q^m}}{1-t^q}.
\end{equation*}
Using \cref{prop:dual_hilbert_series}, one derives the formula for the Hilbert series for the cofixed space $S_U$,
\begin{equation*}
    Hilb(S_U, t) = -\dfrac{1}{t-1}\cdot \sum\limits_{l=0}^{k-1}\dfrac{t^{p^{l+1}-1}-t^{p^l}-1}{t^{p^{l+1}}-1}+ \dfrac{t^{q-1}}{(t-1)(t^{q}-1)}.
\end{equation*}
The given Hilbert series depend on the prime $p$, which are different from the conjectural Hilbert series of $Q(m, n)^{P_{\alpha}}$ and $S_{P^{\alpha}}$, as they only depend on $q$. One can directly see this phenomenon by proving that for any $q$, the element $x^{p-1}y^{p-1}$ is nonzero in $S_U$, but is zero in $S_B$ if $q>p$.
\end{remark}
Now, assume that $q=p$ is a prime number. Recall from \cref{thm:mui_invariants} that $\mathbb{F}_q[x_1,\ldots, x_n]^U = \mathbb{F}_q[V_1, \ldots, V_n]$.
\begin{definition}\label{def:frob_operator}
    For two sequences $I = (i_1, \ldots, i_k)$ and $J = (j_1, \ldots, j_k)$ of nonnegative integers, define the rational function
    \begin{equation*}
        Y_b(I;J) = \delta_{1, b}^{i_1}(V_1^{j_1}\delta_{2, b}^{i_2}(V_2^{j_2}(\cdots (\delta_{k, b}^{i_1}(V_k^{j_k}))\cdots ))).
    \end{equation*}
    Define a Frobenius-like operator\index{Frobenius-like operator!for unipotent invariants} for the above class of rational functions as follows.
    \begin{equation*}
        \Phi(Y_b(I; J)) = Y_{b+1}((0, I); (0, J)).
    \end{equation*}
\end{definition}
\begin{definition}\label{def:basis_unipotent}
    For $n\geq 1$, $m\geq 0$, define $\mathcal{B}_m(n)$ inductively as follows. $\mathcal{B}_0(n) = \{1\}, \text{ for all } n\geq 1$; $\mathcal{B}_m(1) = \{V_1^a \mid 0\leq a\leq p^m-1\}$, for all $m\geq 0$. If $n\geq 2$ and $m\geq 1$, define $\mathcal{B}_m(n)$ as the union of two families of rational functions
    \begin{equation*}
        \mathcal{B}_m(n) = \{\delta_{1; m}(Y)\mid Y\in \mathcal{B}_m(n-1)\}\sqcup \{V_1^a\cdot \Phi(Y)\mid 0\leq a<p^m-1, Y\in \mathcal{B}_{m-1}(n-1)\}.
    \end{equation*}
    Explicitly, one can describe the sets $\mathcal{B}_m(n)$ as follows. $\mathcal{B}_m(n)$ is the disjoint union $\bigsqcup\limits_{k=1}^{\min{(n, m+1)}}\mathcal{B}_m^k(n)$, where $\mathcal{B}_m^k(n)$ denotes the set of all elements $Y_m(I; J)$ for which the sequences $I = (i_1, \ldots, i_k)$, $J = (j_1, \ldots, j_k)$ are such that $i_1+\cdots + i_k = n-k$, and $0\leq j_1<p^m-1, \ldots, 0\leq j_{k-1}<p^{m-k+2}-1, 0\leq j_k\leq p^{m-k+1}-1$.
\end{definition}
\begin{theorem}
    \label{thm: q=p_case}
    $B_m(n)$ is an $\mathbb{F}_q$-basis for the invariant ring $Q(m, n)^{U}$.
\end{theorem}
First and foremost, we show that $\mathcal{B}_m(n)$ is a linearly independent set of polynomials in $Q(m, n)^U$. This is done by the following series of lemmas.
\begin{lemma}[{\cite[Propostition 8.4]{Ha_Hai_Nghia_2024}}]\label{lemma:chain_complex}
    Consider two delta operators
    \begin{align*}
        \delta_{a;b}:\mathbb{F}_q(x_1, \ldots, x_c)&\to \mathbb{F}_q(x_1, \ldots, x_{c+1}),\\
        \delta_{a+1;b}:\mathbb{F}_q(x_1, \ldots, x_{c+1})&\to \mathbb{F}_q(x_1, \ldots, x_{c+2}).
    \end{align*}
    Then, $\delta_{a+1;b}\delta_{a;b}$ is the zero map.
\end{lemma}
\begin{proof}
Using Laplace expansion along the last row, we need to show that
    \begin{equation}\label{eq:showing_chain_complex}
        \sum\limits_{i=1}^{a+1}(-1)^{a+1+i}x_i^{q^b}\left(\sum\limits_{\substack{j=1\\ j\neq i}}^{a+1}(-1)^{m(i, j)}x_j^{q^b}L(x_1, \ldots, \widehat{x_i}, \ldots, \widehat{x_j}, \ldots, x_{a+1})\right) = 0,
    \end{equation}
    where $m(i, j) = (-1)^{a+j}$ if $j<i$, and $m(i, j) = (-1)^{a+j+1}$ if $i>j$. Therefore, we can rearrange the left-hand side of \eqref{eq:showing_chain_complex} as follows.
    \begin{equation*}
        \sum\limits_{1\leq i < j\leq n}L(x_1, \ldots, \widehat{x_i}, \ldots, \widehat{x_j}, \ldots, x_{a+1})(x_i x_j)^{q^b}\cdot ((-1)^{i+j}+(-1)^{i+j+1}) = 0.\qedhere
    \end{equation*}
\end{proof}
\begin{lemma}\label{lemma:functional_eq}
    Consider $f\in \mathbb{F}_q[x_1, \ldots, x_c]$. $\delta_{a;b}(f)$ is a polynomial if and only if the following equation of polynomials holds for any $c_1, \ldots, c_{a-1}\in \mathbb{F}_q$
    \begin{equation}\label{eq:fundamental_functional_eq}
        \sum\limits_{i=1}^{a-1}c_i x_i^{q^b}f(x_1, \ldots, \widehat{x_i},\ldots, x_{a-1}, \sum\limits_{j=1}^{a-1}c_j x_j, \ldots, x_{c+1}) = \left(\sum\limits_{j=1}^{a-1}c_j x_j^{q^b}\right)f(\widehat{x_a}),
    \end{equation}
    where $f(\widehat{x_a})$ is $f(x_1, \ldots, x_{a-1}, x_{a+1}, \ldots, x_{c+1})$.
\end{lemma}
\begin{proof}
    Using the proof of \cref{lemma:f_W}, one sees that the denominator $L(x_1, \ldots, x_a)$ in the definition of $\delta_{a;b}(f)$ is a product of linear combinations of $x_1, \ldots, x_a$; and if $m$ is a nonzero linear combination of $x_1, \ldots, x_a$, then $m\mid L(x_1, \ldots, x_a)$, but $m^2$ is not a divisor of $L(x_1, \ldots, x_a)$. Therefore, $\delta_{a;b}(f)$ is a polynomial if and only if the denominator is divisible by any nonzero linear combination of the variables $x_1, \ldots, x_a$. In particular, for any $c_1, \ldots, c_{a-1}\in \mathbb{F}_q$, let $m = x_a - \sum\limits_{j=1}^{a-1}c_j x_j$, by elementary column transformation, we have 
\begin{equation}\label{eq:det_manip_fundamental_functional_eq}
    \begin{vmatrix}
            x_1 & \cdots & x_a\\
            x_1^q & \cdots & x_a^q\\
            \vdots & \ddots & \vdots\\
            x_1^{q^{a-2}} & \cdots  & x_a^{q^{a-2}}\\
            x_1^{q^b}f(\widehat{x_1}) & \cdots  & x_a^{q^b}f(\widehat{x_a})
        \end{vmatrix} = \begin{vmatrix}
            x_1 & \cdots & m\\
            x_1^q & \cdots & m^q\\
            \vdots & \ddots & \vdots\\
            x_1^{q^{a-2}} & \cdots  & m^{q^{a-2}}\\
            x_1^{q^b}f(\widehat{x_1}) & \cdots  & M
        \end{vmatrix},
\end{equation}
in this equation,
\begin{equation}\label{eq:bezout_functional}
    M=-\sum\limits_{i=1}^{a-1}c_i x_i^{q^b}f(x_1, \ldots, \widehat{x_i},\ldots, x_{c+1}) + x_{a}^{q^b}f(\widehat{x_a}).
\end{equation}
It follows that $M\cdot L(x_1, \ldots, x_{a-1})$ is divisible by $m$; however, $L(x_1, \ldots, x_{a-1})$ is not divisible by $m$, thus $m\mid M$. \eqref{eq:fundamental_functional_eq} expresses the fact that $M=0$ whenever $m=0$.\smallbreak
Conversely, suppose that \eqref{eq:fundamental_functional_eq} holds. For $j<a-1$, letting $c_j = 1$ and $c_i = 0, \text{ for all } i\neq j$, \eqref{eq:fundamental_functional_eq} becomes
\begin{equation*}
    f(x_1, \ldots, \widehat{x_j}, \ldots, x_{a-1}, x_j, x_{a+1}\ldots, x_{c+1}) =f(x_1,\ldots, x_{a-1}, x_{a+1}, \ldots, x_{c+1}).  
\end{equation*}
Therefore, $f$ is symmetric in the first $(a-1)$ variables. Consequently, the numerator of  $\delta_{a;b}(f)$ is anti-symmetric in the first $a$ variables. Furthermore, from \eqref{eq:det_manip_fundamental_functional_eq} and \eqref{eq:bezout_functional}, the numerator of $\delta_{a;b}(f)$ is divisible by the polynomial $\prod\limits_{c_1, \ldots, c_{a-1}\in \mathbb{F}_q}(x_a + \sum\limits_{j=1}^{a-1}c_j x_j)$. Thus, by anti-symmetry of the numerator of $\delta_{a;b}(f)$, it follows that this numerator is divisible by $L(x_1, \ldots, x_a)$, which means that $\delta_{a;b}(f)$ is a polynomial.
\end{proof}
\begin{corollary}\label{cor:lower_delta_is_easier}
    If $f\in \mathbb{F}_q[x_1, \ldots, x_c]$ is a polynomial such that $\delta_{a;b}(f)$ is a polynomial, then for any $a'<a$, $\delta_{a';b}(f)$ is a polynomial.
\end{corollary}
\begin{proof}
    Just observe that when $c_{a'} = \cdots = c_{a-1} = 0$, \eqref{eq:fundamental_functional_eq} for $\delta_{a;b}$ becomes the corresponding \eqref{eq:fundamental_functional_eq} for $\delta_{a';b}$.
\end{proof}
\begin{lemma}[{\cite[Corollary 3.2]{Ha_Hai_Nghia_2024}}]\label{lemma:polynomiality_D_mod}
    If $f\in \mathbb{F}_q[x_1, \ldots, x_c]$ is a polynomial such that $\delta_{a;b}(f)$ is a polynomial, and $g\in \mathbb{F}_q[x_1, \ldots, x_c]$ that is $GL_{a-1}$-invariant in the first $(a-1)$ variables, then $\delta_{a;b}(gf)$ is a polynomial. 
\end{lemma}
\begin{proof}
    Using \cref{lemma:functional_eq}, $f$ satisfies the functional equation \eqref{eq:fundamental_functional_eq} for any $c_1, \ldots, c_{a-1}\in \mathbb{F}_q$. Fixing $c_1, \ldots, c_{a-1}\in \mathbb{F}_q$, for an index $i$ such that $c_i\neq 0$, since $g$ is $GL_{a-1}$-invariant in the first $(a-1)$ variables, we have
    \begin{equation*}
g(x_1, \ldots, \widehat{x_i},\ldots, x_{a-1}, \sum\limits_{j=1}^{a-1}c_j x_j, \ldots, x_{c+1}) = g(x_1, \ldots, x_{a-1}, x_{a+1}, \ldots, x_{c+1}).        
    \end{equation*}
    Therefore, multiplying both sides of \eqref{eq:fundamental_functional_eq} by $g(x_1, \ldots, x_{a-1}, x_{a+1}, \ldots, x_{c+1})$, one sees that $gf$ also satisfies $\eqref{eq:fundamental_functional_eq}$, as desired.
\end{proof}
\begin{lemma}\label{lemma:U_induction}
    Suppose $f\in Q(m, k)^U$, and for some $r\leq k+1$, $g=\delta_{r;m}(f)$ is a polynomial. Then $g\in Q(m, k+1)^U$.
\end{lemma}
\begin{proof}
    We follow the proof of \cite[Proposition 4.3]{Ha_Hai_Nghia_2024}. Let $N$ denote the numerator of $\delta_{r;m}(f)$, which is
    \begin{equation*}
        N = \begin{vmatrix}
            x_1 & \cdots & x_k\\
            \vdots & \ddots & \vdots\\
            x_1^{q^{r-2}} & \cdots & x_r^{q^{r-2}}\\
            x_1^{q^m}f(\widehat{x_1}, x_2, \ldots, x_{k+1}) & \cdots  & x_r^{q^m}f(x_1, \ldots, \widehat{x_r}, \ldots, x_{k+1})
        \end{vmatrix}.
    \end{equation*}
    For any $a\in \mathbb{F}_q^{\times}$, consider the operator $\sigma\in U$ sending $x_j$ to $x_j + ax_i$ for some $1\leq i<j\leq k+1$ and fixing $x_{\ell}, \text{ for all } \ell\neq j$. Because the denominator $L(x_1, \ldots, x_k)$ is invariant under the action of $\sigma$, and $L(x_1, \ldots, x_k)$ is square-free, it suffices to show that $\sigma(N) - N$ is divisible by $x_i^{q^{m}+1}$ if $i\leq r$, and by $x_i^{q^m}$ if $i>r$.
    \begin{itemize}
        \item[(1)] If $i>r$, then $j>i>r$, therefore, the first $(r-1)$ rows of $\sigma(N)$ and $N$ are the same. Furthermore, for all $\ell\leq r$, the $(r, \ell)$-entry of $\sigma(N)-N$ is 
        \begin{equation*}
            x_{\ell}^{q^m}\cdot (f(x_1, \ldots,x_i,\ldots, \widehat{x_{\ell}},\ldots, x_j+ax_i, \ldots)-f(x_1, \ldots,x_i,\ldots, \widehat{x_{\ell}},\ldots, x_j, \ldots)),
        \end{equation*}
        which is divisible by $x_i^{q^m}$ due to the assumption $f\in Q(m, k)^U$.
        \item[(2)] If $i\leq r$, we have two cases. The first case is when $j\leq r$, the $j$th column of $\sigma(N)$ is a sum of two column vectors
            \begin{equation*}
                (x_j, x_j^q, \ldots, x_j^{q-2}, x_j^{q^m}f(\widehat{x_j}))^T + a\cdot (x_i, x_i^q, \ldots, x_i^{q-2}, x_i^{q^m}f(\widehat{x_j}))^T,
            \end{equation*}
            hence, the difference $\sigma(N)-N$ can be written as a sum of two determinants, in which the first have the top $(r-1)$ rows exactly the same as that of $N$, and for any $\ell\in \{1, \ldots, r\}$, the $(r, \ell)$-entry of which is
            \begin{equation*}
                    \delta_{\ell j}x_{\ell}^{q^m}(f(x_1, \ldots, \widehat{x_{\ell}}, \ldots, x_j+x_i, \ldots, x_{k+1}) - f(x_1, \ldots, \widehat{x_{\ell}}, \ldots, x_j, \ldots, x_{k+1})),
            \end{equation*}
            (here $\delta_{\ell r}$ is the Kronecker delta.) If $\ell\neq i$, then by the assumption $f\in Q(m, k)^U$, the $(r, \ell)$-entry is divisible by $x_i^{q^m}$. Furthermore, the $(r, i)$-entry is divisible by $x_i^{q^m+1}$. Since all the elements in the $i$th column are divisible by $x_i$, it follows that the first determinant is divisible by $x_i^{q^m+1}$. The second determinant can be computed explicitly; let 
            \begin{equation*}
                A = f(x_1, \ldots, x_i, \ldots, \widehat{x_j}, \ldots, x_{k+1}) - f(x_1, \ldots, \widehat{x_i}, \ldots, x_{j}+x_i, \ldots, x_{k+1}),
            \end{equation*}
            then the determinant in consideration equals
            \begin{equation*}
                a\cdot x_i^{q^m}L(x_1, \ldots, \widehat{x_j}, \ldots, x_r)\cdot A,
            \end{equation*}
            which is divisible by $x_i^{q^m+1}$ since $x_i\mid L(x_1, \ldots, \widehat{x_j}, \ldots, x_r)$. The second case is when $j>r\geq i$, in this case, only the first determinant occurs. One checks that for $\ell\neq i$, the $(\ell, r)$-entry of this determinant is divisible by $x_i^{q^m}$, and the $(i, r)$-entry of this determinant is divisible by $x_i^{q^m+1}$.\qedhere
        \end{itemize}
\end{proof}
\begin{corollary}\label{cor:Y_is_polynomial}
    $\mathcal{B}_m(n)$ is a linearly independent subset of $Q(m, n)^U$.
\end{corollary}
\begin{proof}
    \cref{cor:lower_delta_is_easier} and \cref{lemma:polynomiality_D_mod} show that $\mathcal{B}_m(n)$ are polynomials. Consider an element $Y_m(I;J) = \delta_{1; m}^{i_1}(V_1^{j_1}\delta_{2; m}^{i_2}(V_2^{j_2}(\cdots (\delta_{k; m}^{i_1}(V_k^{j_k}))\cdots )))$, we have $V_k^{j_k}\in Q(m, k)^{U}$, thus by \cref{lemma:U_induction}, $\delta_{k; m}^{i_1}(V_k^{j_k})\in Q(m, k+i_k)^{U}$. Repeating this argument, we see that $Y_m(I; J)$ is an element of 
    \begin{equation*}
        Q(m, k+i_k + \cdots + i_1)^U = Q(m, n)^U.
    \end{equation*}  
    Consider the graded reversed lexicographic ordering (degrevlex)\index{degrevlex}\index{graded reversed lexicographic ordering} on the monomials of $S$ generated by $x_1<\cdots <x_n$, we prove by induction that the leading monomials of elements in $\mathcal{B}_{m}(n)$ are pairwise distinct nonzero elements of $Q(m, n)$, from which it follows that $\mathcal{B}_m(n)$ is linearly independent. The base case $n=1$ is trivial; assume that the claim is true for $(n-1)$, for any $Y\in \mathcal{B}_m(n-1)$, we have
    \begin{equation*}
        lm(\delta_{1;m}(Y)) = x_1^{q^m-1}lm(Y)(x_2, \ldots, x_n),
    \end{equation*}
    where $lm(f)$ denotes the leading monomial\index{leading monomial} of a nonzero polynomial $f\in S$. Moreover, from the definition of $\delta$ operators, consider a polynomial $f\in \mathbb{F}_q[x_1, \ldots, x_c]$ such that $\delta_{a+1;b+1}(f)$ is a polynomial. Then, we have
    \begin{equation*}
       \delta_{a+1;b+1}(f) = \dfrac{\begin{vmatrix}
            1 & \cdots & x_{a+1}\\
            x_1^{q-1} & \cdots & x_{a+1}^q\\
            \vdots & \ddots & \vdots\\
            x_1^{q^{a-1}-1} & \cdots  & x_{a+1}^{q^{a-1}}\\
            x_1^{q^{b+1}-1}f(\widehat{x_1}, x_2, \ldots, x_{c+1}) & \cdots  & x_{a+1}^{q^{b+1}}f(x_1, \ldots, \widehat{x_{a+1}}, \ldots, x_{c+1})
        \end{vmatrix}}{\begin{vmatrix}
            1 & \cdots & x_{a+1}\\
            x_1^{q-1} & \cdots & x_{a+1}^q\\
            \vdots & \ddots & \vdots\\
            x_1^{q^{a-1}-1} & \cdots  & x_{a+1}^{q^{a-1}}\\
            x_1^{q^{a}-1} & \cdots  & x_{a+1}^{q^{a}}\\
        \end{vmatrix}}. 
    \end{equation*}
    Assigning $x_1=0$, we obtain
    \begin{align*}
        \delta_{a+1;b+1}(f)(0, x_1, \ldots, x_n)&= \dfrac{\begin{vmatrix}
            1 & \cdots & x_{a+1}\\
            0 & \cdots & x_{a+1}^q\\
            \vdots & \ddots & \vdots\\
            0 & \cdots  & x_{a+1}^{q^{a-1}}\\
            0 & \cdots  & x_{a+1}^{q^{b+1}}f(0, \ldots, \widehat{x_{a+1}}, \ldots, x_{c+1})
        \end{vmatrix}}{\begin{vmatrix}
            1 & \cdots & x_{a+1}\\
            0 & \cdots & x_{a+1}^q\\
            \vdots & \ddots & \vdots\\
            0 & \cdots  & x_{a+1}^{q^{a-1}}\\
            0 & \cdots  & x_{a+1}^{q^{a}}\\
        \end{vmatrix}}\\
        &=\delta_{a;b}(f(0, x_2, \ldots, x_n))^q.
    \end{align*}
    Therefore, by definition, we have
    \begin{equation*}
        \Phi(Y_{m-1}(I;J))(0, x_2, \ldots, x_n) = Y_{m-1}(I; J)(x_2, \ldots, x_n)^q, \text{ for all } Y_{m-1}(I; J)\in \mathcal{B}_{m-1}(n-1). 
    \end{equation*}
    It follows that for $0\leq a<q^m-1$ and $Y\in \mathcal{B}_{m-1}(n-1)$, we have
    \begin{equation*}
        lm(V_1^a \cdot \Phi(Y)) = x_1^a \cdot lm(Y)^q.
    \end{equation*}
    Since $a<q^m-1$, the leading monomials of the first and the second families are distinct; by the induction hypothesis, the leading monomials of elements of the first family are distinct, similarly for the second family, proving the claim.
\end{proof}
\begin{proof}[Proof of \cref{thm: q=p_case}]
    By \cref{cor:Y_is_polynomial}, it suffices to prove that $\mathcal{B}_m(n)$ spans $Q(m, n)^U$. We proceed via induction by examining the leading monomial of a nonzero homogeneous element $f\in Q(m, n)^U$. If $lm(f)$ is divisible by $x_1^{q^m-1}$, then 
    \begin{equation*}
        f(x_1, \ldots, x_n) = x_1^{q^m-1}g(x_2, \ldots, x_n),
    \end{equation*}
    where $g\in \mathbb{F}_q[x_1, \ldots, x_{n-1}]$. It is obvious that $g\in Q(m, n-1)^U$, by the induction hypothesis, $f = \delta_{1;m}(g)$ belongs to the $\mathbb{F}_q$-span of the first family in $\mathcal{B}_m(n)$. Otherwise, suppose that
     \begin{equation*}
            f(x_1, \ldots, x_n) = x_1^k f_1(x_2, \ldots, x_n) + x_1^{k+1} g(x_1, \ldots, x_n),
        \end{equation*}
        where $0\leq k<p^m-1$, and $f_1$ is a nonzero polynomial. It is not hard to verify that $f_1\in Q(m, n-1)^{U}$, thus if $f_1 = g_1^p$ for some polynomial $g_1$, then $g_1$ automatically belongs to $Q(m-1, n-1)^{U}$. Now suppose for the sake of contradiction that there exists a monomial $a_{a_2, \ldots, a_n}\cdot x_2^{a_2}\cdots x_n^{a_n}$ with nonzero coefficient in $f_1$, such that $p\nmid \gcd(a_2, \ldots, a_n)$. If $p\nmid a_j$ for some $j\geq 2$, then for all $a\in \mathbb{F}_p$, 
        \begin{align*}
            &f(x_1, \cdots, x_j + a x_1, \cdots, x_n) - f(x_1, \ldots, x_n)  =\\ &=x_1^k \left(a_{a_2, \ldots, a_n}\prod\limits_{j'\neq j}x_{j'}^{a_{j'}}\cdot((x_j+ax_1)^{a_j}-x_j^{a_j})+\right.\\
            &\left.+\sum\limits_{(a'_2, \ldots, a'_n)\neq (a_2, \ldots, a_n)}a_{a'_2, \ldots, a'_n}\prod\limits_{j'\neq j}x_{j'}^{a'_{j'}}\cdot((x_j+ax_1)^{a'_j}-x_j^{a'_j}) \right) +x_1^{k+2} r(x_1, \ldots, x_n),
        \end{align*} 
        for some $r\in \mathbb{F}_p[x_1, \ldots, x_n]$. Because $k< p^{m}-1$, the coefficient of $x_1^{k+1}$ in the right-hand side must be zero for any $a$. The only term in the right-hand side that contributes to the coefficient of the monomial $x_1^{k+1}\prod\limits_{j'\neq j}x_{j'}^{a_{j'}}\cdot x_j^{a_j-1}$ is $a_{a_2, \ldots, a_n}\prod\limits_{j'\neq j}x_{j'}^{a_{j'}}\cdot((x_j+ax_1)^{a_j}-x_j^{a_j})$, and the contribution is 
        \begin{equation*}
            a\cdot a_j\cdot a_{a_2, \ldots, a_n}\prod\limits_{j'\neq j}x_{j'}^{a_{j'}}\cdot x_j^{a_j-1}.
        \end{equation*}
        Therefore,
        \begin{equation*}
            a\cdot a_{a_2, \ldots, a_n}\cdot a_j =0, \text{ for all } a\in \mathbb{F}_p.
        \end{equation*}
        This shows that $a_{a_2, \ldots, a_n} = 0$, a contradiction. We conclude that $f_1 = g_1^p$ for some $g_1\in Q(m, n-1)^U$, and by the induction hypothesis, there exists $Y\in \mathcal{B}_{m-1}(n-1)$ such that $lm(V_1^{k}\Phi(Y)) = lm(f)$, therefore, $f-c\cdot V_1^k \Phi(Y)$ has leading monomial strictly smaller than $lm(f)$ for some $c\in \mathbb{F}_q$. Repeating this process, we deduce that $Q(m, n)^U$ is generated by $\mathcal{B}_m(n)$.
\end{proof}
\begin{remark}
    The proof of \cref{thm: q=p_case} relies heavily on the construction of the Frobenius-like operator, and the proven property of the leading monomial of an invariant element. It is still an open question to generalize this method to compute an $\mathbb{F}_q$-basis for $Q(m, n)^U$ in the general case (when $q>p$.)
\end{remark}
\section{The Delta Operator}\label{sec:delta_operator}
As seen in \Cref{sec:the_conj}, $\delta$ operators are the main ingredient in constructing bases for invariant spaces of truncated polynomial rings. Furthermore, a substantial portion of the proof of \cref{thm: q=p_case}, as well as that of \cite[Theorem 1.6]{Ha_Hai_Nghia_2024}, is devoted to examining the polynomiality\index{polynomiality} property of $\delta$ operators. In this section, we address this problem in more detail and propose a conjecture, along with partial results related to this new topic.\smallbreak
Fix a nonnegative integer $m\geq 0$. In this section, we use the shortened notation $\delta_n$ for $\delta_{n;m}$.
\begin{definition}\label{def:}
    Consider the polynomial ring $S = \mathbb{F}_q[x_1, \ldots, x_n]$. A polynomial $f\in S$ is called \emph{$(n+1, m)$-excellent}\index{$(n+1, m)$-excellent polynomial} if $\delta_{n+1}(f)$ is a polynomial. Denote by $A_{m}^{(n)}$ the set of all $(n+1, m)$-excellent polynomials. 
\end{definition}
\begin{remark}
    \begin{enumerate}
    \item We may consider the set of all $(n+1, m)$-excellent polynomials $f\in \mathbb{F}_q[x_1, \ldots, x_c]$ for any $c\geq n$, but by \cref{lemma:functional_eq}, this property is determined by the first $n$ variables, hence it suffices to consider the space $A_{m}^{(n)}$ as above.
        \item \cref{lemma:functional_eq} shows that $f\in A_{m}^{(n)}$ if and only if $f$ satisfies \eqref{eq:fundamental_functional_eq} for all $c_1, \ldots, c_{n}\in \mathbb{F}_q$. Furthermore, the proof of \cref{lemma:functional_eq} points out the fact that: if $f\in A_{m}^{(n)}$, then $f$ is a symmetric polynomial. Moreover, letting $c_{n} \in \mathbb{F}_q^{\times}$ and $c_1= \cdots = c_{n-1}=0$, we have
        \begin{equation*}
            f(x_1, \ldots, c_n x_n) = f(x_1, \ldots, x_n).
        \end{equation*}
        Consequently, $f\in \mathbb{F}_q[x_1^{q-1}, \ldots, x_n^{q-1}]$.
        \item One checks that the above properties of elements of $A_{m}^{(n)}$ suffice to determine $A_{m}^{(n)}$ in the case $n=1$; in this case, $A_{m}^{(n)} = \mathbb{F}_q[x_1^{q-1}] = \mathcal{D}_1$, for any $m$. However, when $n>1$ and $m>0$, $A_{m}^{(n)}$ is strictly larger than $\mathcal{D}_n$; for example, $(x_1\cdots x_n)^{q^m-1}\in A_{m}^{(n)}\backslash \mathcal{D}_n$.
        \item \cref{lemma:polynomiality_D_mod} establishes that $A_{m}^{(n)}$ is a module over the Dickson algebra $\mathcal{D}_n$. Hence, it is natural to examine the $\mathcal{D}_n$-module structure of $A_{m}^{(n)}$.
    \end{enumerate}
\end{remark}
The following conjecture gives a generating set for $A_{m}^{(n)}$ as a $\mathcal{D}_n$-module, motivated by \cref{lemma:chain_complex} and \cref{lemma:polynomiality_D_mod}.
\begin{conjecture}\label{conj:polynomiality}
   For $n\geq 2$, $A_{m}^{(n)}$ is generated as a $\mathcal{D}_n$-module by the set
    \begin{equation*}
        G_{m}^{(n)} = \{1\}\cup \{\delta_{n}(g)\mid g\in A_{m}^{(n-1)}\}.
    \end{equation*}
\end{conjecture}
\subsection{The Bivariate Case}\label{subsection:poly_bivariate}
We prove a strengthened form of \cref{conj:polynomiality} in the bivariate case.
\begin{theorem}\label{thm:pol_bivariate}
    $A_{m}^{(2)}$ is a free $\mathcal{D}_2$-module with a basis
    \begin{equation*}
        H_{m}^{(2)} = \{1\}\sqcup \{\delta_2(x_1^{s(q-1)})\mid 2\leq s\leq q\}.
    \end{equation*}
    Furthermore, the $\mathcal{D}_2$-module
\begin{equation*}
    \mathcal{M} = \left\{\dfrac{f(x_1, x_2) - f(x_1 + x_2, x_2)}{x_2^{q^m}}\mid f\in A_{m}^{(2)}\right\}
\end{equation*} 
does not depend on the choice of $m$. In particular, $\mathcal{M}$ is the free $\mathcal{D}_2$-module generated by 
\begin{equation*}
    \left\{\dfrac{x_1^{s(q-1)+1} + x_2^{s(q-1)+1} - (x_1 + x_2)^{s(q-1)+1}}{\begin{vmatrix}
        x_1 & x_2\\
        x_1^{q}& x_2^{q}
    \end{vmatrix}}\mid 2\leq s\leq q\right\}.
\end{equation*}
\end{theorem}
\begin{proof}
    For each $s\in \{2, \ldots, q\}$, denote $y_s = \delta_2(x_1^{s(q-1)})$. By \cref{lemma:chain_complex}, it is obvious that $H_m^{(2)}$ is a subset of $A_m^{(2)}$. \smallbreak
    Consider the graded lexicographic order on $\mathbb{F}_q[x_1, x_2]$ by asserting $x_1>x_2$. For each nonzero homogeneous polynomial $f\in \mathbb{F}_q[x_1, x_2]$, let $in(f)$ be the leading monomial of $f$ with respect to the lexicographic order just defined. Consider the map
    \begin{align*}
        G: \mathbb{F}_q[x_1, x_2]&\to \mathbb{F}_q(x_1, x_2)\\ 
            f&\mapsto \dfrac{f(x_1, x_2) - f(x_1+x_2, x_2)}{x_2^{q^m}}.
    \end{align*}
    The map $G$ is of significant importance, as shown in the following lemma.
    \begin{lemma}\label{lemma:crucial_prop}
    Let $f\in \mathbb{F}_q[x_1^{q-1}, x_2^{q-1}]$, such that $f$ is symmetric. Then $f\in A_m^{(2)}$ if and only if $G(f)\in \mathbb{F}_q[x_1, x_2]$ and $G(f)$ is antisymmetric, furthermore, $G(f) = 0$ if and only if $f\in \mathcal{D}_2$. 
\end{lemma}
\begin{proof}
    If $f\in A_{m}^{(2)}$, from \eqref{eq:fundamental_functional_eq}, by specializing $c_1=1$, we see that $G(f)$ must be a polynomial; furthermore,
    \begin{equation*}
        G(f)(x_1, x_2) = \dfrac{f(x_1, x_1+x_2) - f(x_1, x_2)}{x_1^{q^m}}, \text{ for all } f\in A_m^{(2)}.
    \end{equation*}
    Therefore, $G(f)(x_2, x_1) = \dfrac{f(x_2, x_1+x_2) - f(x_1, x_2)}{x_2^{q^m}} = -G(f)(x_1, x_2)$, which implies that $G(f)$ is antisymmetric. Conversely, if $G(f)\in \mathbb{F}_q[x_1, x_2]$, then we must show that \eqref{eq:fundamental_functional_eq} holds for all $c_1\in \mathbb{F}_q$. This equation obviously holds when $c_1=0$; when $c_1\neq 0$, we have 
    \begin{equation*}
        G(f)(x_1, ax_2) = \dfrac{f(x_1, x_2) - f(x_1 + ax_2, x_2)}{ax_2^{q^m}} = \dfrac{f(x_1, x_1+ax_2) - f(x_1, x_2)}{x_1^{q^m}}.
    \end{equation*}
    The second equality holds because $G(f)$ is antisymmetric. The equality of the middle and the right-hand side is exactly \eqref{eq:fundamental_functional_eq}. Finally, $G(f) = 0$ if and only if $f(x_1+x_2, x_2) = f(x_1, x_2)$, because $f$ is symmetric, and $f\in \mathbb{F}_q[x_1^{q-1}, x_2^{q-1}]$, one sees that $G(f) = 0$ if and only if $f$ is $GL_2(\mathbb{F}_q)$-invariant, or $f\in \mathcal{D}_2$.
\end{proof}
Fix a homogeneous polynomial $f\in A_m^{(2)}$. The next lemma provides the functional equation for $G(f)$, from which we can extract useful information about the leading monomial of $G(f)$.
\begin{lemma}
    For all $a\in \mathbb{F}_q$, we have
    \begin{equation}\label{eq:fundamental_eq_for_G}
        G(f)(x_1+ax_2, x_2) + aG(f)(x_1, ax_2) = (a+1)G(f)(x_1, (a+1)x_2).
    \end{equation}
\end{lemma}
\begin{proof}
    By direct manipulation of the formulas, we have
\begin{equation*}
    G(f)(x_1+ax_2, x_2)  = \dfrac{f(x_1+ax_2, x_2) - f(x_1, x_2)}{x_2^{q^m}} + \dfrac{f(x_1, x_2) - f(x_1+(a+1)x_2, x_2)}{x_2^{q^m}}.
\end{equation*}
If $a\neq 0$, we have $\dfrac{f(x_1+ax_2, x_2) - f(x_1, x_2)}{x_2^{q^m}} = -aG(f)(x_1, ax_2)$, it is not hard to check that the above equality also holds when $a=0$. Similarly, $\dfrac{f(x_1, x_2) - f(x_1+(a+1)x_2, x_2)}{x_2^{q^m}} = (a+1)G(f)(x_1, (a+1)x_2), \text{ for all } a\in \mathbb{F}_q$. Substituting into the above equation, we get \eqref{eq:fundamental_eq_for_G}, as desired. 
\end{proof}
Now, fix a polynomial $f\in A_m$ such that $G(f)\neq 0$, assume that the leading monomial $in(G(f)(x_1, x_2)) = x_1^k x_2^l$. Then, since $G(f)$ is antisymmetric, we have $k\geq l$, furthermore,
\begin{equation*}
    G(f)(x_1, x_2) = \sum\limits_{j = l}^{k} c_j x_1^{l+k-j}x_2^{j}.
\end{equation*}
Firstly, we claim that $G(f)$ is divisible by $(x_1^q x_2 - x_2^q x_1)^{l}$. The claim trivially holds if $l=0$. Suppose that $l>0$, then $x_2 | G(f)(x_1, x_2)$. Therefore, by setting $x_1 = 0$, we have $G(f)(0, x_2) = 0$, from \eqref{eq:fundamental_eq_for_G}, we have $G(f)(ax_2, x_2) = 0, \text{ for all } a\in \mathbb{F}_q$, which means that $G(f)(x_1, x_2)$ is divisible by the least common multiple of the homogeneous polynomials of degree $1$ in $\mathbb{F}_q[x_1, x_2]$, which is $L_2 = x_1^q x_2 - x_2^q x_1$. If $l'$ is the largest exponent such that $L_2^{l'}|G(f)(x_1, x_2)$, then let $H(f)(x_1, x_2) = \dfrac{G(f)}{L_2^{l'}}$, \eqref{eq:fundamental_eq_for_G} is equivalent to
\begin{equation*}
    H(f)(x_1+ax_2, x_2) + a^{l'+1}H(f)(x_1, ax_2) = (a+1)^{l'+1}H(f)(x_1, (a+1)x_2), \text{ for all } a\in \mathbb{F}_q.
\end{equation*}
Because $G(f)$ and $L_2$ are both antisymmetric, $H(f)$ is either symmetric or antisymmetric. Furthermore, if $l'<l$, then the leading monomial of $H(f)$ is $x_1^{k-ql'} x_2^{l-l'}$, which is divisible by $x_2$, therefore, $H(f)(0, x_2) = 0$, thus $H_f(ax_2, x_2) = 0, \text{ for all } a\in \mathbb{F}_q$, which contradicts to the maximality of $l'$. Consequently, $l = l'$.\smallbreak
Secondly, we show that there exists unique nonnegative integers $r, e$, such that $p^e < q$, and $l = (q-1)r + p^e - 1$. Indeed, because $in(G(f)(x_1, x_2)) = x_1^k x_2^l$, $in(G(f)(x_1+ax_2, x_2)) = x_1^k x_2^l, \text{ for all } a\in \mathbb{F}_q$. Calculating the coefficients of the leading monomials of the left-hand side and the right-hand side of \eqref{eq:fundamental_eq_for_G}, we have
\begin{equation*}
    1 + a^{l+1} = (a+1)^{l+1}, \text{ for all } a\in \mathbb{F}_q.
\end{equation*}
If $l = (q-1)r + r_1$, for some $0\le r_1\le q-2$, then we have $1+ a^{r_1+1} = (a+1)^{r_1+1}, \text{ for all } a\in \mathbb{F}_q$. Therefore, the polynomial equation $(x+1)^{r_1+1} = x^{r_1+1} + 1$ has $q>r_1+1$ roots in $\mathbb{F}_q$, which means that $(x+1)^{r_1+1} = x^{r_1+1} + 1$ as polynomials in $\mathbb{F}_q$. This only happens when $r_1+1 = p^e$, for some $e\geq 0$, such that $p^e < q$. Therefore, $l = (q-1)r + p^e-1$, as desired. Uniqueness follows from the fact that if $q = p^m$, then $\{1, p,\ldots, p^{m-1}\}$ are pairwise different modulo $q-1$.\smallbreak
From the second claim, we see that $G(f)(x_1, x_2)$ is divisible by the polynomial $(x_1^q x_2 - x_1 x_2^q)^{(q-1)r} = Q_{2, 0}^{r}$, and the polynomial $G_1(f) = \dfrac{G(f)}{Q_{2, 0}^{r}}$ is antisymmetric and satisfies \eqref{eq:fundamental_eq_for_G}, therefore, one may assume that $l = p^e - 1$ for some $e\ge 0$, $p^e<q$. It is obvious that $deg(G(f))$ is $(-1)$ modulo $(q-1)$, from which we can assume that $k = (q-1)(qk_1 + s) - p^e$, for some nonnegative integers $k_1, s$, such that $0\le s\le q-1$. Our third claim is that, $qk_1 + s$ is divisible by $p^e$ but is not divisible by $p^{e+1}$, in other words, $v_p(qk_1+s) = e$. Indeed, 
\begin{itemize}
    \item If $v_p(qk_1+s)=v_p(k) = d<e$, then the coefficient of $x_1^{k-p^d}x_2^{l+p^d}$ in $G(f)(x_1 + ax_2, x_2)$ is of the form 
\begin{equation*}
    C_l a^{p^d} + \alpha(a),
\end{equation*}
where $C_l = c_l \cdot \binom{k}{p^d}\neq 0$, and $\alpha\in \mathbb{F}_q[x]$ is a polynomial of degree strictly less than $p^d$. It is not hard to see that the coefficient of $x_1^{k-p^d}x_2^{l+p^d}$ in $(a+1)G(f)(x_1, (a+1)x_2)-aG(f)(x_1, ax_2)$ is
\begin{equation*}
    c_{p^d+l}((a+1)^{p^d + p^e+1}-a^{p^d+p^e+1}) = c_{p^d + l}(a^{p^d} + a^{p^e}+1).
\end{equation*}
By \eqref{eq:fundamental_eq_for_G}, these coefficients are equal, therefore,
\begin{equation*}
    C_l a^{p^d} + \alpha(a) = c_{p^d + l}(a^{p^d} + a^{p^e}+1), \text{ for all } a\in \mathbb{F}_q.
\end{equation*}
As $p^d < p^e<q$, and $\deg(\alpha)<p^d$, it follows that $c_{p^d+l} = 0$, and hence $C_l = 0$, a contradiction.
\item If $v_p(qk_1+s)= d>e$, then $k \equiv p^{e+1} - p^{e}(\text{mod }p^{e+1})$. Therefore, the coefficients of $x_1^{k+p^e - p^{e+1}}x_2^{p^{e+1}-1}$ in $G(f)(x_1+ax_2, x_2)$ is of the form 
\begin{equation*}
    C_l a^{p^{e+1}-p^e} + \alpha(a),
\end{equation*}
where $C_l = c_l \cdot \binom{k}{p^{e+1}-p^e}\neq 0$, and $\alpha\in \mathbb{F}_q[x]$ is a polynomial of degree strictly less than $p^{e+1}-p^e$. It is not hard to see that the coefficient of $x_1^{k+p^e - p^{e+1}}x_2^{p^{e+1}-1}$ in $(a+1)G(f)(x_1, (a+1)x_2)-aG(f)(x_1, ax_2)$ is
\begin{equation*}
    c_{p^{e+1}-1}((a+1)^{p^{e+1}}-a^{p^{e+1}}) = c_{p^{e+1}-1}.
\end{equation*}
By \eqref{eq:fundamental_eq_for_G}, these coefficients are equal, therefore,
\begin{equation*}
    C_l a^{p^{e+1}-p^e} + \alpha(a) = c_{p^{e+1}-1}, \text{ for all } a\in \mathbb{F}_q.
\end{equation*}
Since $p^{e}<q$, we have $p^{e+1}-p^e < q$, therefore, the above equality is a contradiction. 
\end{itemize}
The arguments just presented have shown that $k = (q-1)(q(k_1+r) + p^e s_0) - p^e$, where $p^e s_0 \in \{1, \ldots, q-1\}$, and $s_0$ is not divisible by $p$, recall that $l = (q-1)r + p^e-1$. Choose $s = 1 + p^e s_0$, then $s\in \{2, \ldots, q\}$, we easily see that if $f_0 = Q_{2, 1}^{k_1} Q_{2, 0}^r\cdot y_s \in A_m$, then $in(G(f_0)) = x_1^{k}x_2^{l}$, therefore, $f-c_l f_0\in A_m$, and $G(f-c_l f_0)$ is either zero or has the leading monomial strictly smaller than $x_1^k x_2^l$. This allows us to perform an inductive argument to show that $\mathcal{M}$ is generated over $\mathcal{D}_2$ by $G(y_2), \ldots, G(y_q)$, and by \cref{lemma:crucial_prop}, it follows that $A_m^{(2)}$ is generated over $\mathcal{D}_2$ by $\{1, y_2, \ldots, y_q\}$.\smallbreak
Finally, if there are $f_2, \ldots, f_q\in \mathcal{D}_2$ which are not all zero, such that
\begin{equation*}
    f_2 y_2 + \cdots + f_q y_q\in \mathcal{D},
\end{equation*}
then $f_2 G(y_2) + \cdots + f_q G(y_q) = 0$, consequently, there exists $2\leq s < t \leq q$, such that $f_s, f_t\neq 0$, and $in(f_s G(y_s)) = in(f_t G(y_t)) = x_1^k x_2^l$. Because $f_s$ and $f_t$ are in $\mathbb{F}_q[x_1^{q-1}, x_2^{q-1}]$, we see that 
\begin{equation*}
    p^{v_p(s-1)}-1 \equiv l\equiv p^{v_p(t-1)}-1 (\text{mod }q-1),
\end{equation*}
since $p^{v_p(s-1)}<q$ and $p^{v_p(t-1)}<q$, it follows that $v_p(s-1) = v_p(t-1) = d$, thus $s = 1+p^d s_0, t = 1+p^d t_0$ for some $s_0, t_0$ not divisible by $p$. Now, because $f_s, f_t\in \mathcal{D}_2$, one may assume that 
\begin{equation*}
    in(f_s) = x_1^{q(q-1)(k_s+l_s)}x_2^{(q-1)l_s}, \quad in(f_t) = x_1^{q(q-1)(k_t+l_t)}x_2^{(q-1)l_t}
\end{equation*}
for some nonnegative integers $k_s, l_s, k_t, l_t$. It follows that $(q-1)l_t + p^d - 1 = (q-1)l_s + p^d-1$, because they both equal the exponent of $x_2$ in the monomials
\begin{equation*}
    in(f_s G(y_s))= in(f_t G(y_t)),
\end{equation*}
hence $l_s = l_t$, consider the exponent of $x_1$ in $in(f_s G(y_s)) = f_t G(y_t))$, we have
\begin{equation*}
    q(q-1)k_s + (s-1)(q-1) = q(q-1)k_t + (t-1)(q-1).
\end{equation*}
Consequently, $q\mid (t-s)$, which is impossible since $2\le s<t\le q$. We conclude that $A_m$ is free over $\mathcal{D}_2$ with basis $\{1, y_2, \ldots, y_q\}$, and $\mathcal{M}$ is free over $\mathcal{D}_2$ with basis $\{G(y_2), \ldots, G(y_q)\}$.
\end{proof}
\begin{remark}
    The proof of \cref{thm:pol_bivariate} also gives an algorithm to express an element $f\in A_m^{(2)}$ as a $\mathcal{D}_2$-linear combination of $H_m^{(2)}$ by examining the leading monomial of $G(f)$.
\end{remark}
\subsection{Partial Results for the General Case}
We present an alternative approach to \cref{conj:polynomiality} in the general case, employing the following results.
\begin{remark}
    We may regard the formula for the $\delta$ operators as a root produced by Cramer's rule, with respect to a specific system of linear equations. In particular, if $f$ is a rational function in $r$ variables and $r\geq n$, by Laplace expansion, for all $i\in \{0,\ldots, n\}$,
    \begin{equation*}
        \delta_{n+1;m}(Q_{n,i}f) = \dfrac{\begin{vmatrix}
        x_1 & x_2 & \cdots & x_{n+1}\\
        \vdots & \vdots & \ddots & \vdots\\
        \widehat{x_1^{q^i}} & \widehat{x_2^{q^i}} & \cdots & \widehat{x_{n+1}^{q^i}}\\
        \vdots & \vdots & \ddots & \vdots\\
        x_1^{q^{n}} & x_2^{q^{n}} & \cdots & x_{n+1}^{q^{n}}\\
        x_1^{q^m}f(\widehat{x_1}) & x_2^{q^m}f(\widehat{x_2}) & \cdots & x_{n+1}^{q^m}f(\widehat{x_{n+1}})
    \end{vmatrix}}{\begin{vmatrix}
        x_1 & x_2 & \cdots & x_{n+1}\\
        x_1^{q} & x_2^{q} & \cdots & x_{n+1}^q\\
        \vdots & \vdots & \ddots & \vdots\\
        x_1^{q^{n-1}} & x_2^{q^{n-1}} & \cdots & x_{n+1}^{q^{n-1}}\\
        x_1^{q^n} & x_2^{q^n} & \cdots & x_{n+1}^{q^n}
    \end{vmatrix}}.
    \end{equation*}
    Therefore, $(\delta_{n+1}(Q_{n, 0}f), \ldots, \delta_{n+1}(Q_{n, n}f))^{T}$ is the unique solution of the following system of equations
    \begin{equation}\label{eq:sys_of_eqs}
        \begin{aligned}
         \sum\limits_{i=0}^{n} x_{1}^{q^i}\cdot (-1)^{n+i}\delta_{n+1;m}(Q_{n,i}f)&=x_{1}^{q^m}f(x_2, \ldots, x_{r+1}) \\
         \sum\limits_{i=0}^{n} x_{2}^{q^i}\cdot (-1)^{n+i}\delta_{n+1;m}(Q_{n,i}f)&=x_{2}^{q^m}f(x_1, x_3 \ldots, x_{r+1}) \\
         \cdots &\\
         \sum\limits_{i=0}^{n} x_{n+1}^{q^i}\cdot (-1)^{n+i}\delta_{n+1;m}(Q_{n,i}f)&=x_{n+1}^{q^m}f(x_1, \ldots, x_{n}, x_{n+2}, \ldots, x_{r+1}).
    \end{aligned}
    \end{equation}
\end{remark}
\begin{proposition}[{\cite{ha_hai_nghia_GL}}]\label{prop:commutator}
    For each rational function $f$ in $r$ variables with $r\geq n$, the following identity holds for all $i\in \{0, \ldots, n-1\}$.
    \begin{equation}\label{eq:commutator}
        Q_{n, i}\delta_{n+1}(f) - \delta_{n+1}(Q_{n,i}f) = \delta_{n}(Q_{n-1, i}f).
    \end{equation}
\end{proposition}
\begin{proof}
    For any $j\in \{1, \ldots, n\}$, by the fundamental equation, we have
    \begin{equation*}
        \sum\limits_{i=0}^{n-1}x_j^{q^i}Q_{n, i}\cdot (-1)^{n+i}\delta_{n+1}(f) = -x_j^{q^n}\delta_{n+1}(f).
    \end{equation*}
    By the above remark, we also have
    \begin{equation*}
        \sum\limits_{i=0}^{n-1}x_j^{q^i}\cdot (-1)^{n+i}\delta_{n+1}(Q_{n, i} f) = x_j^{q^m}f(x_1, \ldots, \widehat{x_j}, \ldots, x_{r+1}) - x_j^{q^n}\delta_{n+1}(f).
    \end{equation*}
    Consequently, for $1\leq j\leq n$, we have
    \begin{equation}\label{eq:part_1_of_eqsys}
        \sum\limits_{i=0}^{n-1}x_j^{q^i}\cdot (-1)^{n+i}(Q_{n, i}\delta_{n+1}(f)-\delta_{n+1}(Q_{n, i} f)) = x_j^{q^m}f(x_1, \ldots, \widehat{x_j}, \ldots, x_{r+1}).
    \end{equation}
    The proof is completed by combining \eqref{eq:sys_of_eqs} (applied for $\delta_n$) and \eqref{eq:part_1_of_eqsys}. 
\end{proof}
The next lemma shows that the $\delta$ operators constitute an exact sequence.
\begin{lemma}
    \label{lemma:exactness}
If $f$ is a polynomial of $r\geq n$ variables, and $\delta_{n+1}(f) = 0$, then there exists a polynomial $g$ of $(r-1)$ variables such that $f = \delta_n(g)$. 
\end{lemma}
\begin{proof}
    If $r\geq n+1$, we may write $f = \sum\limits_{i=1}^{k} f_i(x_1, \ldots, x_n) g_i(x_{n+1}, \ldots, x_r)$, where $g_1, \ldots, g_k$ are distinct nonzero monomials, then $\delta_{n+1}(f) = \sum\limits_{i=1}^{k}\delta_{n+1}(f_i)\cdot g_i(x_{n+2}, \ldots, x_{r+1})$, thus it suffices to prove \cref{lemma:exactness} in the case $r=n$. Suppose that $f$ is homogeneous of degree $d$, let $F(x_1, \ldots, x_n) = f(x_1, \ldots, x_n)\cdot L(x_1, \ldots, x_n)$, we have
    \begin{equation*}
        F(x_1, \ldots, x_n) = \sum\limits_{i=0}^{d+\frac{q^n-1}{q-1}}x_n^{i}\cdot F_{d+\frac{q^n-1}{q-1}-i}(x_1, \ldots, x_{n-1}).
    \end{equation*}
    Because $\delta_{n+1}(f) = 0$, we have
    \begin{equation*}
        \sum\limits_{j=1}^{n+1}(-1)^{j+1}x_j^{q^m}\cdot F(x_1, \ldots, \widehat{x_j}, \ldots, x_{n+1}) = 0,
    \end{equation*}
    which implies that
    \begin{align*}
        0&=(-1)^{n}x_{n+1}^{q^m}\cdot F(x_1, \ldots, x_n) + \sum\limits_{i=0}^{d+\frac{q^n-1}{q-1}}x_{n+1}^i\sum\limits_{j=1}^{n}(-1)^{j+1}x_j^{q^m}F_{d+\frac{q^n-1}{q-1}-i}(\widehat{x_j})\\
        &=(-1)^{n}x_{n+1}^{q^m}\cdot F(x_1, \ldots, x_n) + \sum\limits_{i=0}^{d+\frac{q^n-1}{q-1}}x_{n+1}^i\cdot L(x_1, \ldots, x_n) \delta_n\left(\dfrac{F_{d+\frac{q^n-1}{q-1}-i}}{L(x_1, \ldots, x_{n-1})}\right).
    \end{align*}
    Note that $F$ is divisible by $L(x_1, \ldots, x_n)$, hence for each $i$, $F_{d+\frac{q^n-1}{q-1}-i}$ is divisible by $L(x_1, \ldots, x_{n-1})$. Dividing both sides by $L(x_1, \ldots, x_n)$, and consider the coefficient of $x_{n+1}^{q^m}$ of both sides, we have
    \begin{equation*}
        (-1)^{n}f(x_1, \ldots, x_n) + \delta_n\left(\dfrac{F_{d+\frac{q^n-1}{q-1}-q^m}}{L(x_1, \ldots, x_{n-1})}\right) = 0.
    \end{equation*}
    In particular, $f = \delta_n\left((-1)^{n+1}\dfrac{F_{d+\frac{q^n-1}{q-1}-q^m}}{L(x_1, \ldots, x_{n-1})}\right)$, as desired.
\end{proof}
In order to generalize \cref{lemma:exactness}, we need the following notion.
\begin{definition}\label{def:h-vanishing}
For each $h\in \mathbb{N}$: 
\begin{enumerate}
    \item A polynomial $f$ of $r\geq n$ variables is called $h$-vanishing if $\delta_{n+1}^{h}(f) = 0$. We simply say that $f$ is vanishing if there exists an $h$ such that $f$ is $h$-vanishing.\index{vanishing polynomial}
\item A polynomial $f$ of $r\geq n$ variables is called $h$-constructible if there exists a set of polynomials of $(r-1)$ variables \index{constructible polynomial}
\begin{equation*}
S = \{g_{i_{0}i_{1}\cdots i_{n-1}} \mid i_0 + i_1 + \cdots + i_{n-1} \leq h-1\},    
\end{equation*}
such that $\delta_n(g)$ is a polynomial for all $g\in S$, and 
\begin{equation*}
    f = \sum\limits_{i_0 + i_1 + \cdots + i_{n-1} \leq h-1}Q_{n,0}^{i_0}Q_{n,1}^{i_1}\cdots Q_{n, n-1}^{i_{n-1}}\delta_n(g_{i_{0}i_{1}\cdots i_{n-1}}).
\end{equation*}
\end{enumerate}
\end{definition}
\begin{lemma}
    \label{lemma:general_exact}
Suppose that $f$ is a polynomial of $r\geq n$ variables. Then $f$ is $h$-constructible if and only if $f$ is $h$-vanishing.
\end{lemma}
\begin{proof}
    Applying $\delta_{n+1}$ to both sides of \eqref{eq:commutator}, we have
    \begin{equation*}
        \delta_{n+1}(Q_{n, i}\delta_{n+1}(f)) = \delta_{n+1}^2(Q_{n, i}f).
    \end{equation*}
    Consequently, by induction, for any $i_1, \ldots, i_{h-1}\in \{0, \ldots, n\}$, 
    \begin{equation*}
        \delta_{n+1}^h(Q_{n, i_1}\cdots Q_{n, i_{h-1}} f) = \delta_{n+1}(Q_{n, i_1}\delta_{n+1}(Q_{n, i_2}(\cdots(Q_{n, i_{h-1}}\delta_{n+1}(f))\cdots ). 
    \end{equation*}
    If $f = \delta_n(g)$, then $\delta_{n+1}(f) = 0$, hence $\delta_{n+1}^h(Q_{n, i_1}\cdots Q_{n, i_{h-1}} f) = 0$. It follows that if $f$ is $h$-constructible then $f$ is $h$-vanishing.\smallbreak
    We prove the converse by induction on $h$. The base case $h=1$ is proven in \cref{lemma:exactness}. Suppose that \cref{lemma:general_exact} is true for some $h \in \mathbb{N}$. If $f$ is an $(h+1)$-vanishing polynomial of $r\geq n$ variables, then $\delta_{n+1}(f)$ is an $h$-vanishing polynomial of $(r+1)$ variables. By the induction hypothesis, $\delta_{n+1}(f)$ is $h$-constructible. Therefore, we have a set of polynomials of $r$ variables 
\begin{equation*}
S = \{g_{i_{0}i_{1}\cdots i_{n-1}} \mid i_0 + i_1 + \cdots + i_{n-1} \leq h-1\},    
\end{equation*}
such that $\delta_n(g)$ is a polynomial for all $g\in S$, and 
\begin{equation}\label{eq:description_of_delta_f}
    \delta_{n+1}(f) = \sum\limits_{i_0 + i_1 + \cdots + i_{n-1} \leq h-1}Q_{n,0}^{i_0}Q_{n,1}^{i_1}\cdots Q_{n, n-1}^{i_{n-1}}\delta_n(g_{i_{0}i_{1}\cdots i_{n-1}}).
\end{equation}
As in previous arguments, we can assume without loss of generality that $r = n$ and $f$ is homogeneous of degree $d$. In this case, let $F = f\cdot L(x_1, \ldots, x_n)$, suppose that $F = \sum\limits_{i=0}^{d+\frac{q^n-1}{q-1}}x_n^{i}F_{d+\frac{q^n-1}{q-1}-i}(x_1, \ldots, x_{n-1})$, then the above equation is equivalent to
\begin{equation*}
    (-1)^{n}x_{n+1}^{q^m}\cdot f + \sum\limits_{i=0}^{d+\frac{q^n-1}{q-1}}x_{n+1}^i\cdot  \delta_n\left(\dfrac{F_{d+\frac{q^n-1}{q-1}-i}}{L(x_1, \ldots, x_{n-1})}\right) = \dfrac{L(x_1, \ldots, x_{n+1})}{L(x_1, \ldots, x_n)}\delta_{n+1}(f).
\end{equation*}
Notice that $\dfrac{L(x_1, \ldots, x_{n+1})}{L(x_1, \ldots, x_n)} = \sum\limits_{i=0}^{n}(-1)^i Q_{n, n-i}x_{n+1}^{q^{n-i}}$, and by \eqref{eq:description_of_delta_f}, $\delta_{n+1}(f)$ is a polynomial of the variable $x_{n+1}$ with the coefficients being $h$-constructible polynomials of $n$ variables. Therefore, the right-hand side of the above equation is a polynomial of the variable $x_{n+1}$ with the coefficients being $(h+1)$-constructible polynomials of $x_1, \ldots, x_n$. Comparing with the coefficient of $x_{n+1}^{q^m}$ in the left-hand side, it follows that $f$ is $(h+1)$-constructible, as desired.
\end{proof}
This result shows a method of explicitly "detecting" products of $GL_n$-invariant polynomials and images of $\delta_n$. 
\begin{corollary}\label{cor:weak_version_m_less_than_n}
\cref{conj:polynomiality} is true whenever $m< n$.
\end{corollary}
\begin{proof}
    Suppose that $f$ is a homogeneous polynomial in $A_m^{(n)}$. Then for any $h\in \mathbb{N}$, if $\delta_{n+1}^h(f)\neq 0$, then $\delta_{n+1}^h(f)$ is of degree $\mathrm{deg}(f) + h\cdot(q^m - q^n)$, as $q^m - q^n \le -1$, it implies that $\delta_{n+1}^{\mathrm{deg}(f)}(f) = 0$, which means that $f$ is $\mathrm{deg}(f)$-vanishing. By \cref{lemma:general_exact}, $f$ lies in the $\mathcal{D}$-module generated by $\delta_n(A_{m}^{(n-1)})$.    
\end{proof}
Considering the case $m\geq n$, we show that
\begin{proposition}
    \label{prop:weakened_version}
For any $n\geq 2$ and $m\geq n$, $Q_{n, 0}^{\frac{q^m-q^n}{q-1}+1}A_{m}^{(n)}$ lies in the $\mathcal{D}_n$-module generated by 
\begin{equation*}
    \{1\}\cup \{\delta_{n;m}(f)\mid f\in A_{m}^{(n-1)}\}. 
\end{equation*}
\end{proposition}
\begin{proof}
    Suppose that $f\in A_{m}^{(n)}$ is a homogeneous polynomial of degree $d > 0$, and $g(x_1, \ldots, x_{n+1}) = \delta_{n+1}(f)$. Then we have
    \begin{equation*}
        g = \sum\limits_{i=0}^{d+q^m - q^n}x_{n+1}^{i}\cdot g_{d+q^m-q^n-i}(x_1, \ldots, x_n).
    \end{equation*}
    Similarly, let $F = f\cdot L(x_1, \ldots, x_n)$, suppose that 
    \begin{equation*}
        F = \sum\limits_{i=0}^{d+\frac{q^n-1}{q-1}}x_n^{i}F_{d+\frac{q^n-1}{q-1}-i}(x_1, \ldots, x_{n-1}).
    \end{equation*}
    Following the computation in \cref{lemma:general_exact}, we have
    \begin{equation}\label{eq:g_induct}
        g\cdot \left(\sum\limits_{i=0}^{n}(-1)^i Q_{n, n-i}x_{n+1}^{q^{n-i}}\right) =  (-1)^{n}x_{n+1}^{q^m}\cdot f(x_1, \ldots, x_n) + \sum\limits_{j=0}^{d+\frac{q^n-1}{q-1}}x_{n+1}^j\cdot  \delta_n\left(\dfrac{F_{d+\frac{q^n-1}{q-1}-j}}{L_{n-1}}\right).
    \end{equation}
    Comparing the coefficient of $x_{n+1}^{q^m}$ of \eqref{eq:g_induct}, we have
    \begin{equation*}
    \sum\limits_{i=0}^{n}(-1)^i \cdot Q_{n, n-i}\cdot g_{d+q^{n-i} - q^n} = (-1)^n f + \delta_n\left(\dfrac{F_{d+\frac{q^n-1}{q-1}-q^m}}{L_{n-1}}\right).    
    \end{equation*}
    For $j>q^m \geq q^n$, comparing the coefficient of $x_{n+1}^j$ of \eqref{eq:g_induct} gives
    \begin{equation}\label{eq:f_induct_explicit}
    \sum\limits_{i=0}^{n}(-1)^i \cdot Q_{n, n-i}\cdot g_{d+q^m - q^n - j+q^{n-i}} =  \delta_n\left(\dfrac{F_{d+\frac{q^n-1}{q-1}-j}}{L_{n-1}}\right).    
    \end{equation}
    Notice that the coefficient of $x_{n+1}^{d+q^m}$ of the left-hand side is $ g_0$, hence $g_0$ is vanishing. By induction, we see that $g_0, \ldots, g_{d-1}$ are vanishing polynomials. On the other hand, since $g = \delta_{n+1}(f)$, we have $g\in A_{m}^{(n+1)}$; in particular, $g\in \mathbb{F}_q[x_1^{q-1}, \ldots, x_{n+1}^{q-1}]$. Finally, using \eqref{eq:f_induct_explicit}, it suffices to prove by induction on $s\leq \dfrac{q^m-q^n}{q-1}$ that, 
    $Q_{n, 0}^{s+1}g_{d+q^m-q^n - s(q-1)}$ is vanishing. The case $s = 0$ follows from comparing coefficient of $x_{n+1}^{1}$ of \eqref{eq:g_induct}; and if $Q_{n, 0}^{s+1}g_{d+q^m-q^n - s(q-1)}$ is vanishing for all $s\leq s_0-1$ (for some $s_0\leq \frac{q^m-q^n}{q-1}$,) examine the coefficient of $x_{n+1}^{1+s_0(q-1)}$, notice that 
    \begin{equation*}
        1+s_0(q-1)\leq q^m - q^n + 1 < q^m,
    \end{equation*}
    it follows that the coefficient of $x_{n+1}^{1+s_0(q-1)}$ at the right-hand side of \eqref{eq:g_induct} is vanishing; considering the left-hand side, it follows that $Q_{n, 0}^{s_0+1}g_{d+q^m-q^n - s_0(q-1)}$ is vanishing. The proof is completed.
\end{proof}
\chapter{Conclusion}\label{chap:conclusion}
In previous chapters, we presented some topics of our interest in modular invariant theory, ranging from $q$-analogues of Schur functions, to invariant spaces of truncated polynomial rings. In general, substantial work remains to address the existing gaps in the current theory, as well as to simplify, generalize, and unify the underlying structures. Furthermore, we aim to establish connections between these structures and other areas of mathematics, such as algebraic topology and combinatorics. In this final chapter, we offer insights that may prove valuable and propose questions intended to inspire future research in this intriguing area of study.
\begin{enumerate}
    \item On Schur functions over finite fields.\smallbreak
    In the theory of Schur functions, the Littlewood-Richardson rule \cite{littlewood_richardson} provides a combinatorial formula representing the product of two Schur functions as a linear combination of Schur functions. One may ask for a $q$-analogue of the Littlewood-Richardson rule, the dual Pieri rule (a special case of the Littlewood-Richardson rule).
    As shown in \cref{prop:stanton_basis}, unlike the classical Schur functions, Schur functions over $\mathbb{F}_q$ do not form a basis for $\mathcal{D}_n$. Alternatively, one can show that
    \begin{proposition}\label{prop:proper_definition_of_schur}
        Let $\lambda = (\lambda_1\geq  \ldots\geq \lambda_n)$ be a partition. Then there exists a unique element $\overline{S}_{\lambda}$ in the basis $\mathcal{B}$ (see \cref{prop:stanton_basis}) having $x_1^{q^{n-1}(q-1)\lambda_1}x_2^{q^{n-2}(q-1)\lambda_2}\cdots x_n^{(q-1)\lambda_n}$ as its leading monomial; furthermore, any element of $\mathcal{B}$ is of the form $\overline{S}_{\lambda}$ for some partition $\lambda$. 
    \end{proposition}
    The polynomials $\overline{S}_{\lambda}$ should be more suitable for finding an analogue of the Littlewood-Richardson rule.
    \item On polynomiality of the $\delta$ operators.\smallbreak
    \begin{enumerate}
        \item We initially intended to approach \cref{conj:polynomiality} using results in \cite{ha_hai_nghia_GL} about the spanning property of certain elements in $Q(m, n)^{GL_n(\mathbb{F}_q)}$. In particular, using the canonical projection, a polynomial $f\in A_{m}^{(n)}$ becomes an element $[f]\in Q(m, n)^{GL_n(\mathbb{F}_q)}$, therefore, roughly speaking, we only have to analyze the part of $f$ that belongs to $I_{n, m}$; however, this seems to be intractable.
        \item The proof of \cref{thm:pol_bivariate} reveals an interesting feature of $A_m^{(n)}$, namely, after suitably "differentiating", we get a $\mathcal{D}_n$-module that does not depend on $m$. The first part in the proof of \cref{thm:pol_bivariate} can be extended to the general case; however, one needs to examine the exact form of the leading monomials (with respect to a monomial order) of the differentiated polynomials. 
        \item The proofs of \cref{cor:weak_version_m_less_than_n} and \cref{prop:weakened_version} follow from a different approach. The missing piece of this approach is the following question. Suppose that for some $f\in A_{m}^{(n-1)}$, $\delta_{n}(f)$ is divisible by $Q_{n,0}$. By \eqref{eq:fundamental_functional_eq}, $g=\dfrac{\delta_n(f)}{Q_{n,0}}\in A_{m}^{(n)}$. How does one express $g$ as a $\mathcal{D}_n$-linear combination of $G_m^{(n)}$?
    \end{enumerate}
    \item On analyzing the cofixed space.\smallbreak
    Let $G$ be a parabolic subgroup of $GL_n(\mathbb{F}_q)$. As stated earlier, we aim to find an $\mathbb{F}_q$-basis for $S_{G}$, or determine the $S^G$-module structure of $S_G$ as explicitly as possible.
    \begin{enumerate}
        \item A potential approach is to employ \cref{prop:dual_hilbert_series} to find a basis for $S_G$ containing elements in $Q(m, n)_G$ that are "stable" after taking $m\to +\infty$. The main issue of this approach is to calculate the dual basis with respect to the given basis of the invariant space. An instance where this approach proves effective is when $G=B$, the Borel subgroup. In this scenario, as established in \cite[Corollary 5.4]{Ha_Hai_Nghia_2024}, the leading monomials with respect to degrevlex of $\mathcal{B}_m(n)$ are pairwise distinct, from which the dual basis can be explicitly computed. However, in the most interesting case, $G = GL_n(\mathbb{F}_q)$, the proposed basis does not have this property.
        \item Another approach stemming from the Lewis-Reiner-Stanton conjecture is to inductively construct bases for $S_G$. In particular, let $F_G^{(n)}(t) = Hilb(S_G, t)$, from \cite[Equation (5.9)]{Lewis_Reiner_Stanton_2017}, we have the recursive formula
        \begin{equation*}
            F_{GL_n(\mathbb{F}_q)}^{(n)}(t) = F_{GL_{n-1}(\mathbb{F}_q)}^{(n-1)}(t) - t^{(n-1)(q-1)}F_{GL_{n-1}(\mathbb{F}_q)}^{(n-1)}(t^q) + \dfrac{t^{(n-1)(q^n-1)}}{\prod\limits_{i=0}^{n-1}(1-t^{q^n-q^i})}.  
        \end{equation*}
        This formula suggests that there should be an operator (corresponding to the first two terms of the right-hand side) that maps the cofixed space in $(n-1)$ variables to the cofixed space in $n$ variables.
        \item Finally, understanding the structure of $S$ as an $\mathbb{F}_q G$-module is another method to investigate the cofixed space $S_G$. In \cite{karagueguzian_symonds_finiteness, symonds_structure_theorem}, it has been proven that for $G=GL_n(\mathbb{F}_q)$, $S$ can be decomposed as
        \begin{equation}\label{eq:structure_theorem}
            S = \bigoplus\limits_{I\subseteq \{0, \ldots, n-1\}}\mathbb{F}_q[Q_{n, i}\mid i\in I]\otimes_{\mathbb{F}_q}X_{I},
        \end{equation}
        where each $X_{I}$ is a finite-dimensional $\mathbb{F}_q G$-submodule of $S$. To apply this result to the problem of determining $S_{GL_n(\mathbb{F}_q)}$, we aim to refine \eqref{eq:structure_theorem} by providing an explicit description  of the modules $X_I$ appearing in the decomposition.  
    \end{enumerate}
\end{enumerate}
\appendix
\chapter{Brauer Character Theory}\label{appendix:brauer_character}
In this appendix, we collect relevant results about the theory of modular characters that serve to clarify the assertions in the proof of \cref{cor:brauer_iso}. 
\section{Preliminary}\label{sec:1}
\subsection{The algebraic closure $\overline{\mathbb{Q}_p}$ of the $p$-adic field $\mathbb{Q}_p$ is isomorphic to $\mathbb{C}$ as fields}
We state (without proof) some results about discrete valuation fields that we shall use here. Basically, they help us work more conveniently in the field of complex numbers, instead of the algebraic closure of the field of $p$-adic numbers. \smallbreak
Let $p$ be a prime number. We have the field of $p$-adic numbers\index{$p$-adic numbers} $\mathbb{Q}_p$, which is the metric space completion of the rationals with respect to the $p$-adic norm $\left|\dfrac{m}{n}\right|_p = p^{v_p(n) - v_p(m)}$. Explicitly, an element of $\mathbb{Q}_p$ is a series $\sum\limits_{k = m}^{+\infty}a_k p^k$, where $m\in \mathbb{Z}$, and $a_k\in \{0, \ldots, p-1\}, \text{ for all } k\in \{0, \ldots, p-1\}$. We prove the following theorem about the algebraic closure of $\mathbb{Q}_p$.
\begin{theorem}[{\cite[Remark 9.17]{milne_fields_galois}}]\label{thm: Q_p and C are the same}
Let $\overline{\mathbb{Q}_p}$ be the algebraic closure of $\mathbb{Q}_p$. Then $\overline{\mathbb{Q}_p}$ is isomorphic to $\mathbb{C}$ as fields.
\end{theorem}
\begin{proof}
    It is not hard to see that $\mathbb{Q}_p$ has the same cardinality as $\mathbb{C}$, so any algebraic extension of $\mathbb{Q}_p$ (in particular, $\overline{\mathbb{Q}_p}$) has the same cardinality as $\mathbb{C}$, and these fields have uncountably many elements. Thus, if $S$ is a transcendence basis of $\mathbb{C}$ over $\mathbb{Q}$, $T$ is a transcendence basis of $\overline{\mathbb{Q}_p}$ over $\mathbb{Q}$, then $S$ and $T$ both have the same cardinality as $\mathbb{C}$. One readily sees that any bijection $S\to T$ extends to a field isomorphism $\mathbb{Q}(S) \to \mathbb{Q}(T)$. But $\overline{\mathbb{Q}_p}$ is the algebraic closure of $\mathbb{Q}(T)$, whilst $\mathbb{C}$ is the algebraic closure of $\mathbb{Q}(S)$, by uniqueness of the algebraic closure of a field, one obtains the desired result.   
\end{proof}
This result helps us work with the algebraic closure of $\mathbb{Q}_p$ more explicitly, just like working with the complex numbers. Another useful thing that we shall extensively use is the following result about finite separable extensions of $\mathbb{Q}_p$ (each of these finite separable extensions can be regarded as a subfield of $\mathbb{C}$, if needed), for which the proofs are omitted, as they are beyond the scope of this manuscript.
\begin{theorem}[{\cite[Chapter 2, Section 5]{neukirch_algebraic_2013}}]\label{thm: local_fields}
Let $K$ be a finite separable extension of $\mathbb{Q}_p$. There exists a metric $v: \mathbb{K}\to [0, +\infty)$ which restricts to $|\cdot |_p$ such that $v(xy) = v(x)v(y), \text{ for all } x, y\in K$, and $ v(x+y)\leq \max\{v(x), v(y)\}$, for all $x, y\in K$. Furthermore, the ring of integers $\mathcal{O}_K = \{x\in K\mid v(x)\leq 1\}$ is a local ring with a unique maximal ideal $\mathfrak{m} = \{x\in K\mid v(x)\leq p^{-1}\}$, and the residue field $k = \mathcal{O}_K/\mathfrak{m}$ is a finite field of characteristic $p$. 
\end{theorem}
\subsection{Composition series of a module}
In this section, we give the definition of composition series\index{composition series} of modules, which is a key notion in the theory of modular characters. We prove the Jordan-Holder theorem about uniqueness of composition multiplicity, which, in some sense, allows one to "compare" modules.
\begin{definition}\label{def: composition_series}
Let $R$ be a ring and $M$ is a nonzero left $R$-module. A composition series of $M$ is a filtration 
\begin{equation*}
    0 = M_0 \subsetneq M_1\subsetneq \cdots \subsetneq M_n = M,
\end{equation*}
where $n$ is a positive integers, $M_0, \ldots, M_n$ are submodules of $M$, and for each $i\in \{1, \ldots, n\}$, the quotient $M_{i+1}/M_i$ is a simple left $R$-module.
\end{definition}
For example, if $R = kG$, the group algebra over a field $k$ with respect to a group $G$, and $V$ is a nonzero left $R$-module, such that $V$ is a finite-dimensional $k$-vector space, then $V$ has a composition series. This comes from the fact that any nonzero $R$-module is a nonzero finite-dimensional $k$-vector space.\smallbreak
We state and prove the Jordan-Holder theorem\index{Jordan-Holder theorem}.
\begin{theorem}[{\cite[Theorem 3]{mahmoudi_jordan-holder_2012}}]\label{thm: Jordan-Holder}
Suppose $M$ is a left $R$-module that is both noetherian and artinian. Then, $M$ has a composition series. Furthermore, if
\begin{align*}
    0 &= M_0 \subsetneq M_1\subsetneq \cdots \subsetneq M_n = M,\\
    0 &= N_0 \subsetneq N_1\subsetneq \cdots \subsetneq N_m = M,
\end{align*}
are two composition series of $M$, then $n = m$, and two lists 
\begin{align*}
 [M_1/M_0, M_2/M_1, \ldots, M_n/M_{n-1}],\\   
[N_1/N_0, N_2/N_1, \ldots, N_m/N_{m-1}]
\end{align*}
are the same up to a permutation. \smallbreak
This list is called the composition factors of $M$.
\end{theorem}
\begin{proof}
    First, the existence of a composition series is guaranteed by seeing that, the artinianity of $M$ guarantees the existence of a minimal nonzero module $M_1\subseteq M$, which must be a simple module, and the noetherianity of $M$ guarantees that the process of taking the smallest module that properly contains the module constructed in the previous step must halt after finitely many iterations.\smallbreak
  We prove the result by induction on $k$, where $k$ is the length of a Jordan-Holder series of $M$ of minimum length. Without loss of generality, suppose that $k = n$, in particular we have $m \geq n$. If $n=1$ then $M$ is a simple module and the length of every other Jordan-Holder series of $M$ is also 1 and the only composition factor is $M$ and the result is proved.\smallbreak
Now suppose that $n>1$. Consider two submodules $M_{n-1}$ and $N_{m-1}$ and put $K=M_{n-1} \cap N_{m-1}$. There are two possibilities,
\begin{itemize}
    \item[(i)] $M_{n-1}=N_{m-1}$.
    \item[(ii)] $M_{n-1} \neq N_{m-1}$.
\end{itemize}
In the first case we have $K=M_{n-1}=N_{m-1}$. Consider two Jordan-Holder series
\begin{equation*}
0=M_{0} \subsetneq M_{1} \subsetneq \cdots \subsetneq M_{n-1}=K,
\end{equation*}
\begin{equation*}
0=N_{0} \subsetneq N_{1} \subsetneq \cdots \subsetneq N_{m-1}=K.
\end{equation*}
The above series shows that $K$ has a Jordan-Holder series of length at most $ n-1$, so the induction hypothesis implies that $n-1=m-1$ and the composition factors of above series are the same. Consequently the two original Jordan-Holder series have the same length and the same composition factors, explicitly, the composition factors of $M$ are those of $K$, and an $M/K$ added.\smallbreak
In the second case, we have $K \subsetneq M_{n-1}$ and $K \subsetneq N_{m-1}$. As $M_{n-1} \neq$ $N_{m-1}$ and $M_{n-1}$ and $N_{m-1}$ are maximal in $M$, we obtain $M_{n-1}+N_{m-1}=M$. Consequently, we have
\begin{equation*}
M_{n-1} / K=M_{n-1} /\left(M_{n-1} \cap N_{m-1}\right) \cong\left(M_{n-1}+N_{m-1}\right) / N_{m-1}=M / N_{m-1}.
\end{equation*}
So
\begin{equation*}
M_{n-1} / K \cong M / N_{m-1},
\end{equation*}
similarly, we have
\begin{equation*}
N_{m-1} / K \cong M / M_{n-1}.
\end{equation*}
In particular, two quotient modules $M_{n-1} / K$ and $N_{m-1} / K$ are simple modules. As $M$ is both artinian and notherian, $K$ is as well. In particular $K$ has a Jordan-Holder series as follows.
\begin{equation*}
0=K_{0} \subsetneq K_{1} \subsetneq \cdots \subsetneq K_{r}=K.
\end{equation*}
We therefore obtain two new Jordan-Holder series for $M$, which are
\begin{equation*}
\begin{aligned}
& 0=K_{0} \subsetneq K_{1} \subsetneq \cdots \subsetneq K_{r}=K \subsetneq M_{n-1} \subsetneq M_{n}=M, \\
& 0=K_{0} \subsetneq K_{1} \subsetneq \cdots \subsetneq K_{r}=K \subsetneq N_{m-1} \subsetneq N_{m}=M.
\end{aligned}
\end{equation*}
It is trivial that $M_{n-1}$ has a Jordan-Holder series of length at most $ n-1$, so we can apply the induction hypothesis for $M_{n-1}$, hence, all Jordan-Holder series of $M_{n-1}$ are of the same length. The above composition series show that $M_{n-1}$ has a Jordan-Holder series of length $r+1$ and the original composition series shows that $M_{n-1}$ has a Jordan-Holder series of length $n-1$, so we have $r+1=n-1$ and two Jordan-Holder series
\begin{equation*}
0=K_{0} \subsetneq K_{1} \subsetneq \cdots \subsetneq K_{r}=K \subsetneq M_{n-1}
\end{equation*}
and
\begin{equation*}
0=M_{0} \subsetneq M_{1} \subsetneq \cdots \subsetneq M_{n-1}
\end{equation*}
have the same composition factors. Hence the length and the composition factors of 
\begin{equation*}
     0=K_{0} \subsetneq K_{1} \subsetneq \cdots \subsetneq K_{r}=K \subsetneq M_{n-1} \subsetneq M_{n}=M
\end{equation*}
and
\begin{equation*}
     0 = M_0 \subsetneq M_1\subsetneq \cdots \subsetneq M_n = M
\end{equation*}
are the same. Similarly, $N_{m-1}$ has a series of length $r+1=n-1$. By induction, the length and the composition factors of the below Jordan-Holder series of $N_{m-1}$ are the same,
\begin{equation*}
0=K_{0} \subsetneq K_{1} \subsetneq \cdots \subsetneq K_{r}=K \subsetneq N_{m-1}
\end{equation*}
and
\begin{equation*}
0=N_{0} \subsetneq N_{1} \subsetneq \cdots \subsetneq N_{m-1}.
\end{equation*}
Consequently, the length and the composition factors of two series 
\begin{equation*}
     0=K_{0} \subsetneq K_{1} \subsetneq \cdots \subsetneq K_{r}=K \subsetneq N_{m-1} \subsetneq N_{m}=M
\end{equation*}
and
\begin{equation*}
     0 = N_0 \subsetneq N_1\subsetneq \cdots \subsetneq N_m = M
\end{equation*}
are the same. The proven results conclude the proof, because $[N_m/N_{m-1}, N_{m-1}/K]$ and $[M_n/M_{n-1}, M_{n-1}/K]$ are the same list.
\end{proof}
The Jordan-Holder theorem simply asserts that, if $M$ is a left $R$-module that is both artinian and noetherian, and for each simple module $S$, the number of occurrences of $S$ in the composition factors of $M$ is a well-defined number, we call it the composition multiplicity of $S$ with respect to $M$.
\subsection{Semisimple representations are uniquely determined by characteristic polynomials}
Let $R$ be a ring, and $M$ be a nonzero $R$-module that is both noetherian and artinian. Then, \cref{thm: Jordan-Holder} shows that $M$ admits a composition series 
\begin{equation*}
    0 = M_0 \subsetneq M_1\subsetneq \cdots\subsetneq M_n = M.
\end{equation*}
Furthermore, it is not hard to see that the \emph{semisimplification}\index{semisimplification} $M_{s}=\bigoplus_{i=1}^{n}M_i/M_{i-1}$ is uniquely determined by $M$, and the composition factors of $M$ and $M_s$ are the same. Therefore, if one only cares about comparing composition factors of $R$-modules, in certain cases, one can assume without loss of generality that the modules in consideration are semisimple, i.e., they are direct sums of simple $R$-modules.\smallbreak
We prove that given any field $k$ and a finite group $G$, two semisimple $kG$-modules $M$ and $N$ that are finite-dimensional over $k$, if for any $g\in G$, the characteristic polynomials of $g_M: M\to M$ and $g_N: N\to N$ are equal, then $M$ and $N$ are isomorphic. In order to successfully prove this, we need some technical lemmas.
\begin{lemma}\label{lemma: 5.8}
Suppose that $a\in kG$, and $a_M: M\to M$ has eigenvalue $\lambda\in k$. Then the characteristic polynomial $\chi_{M}(a)$, viewed as an element of $k[T]$, is divisible by $(1-\lambda T)$. If $\lambda \neq 0$, one has $\chi_{M}(a) \neq 1$.
\end{lemma}
\begin{proof}
    If $a_{M}$ has eigenvalue $\lambda \in k$, there is some $k$-basis $\left\{m_{1}, \ldots, m_{n}\right\}$ of $M$ satisfying $a_{M} m_{1}=\lambda m_{1}$. We easily see that for $M^{\prime}=k m_{1}$ and $M^{\prime \prime}=M / M^{\prime}$, we have $\chi_{M}(a)=\chi_{M^{\prime}}(a) \cdot \chi_{M^{\prime \prime}}(a)$. As $\chi_{M^{\prime}}(a)=(1-\lambda T)$, one has $\chi_{M}(a)$ is divisible by $(1-\lambda T)$. If $\lambda \neq 0$, one has $(1-\lambda T)$ is not invertible in $k[T]$, hence it follows that $\chi_{M}(a) \neq 1$.
\end{proof}
\begin{lemma}[{\cite[Lemma 5.9]{eggermont_generalizations_2011}}]\label{lemma: 5.9}
Let $R$ be a semisimple ring and let $S$ be a simple $R$-module. Then for each $s \in S \backslash\{0\}$ there is some $r \in R$ with $r s=s$ and such that for any simple $R$-module $T$ that is not isomorphic to $S$, one has $r T=0$.
\end{lemma}
\begin{proof}
    Let $s \in S \backslash\{0\}$. Consider the $R$-linear map $f: R \rightarrow S$ defined by $f(r)=r s$ for each $r \in R$. As $S$ is simple and $f(1)=s \neq 0$, one has $f(R)=S$. Consider the exact sequence $0 \rightarrow \operatorname{Ker}(f) \rightarrow R \xrightarrow{f} S \rightarrow 0$. Since $R$ is semisimple, this sequence splits, hence there is an $R$-linear map $\phi: S \rightarrow R$ such that $f \circ \phi=\operatorname{Id}_{S}$. Define $r=\phi(s)$. Then one has 
    \begin{equation*}
        r s=f(r)=f(\phi(s))=(f \circ \phi)(s)=s.
    \end{equation*}
Let $T$ be a simple $R$-module and suppose it is not isomorphic to $S$. Let $t \in T$. Consider the $R$-linear map $g: S \rightarrow T$ defined by $s^{\prime} \mapsto \phi\left(s^{\prime}\right) t$ for $s^{\prime} \in S$. If $g$ is injective, it is an isomorphism since $T$ is simple and $S$ is not $0$. This is false by assumption, hence $\operatorname{Ker}(g) \neq 0$. Since $S$ is simple, it follows $\operatorname{Ker}(g)=S$. In particular, one has $r t=\phi(s) t=g(s)=0$. Hence one has $r T=0$.
\end{proof}
Recall that a ring is semisimple\index{semisimple!ring} if any finitely generated module over that ring is semisimple, or equivalently, if the ring is a semisimple module over itself (proof: any finitely generated module is a quotient of a finite direct sum of the ring, thus it suffices to prove that any finite direct sum of the ring is semisimple, which is trivial.) There are two equivalent definition of semisimple modules, the first is that $M$ is a semisimple module if every epimorphism from $M$ splits, while the second is that $M$ is a direct sum of simple modules. A standard property of semisimple modules\index{semisimple!module} is that, submodules of semisimple modules are semisimple, in fact,  submodules are always direct summands of a semisimple module.
\begin{lemma}[{\cite[Lemma 2.9]{eggermont_generalizations_2011}}]\label{lemma: 2.9}
If $k$ is a field and $A$ is a $k$-algebra that is finite-dimensional over $k$, and the Jacobson radical $J(A) = (0)$, then $A$ is a semisimple ring.\smallbreak
The Jacobson radical\index{Jacobson radical} $J(A)$ is defined to be the intersection of all maximal left ideals of $A$.
\end{lemma}
\begin{proof}
    Since the dimension over $k$ of any finite intersection of maximal left ideal cannot exceed $\dim(A)$, therefore, there exists a finite intersection $M = \mathfrak{m}_1\cap \ldots\cap \mathfrak{m}_l$ of maximal left ideals of minimal dimension $d$. If $d>0$, there exists some nonzero $x\in M$, and because $J(A) = 0$, there is some maximal left ideal $\mathfrak{m}_0$ that does not contain $x$, then $M\cap \mathfrak{m}_0$ is properly contained in $M$, a contradiction. Therefore, $d=0$, hence, $M = J(A)$. Now, we have an injective $R$-map
    \begin{equation*}
        R = R/M \xhookrightarrow{} \bigoplus\limits_{i=1}^{l}R/\mathfrak{m}_i
    \end{equation*}
    that is induced by the projections $R\to R/\mathfrak{m}_i$. This means that $R$ is a submodule of $\bigoplus\limits_{i=1}^{l}R/\mathfrak{m}_i$, and each $R/\mathfrak{m}_i$ is a simple module, hence, $\bigoplus\limits_{i=1}^{l}R/\mathfrak{m}_i$ is semisimple, which implies that $R$ is semisimple.
\end{proof}
\begin{lemma}[{\cite[Lemma 2.11]{eggermont_generalizations_2011}}]\label{lemma: 2.11}
If $k$ is a field and $A$ is a $k$-algebra, let $M$ be a semisimple $A$-module that is finite-dimensional over $k$. Then $A/\mathrm{Ann}(M)$ is finite-dimensional over $k$, and is a semisimple ring.
\end{lemma}
\begin{proof}
    Since $M$ is finite-dimensional over $k$, $\mathrm{End}_k(M)$ is finite-dimensional, and because one has a monomorphism $A/\mathrm{Ann}(M) \to \mathrm{End}_k(M)$, the algebra $A/\mathrm{Ann}(M)$ is also finite-dimensional.\smallbreak
    Suppose that $M = \bigoplus\limits_{i=1}^l S_i$, where each $S_i$ is a simple $A$-module. It follows that $\mathrm{Ann}(M) = \bigcap\limits_{i=1}^l \mathrm{Ann}(S_i)$ is an intersection of maximal left ideals of $A$, in particular, $J(A)\subseteq \mathrm{Ann}(M)$. Therefore, $J(A/\mathrm{Ann}(M)) = (0)$, and $J(A/\mathrm{Ann}(M)) = \bigcap\limits_{i=1}^l \mathrm{Ann}(S_i)/\mathrm{Ann}(M)$ is an intersection of finitely many maximal ideals, hence by \cref{lemma: 2.9}, $A/\mathrm{Ann}(M)$ is a semisimple ring. 
\end{proof}
\begin{theorem}[{\cite[Theorem 5.10]{eggermont_generalizations_2011}}]\label{thm: 5.10}
Suppose that $A$ is a $k$-algebra. Let $M, N$ be semisimple $A$-modules that are finite-dimensional over $k$. Suppose that for all $a \in A$, one has $\chi_{M}(a)=\chi_{N}(a)$. Then, $M$ and $N$ are isomorphic as $kG$-modules.
\end{theorem}
Before proving this theorem, we present the corollary that we shall use, that is, if $A = kG$, the group algebra with respect to a group $G$ over a field $k$, suppose that the characteristic polynomial of $g_N$ and $g_M$ are equal for any $g\in G$, then because $G$ generates $kG$, it follows that the assumptions of \cref{thm: 5.10} are fulfilled. Therefore, $M$ and $N$ are isomorphic as semisimple representations of $G$ over $k$.
\begin{proof}
    First, we show that we may assume that $A$ is finite-dimensional over $k$ and semisimple. Let $I \subset A$ be the annihilator of $M \bigoplus N$ (obviously, $M\bigoplus N$ is semisimple.) Then both $M$ and $N$ are $A / I$-modules and $A / I$ is semisimple and finite-dimensional over $k$ by \cref{lemma: 2.11}. As $(a+I)_{M}=a_{M}$ and $(a+I)_{N}=a_{N}$ for any $a \in A$, we have $\chi_{M}(a+I)=$ $\chi_{N}(a+I)$ for all $a \in A$. Note that we have a canonical bijection $\operatorname{Hom}_{A}(M, N) \cong\operatorname{Hom}_{A / I}(M, N)$, since any $A$-linear map from $M$ to $N$ is also $A / I$-linear and vice versa. Hence we have $M \cong_{A} N$ if and only if $M \cong_{A / I} N$. Moreover, $M$ and $N$ are still semisimple as $A / I$-modules (because $A/I$ is semisimple.) So if the theorem holds with $A$ replaced by $A / I$, it holds for $A$ as well.\smallbreak
    Assume that $A$ is semisimple and finite-dimensional over $k$. Write $M=\bigoplus\limits_{i=1}^{d} S_{i}$ and $N=\bigoplus\limits_{i=1}^{e} T_{i}$ with $d, e \in \mathbb{Z}_{\geq 0}$, and $S_{1}, \ldots, S_{d}$, $T_{1}, \ldots, T_{e}$ are simple $A$-modules. Assume $d \geq e$ without loss of generality. We apply induction to $d$.\smallbreak
If $d=0$, both $M$ and $N$ are the zero module, and hence one has $M=N$.
Suppose $d>0$ and that the result is true for all $d^{\prime}<d$. View a component $S_{1}$ of $M$ as a subset of $M$. Let $s \in S_{1} \backslash\{0\}$. Let $a \in A$ such that as $=1$ and $a T=0$ for each simple $A$-module $T$ that is not isomorphic to $S_{1}$. Such $a$ exists by \cref{lemma: 5.9}.\smallbreak
Suppose $S_{1}$ is not isomorphic to $T_{i}$ for any $i \in\{1,2, \ldots, e\}$. Then one has $a T_{i}=0$ for each $i \in\{1,2, \ldots, e\}$. In particular, it follows that $a N=0$, yielding that $\chi_{N}(a)$ is the constant polynomial $1$.\smallbreak
On the other hand, viewing $S_{1}$ as a nonzero $A$-submodule of $M$, we see that $a_{M}$ has eigenvalue $1$, since we have $a s=s$. Thus $\chi_{M}(a)$ is not the constant polynomial $1$ by \cref{lemma: 5.8}, contradicting $\chi_{M}(a)=\chi_{N}(a)$. So at least one of the $T_{i}$ is isomorphic to $S_{1}$.\smallbreak
Assume without loss of generality that $T_{1}$ and $S_{1}$ are isomorphic. Now consider the modules $M / S_{1}$ and $N / T_{1}$; these satisfy $\chi_{M / S_{1}}(a)=\chi_{M}(a) / \chi_{S_{1}}(a)=$ $\chi_{N}(a) / \chi_{T_{1}}(a)=\chi_{N / T_{1}}(a)$ for all $a \in A$, using $\chi_{S_{1}}(a)=\chi_{T_{1}}(a)$ since $S_{1} \cong T_{1}$. As both $M / S_{1}$ and $N / T_{1}$ have precisely one fewer simple submodule in their decomposition, we can apply the induction hypothesis to $M / S_{1}$ and $N / T_{1}$ and conclude $M / S_{1} \cong N / T_{1}$. As $M \cong M / S_{1} \bigoplus S_{1}$ and $N \cong N / T_{1} \bigoplus T_{1}$ since $M$ and $N$ are semisimple, it follows that $M$ and $N$ are isomorphic.
\end{proof}
\section{Modular characters}\label{sec:2}
The following definitions are taken from \cite[Chapter 18]{serre1977representations}. Let $p$ be a prime number, and $G$ be a finite group. Denote by $G_{reg}$ the set of $p$-regular elements\index{$p$-regular element} of $G$, that is, the set of elements $g\in G$ such that $\gcd(|g|, p) = 1$, let $m'$ be the least common divisor of the orders of elements in $G_{reg}$. Suppose that $\mathbb{K}$ be a finite extension of $\mathbb{Q}_p$ that contains every $m'$th root of unity, let $\mathfrak{m}$ be the maximal ideal of the ring of integers of $\mathbb{K}$. Because $\gcd(m', p) = 1$, let $\mu_{\mathbb{K}}$ be the group of $m'$th roots of unity in $\mathbb{K}$, and $\mu_{k}$ be the group of $m'$th roots of unity in $k=\mathcal{O}_{\mathbb{K}}/\mathfrak{m}$, then both $\mu_{\mathbb{K}}$ and $\mu_k$ have exactly $m'$ elements (as the polynomial $x^{m'}-1$ is separable in $k$), which means that the map 
\begin{align*}
    \pi_{\mathfrak{m}}&:  \mu_{\mathbb{K}}\to \mu_k\\
    x &\mapsto [x] 
\end{align*}
is a group isomorphism (notice that a root of unity must have norm $1$, which implies that it belongs to the ring of integers.) For each $\lambda\in \mu_k$, let $\Tilde{\lambda} = \pi_{\mathfrak{m}}^{-1}(\lambda)$. \smallbreak
Let $n$ be a positive integer, and $E$ is a left $kG$-module, that is also an $n$-dimensional $k$-vector space. Suppose that $s\in G_{reg}$, $s_E: E\to E$ is the endomorphism of $E$ induced by $s$. If $s_E$ is of order $a$, then by definition, $a\mid m'$, thus there are $a$ distinct roots of unity $\epsilon_1, \ldots, \epsilon_a$ in $k$. Therefore, we have
\begin{equation*}
    a\cdot s_E^{-1} = a\cdot s_E^{a-1} = (s_E^a - I)' = \left(\prod\limits_{j = 1}^{a}(s_E - \epsilon_j I) \right)' = \sum\limits_{j=1}^{a}\prod\limits_{i = 1, i\neq j}^{a}(s_E - \epsilon_i I).
\end{equation*}
This implies that, for any $v\in E$, we have
\begin{equation*}
    v = a s_E^{-1} (a^{-1} s_E(v)) = \sum\limits_{j=1}^{a}\prod\limits_{i = 1, i\neq j}^{a}(s_E - \epsilon_i I)(a^{-1}s_E(v)).
\end{equation*}
For each $j$, the summand $\prod\limits_{i = 1, i\neq j}^{a}(s_E - \epsilon_i I)(a^{-1}s_E(v))$ lies in the kernel of $A-\epsilon_j I$, because $\prod\limits_{i = 1}^{a}(s_E - \epsilon_i I) = 0$, hence, the above equation shows that $E$ can be decomposed into a direct sum of eigenspaces of $s_E$, which means that $s_E$ is diagonalizable.\smallbreak
Let $\lambda_1, \ldots, \lambda_n \in k$ be the eigenvalues of $s_E$, put
\begin{equation*}
    \phi_E(s) = \sum\limits_{i=1}^{n}\Tilde{\lambda_i}.
\end{equation*}
The function $\phi_E: G_{reg}\to \mathcal{O}_{\mathbb{K}}$ just defined is called the \emph{modular character}\index{modular character} of $E$. We have the following simple properties of modular characters.
\begin{proposition}[{\cite[Subsection 18.1]{serre1977representations}}]\label{prop: properties of Brauer_char}
With the above notations, the following claims are true.
\begin{enumerate}
    \item $\phi_E(1) = n$.
    \item $\phi_E(tst^{-1}) = \phi_E(s), \text{ for all } t\in G, s\in S_{reg}$, i.e., $\phi_E$ is a class function on $G_{reg}$.
    \item If $0\xrightarrow{} E' \xrightarrow{} E \xrightarrow{} E" \xrightarrow{} 0$ is an exact sequence of left $kG$-modules, then $\phi_E = \phi_{E'}+\phi_{E"}$.
\end{enumerate}
\end{proposition}
\begin{proof}
    The first property is trivial. For the second one, just notice that two similar matrices have the same characteristic polynomial. For the last property, simply notice that, for any $s\in G_{reg}$, the list of eigenvalues of $s_E: E\to E$ is the union of the list of eigenvalues of $s_{E'}: E'\to E'$ and the list of eigenvalues of $s_{E"}: E"\to E"$, because $E" = E/E'$, and $E'$ is an invariant subspace of $E$.    
\end{proof}
In particular, \cref{prop: properties of Brauer_char} (3) implies that if $E_0 \subseteq E_1\subseteq \cdots \subseteq E_m = E$ is a filtration of $E$, then $\phi_E = \phi_{E_0} + \phi_{E_1/E_0} + \cdots + \phi_{E_m/E_{m-1}}$. Furthermore, one sees that if $E$ and $E'$ are two $kG$ modules which have the same list of composition factors, then $\phi_E = \phi_{E'}$ (take a composition series of $E$, by the above summation formula and commutativity of addition, $\phi_E$ only depends on the composition factors of $E$). We now prove that the converse also holds, i.e., if $E$ and $E'$ are two $kG$-modules having the same modular character, then they are \emph{Brauer isomorphic}\index{Brauer isomorphic}, i.e., they have the same list of composition factors. Additionally, if $E$ and $E'$ are semisimple (e.g., if $G = G_{reg}$), then $E\cong E'$.
\begin{theorem}[{\cite[Subsection 18.2, Corollary 1]{serre1977representations}}]\label{thm:main_thm_brauer}
Let $E_1$ and $E_2$ be left $kG$-modules such that $\phi_{E_1} = \phi_{E_2}$. Then they have the same composition factors.
\end{theorem}
\begin{proof}
    Let $g\in G_{reg}$, then for any $k\geq 0$, we have $\phi_{E_1}(g^k) = \phi_{E_2}(g^k)$. If $[\lambda_1, \ldots, \lambda_n]$ are eigenvalues of $g_{E_1}$, $[\mu_1, \ldots, \mu_m]$ are eigenvalues of $g_{E_2}$, then the above identities simplify to 
    \begin{equation*}
        \sum\limits_{i = 1}^{n}\Tilde{\lambda_i}^{k} = \sum\limits_{j = 1}^{m}\Tilde{\mu}_j^{k}, \text{ for all } k\geq 0. 
    \end{equation*}
    Because all the eigenvalues are nonzero elements in the field $\mathbb{K}$ of characteristic zero, inserting $k=0$ gives $m=n$, then invoking the Newton's identities and Viete's theorem, we see that the two lists $[\lambda_1, \ldots, \lambda_n]$ and $[\mu_1, \ldots, \mu_m]$ must be the same up to a permutation of indices. In particular, the characteristic polynomials of $g_{E_1}$ and $g_{E_2}$ are the same for any $g\in G_{reg}$.\smallbreak
    Next, for any $g\in G$, suppose that $|g| = p^m\cdot m'$, where $m\ge 0$, and $\gcd(m', p) = 1$. Then there exist $x, y\in \mathbb{Z}$ such that $xp^m + ym' =1$. Then,
    \begin{equation*}
        g = g^{xp^m + ym'} = g^{xp^m}g^{ym'}, 
    \end{equation*}
    we see that $a= g^{xp^m}\in G_{reg}$, and the order of $h = g^{ym'}$ is a power of $p$. If $|h| = p^l$, then because $h_{E_1}$ and $a_{E_1}$ commute, we have
    \begin{equation*}
        \det(g_{E_1}-\lambda I)^{p^l} = \det(h_{E_1}^{p^l} a^{p^{l}}_{E_1} - \lambda^{p^l} I) = \det(a_{E_1}-\lambda I)^{p^l}.
    \end{equation*}
    Similarly, $\det(g_{E_1}-\lambda I) = \det(a_{E_1}-\lambda I) = \det(a_{E_2}-\lambda I) = \det(g_{E_2}-\lambda I)$, i.e., the characteristic polynomial of $g_{E_1}$ and $g_{E_2}$ are equal, for any $g\in G$.
    \smallbreak Using the remark following the statement of \cref{thm: 5.10}, it follows that the semisimplification of $E_1$ and $E_2$ are isomorphic (since the characteristic polynomial only depends on the composition factors,) thus $E_1$ and $E_2$ have the same composition factors. 
\end{proof}
Recall the unproven claim in \cref{cor:brauer_iso} that two $\mathbb{K}G$-modules are Brauer isomorphic if and only if they are Brauer isomorphic after extending scalars. Using \cref{thm:main_thm_brauer}, this statement becomes trivial, since extending scalars does not change the modular character map. 
\bibliographystyle{plain} 
\bibliography{citations} 
\addcontentsline{toc}{chapter}{{\bf{Bibliography}}\rm }
\printindex
\end{document}